\documentclass[a4paper,11pt]{article} 
\usepackage{
  amsmath,
  amsthm,
  amsfonts,
  amssymb,
  enumerate,
  xfrac,
  tikz,
  comment,
  pgfplots,
  xcolor
}
 \pgfplotsset{compat=newest,compat/show suggested version=false}
\usepackage[colorlinks,citecolor=blue,urlcolor=blue]{hyperref}
\usetikzlibrary{shapes.geometric,positioning,matrix,intersections,calc,patterns,arrows.meta}
\excludecomment{codes}
\usepackage[text={6.5in,9.0in},centering]{geometry}

\usepackage{graphicx,url,subfig}
\usepackage{authblk}
\RequirePackage[OT1]{fontenc}
\usepackage{datetime}
\usepackage[normalem]{ulem} 
\usepackage{soul} 
\numberwithin{equation}{section}
\numberwithin{figure}{section}

\usepackage{mathtools}
\usepackage{amssymb}
\usepackage{pifont}
\usepackage{comment}
\usepackage{lmodern}

\theoremstyle{plain}
\newtheorem{theorem}{Theorem}[section]

\newtheorem{lemma}[theorem]{Lemma}
\newtheorem{proposition}[theorem]{Proposition}

\theoremstyle{definition}
\newtheorem{definition}[theorem]{Definition}
\newtheorem{hypothesis}[theorem]{Hypothesis}

\newtheorem{remark}[theorem]{Remark}

\newcommand{\E}{\mathbb{E}}

\newcommand{\ud}{\ensuremath{\mathrm{d} }}

\newcommand{\Norm}[1]{\left\| #1 \right\|}

\DeclareMathOperator{\sign}{\ensuremath{sign}}

\newcommand{\calG}{\mathcal{G}}
\newcommand{\calJ}{\mathcal{J}}

\newcommand{\calI}{\mathcal{I}}

\newcommand*{\one}{{ {\rm 1\mkern-1.5mu}\!{\rm I} }}

\newcommand{\R}{\mathbb{R}}
\renewcommand{\P}{\mathbb{P}}
\newcommand{\myEnd}{\hfill$\square$}

\newcounter{Pq}
\newcounter{Le}
\usepackage[framemethod=tikz]{mdframed}
\mdfsetup{%
	middlelinecolor=lightgray,
	middlelinewidth=2pt,
	backgroundcolor=red!20,
	roundcorner=10pt}
{\refstepcounter{Pq}
	\begin{mdframed}[backgroundcolor=lightgray]\textbf{Qanqiu's comment \thePq \: #1}\\[+1.5em]}%
	{\end{mdframed}}
{\refstepcounter{Le}
  \begin{mdframed}[backgroundcolor=red!10]\textbf{Le's comment \theLe \: #1}\\[0.5em]}%
	{\end{mdframed}}
\numberwithin{Pq}{section}
\numberwithin{Le}{section}
\usepackage[strict=true,style=english]{csquotes}
\usepackage[backend=biber,
            style=alphabetic,
            natbib=true,
            maxbibnames=99,
            sorting=nyt,
            abbreviate=true]{biblatex}
\AtEveryBibitem{
 \clearfield{doi}
 \clearfield{url}
 \clearfield{issn}
 \clearlist{location}
 \clearfield{month}
 \clearfield{series}

 \ifentrytype{book}{}{
	\clearlist{publiE:SHEr}
  \clearname{editor} }
}

\title{
  Asymptotic properties of stochastic partial differential equations\\
  in the sublinear regime
}

\author[]{ Le Chen\footnote{Research is partially
supported by \textit{Mathematics and Physical Sciences-Collaboration Grants for
Mathematicians} from \textit{Simons Foundation} (Award Number: 959981). Email:
  \href{mailto:le.chen@auburn.edu}{le.chen@auburn.edu}, } }

\author[]{
  Panqiu Xia\footnote{Email: \href{mailto:pqxia@auburn.edu}{pqxia@auburn.edu}}
}

\affil{
  Department of Mathematics and Statistics\\
  Auburn University
}
\date{\today}
\begin{document}
\maketitle
\setlength{\parindent}{1.5em}

\begin{abstract}
  In this paper, we investigate stochastic heat equation with sublinear
  diffusion coefficients. By assuming certain concavity of the diffusion
  coefficient, we establish non-trivial moment upper bounds and almost sure
  spatial asymptotic properties for the solutions. These results shed light on
  the \textit{smoothing intermittency effect} under \textit{weak diffusion}
  (i.e., sublinear growth) previously observed by Zeldovich \textit{et
  al.}~\cite{zel-dovich.molchanov.ea:87:self-excitation}. The sample-path
  spatial asymptotics obtained in this paper partially bridge a gap in earlier
  works of Conus \textit{et al.}~\cite{conus.joseph.ea:13:on,
  conus.joseph.ea:13:on*1}, which focused on two extreme scenarios: a linear
  diffusion coefficient and a bounded diffusion coefficient. Our approach is
  highly robust and applicable to a variety of stochastic partial differential
  equations, including the one-dimensional stochastic wave equation and the
  stochastic fractional diffusion equations.

  \vspace{1em}


\end{abstract}

{
  \hypersetup{linkcolor=black}
  \tableofcontents
}

\section{Introduction}
In this paper, we investigate the \textit{stochastic partial differential
equations} (SPDEs) in the sublinear regime. More precisely, we focus on the
asymptotic behavior of the following \textit{stochastic heat equation} (SHE)
\begin{align}\label{E:SHE}
  \begin{dcases}
    \frac{\partial}{\partial t} u(t,x) = \frac{1}{2} \Delta u(t,x) + \rho(u(t,x)) \dot{W} (t,x), & t>0, \: x\in\R^d, \\
    u(0,\cdot ) = \mu,
  \end{dcases}
\end{align}
where the diffusion coefficient $\rho(\cdot)$ is assumed to be locally bounded
and exhibits sublinear growth at infinity. Previous studies have extensively
examined the case when the diffusion coefficient $\rho(\cdot)$ has linear growth
at infinity, which results in an intermittent solution. In particular, a
solution is said to be \textit{intermittent} if the \textit{moment Lyapunov
exponents} $\overline{\lambda}_p$ and $\underline{\lambda}_p$ of the solution,
defined by
\begin{align*}
  \overline{\lambda}_p  \coloneqq \frac{1}{p} \limsup_{t \to \infty} \log \E\left[ \left| u(t,x) \right|^p \right] \quad \text{and} \quad
  \underline{\lambda}_p \coloneqq \frac{1}{p} \liminf_{t \to \infty} \log \E\left[ \left| u(t,x) \right|^p \right] \quad (p\ge 2),
\end{align*}
satisfy the property that $\underline{\lambda}_2>0$. The literature on this
topic is extensive, and interested readers may
consult~\cite{carmona.molchanov:94:parabolic,
foondun.khoshnevisan:09:intermittence, chen.dalang:15:moments, chen:15:precise,
khoshnevisan.kim.ea:17:intermittency} and references therein. Zeldovich
\textit{et al.}~\cite[Chapter 9]{zel-dovich.ruzmauikin.ea:90:almighty} have
observed that intermittency is a universal phenomenon that occurs irrespective
of the underlying properties of the instability in a random medium, as long as
the random field is of multiplicative type. However, they have also noted
in~\cite{zel-dovich.molchanov.ea:87:self-excitation} (see also ~\cite[Section
8.9]{zel-dovich.ruzmauikin.ea:90:almighty}) that ``\textit{smoothing
intermittency}'', where the high maxima of the solution have smaller growth,
should be expected in the presence of ``\textit{weak diffusion}''---when
$\rho(\cdot)$ exhibits sublinear growth at infinity. The authors substantiated
their statement by highlighting the power growth of the moments when
$\rho(\cdot)$ is bounded, or more specifically, when
\begin{align}\label{E:Ex-BoundedRho}
  \rho(u) = \frac{u}{1+u}\,,
\end{align}
and when the noise is white in time and space-independent. Inspired by their
previous work, the objective of this paper is to conduct a comprehensive
analysis of the smoothing intermittency property for some more general weak
diffusion cases and demonstrate how the moment growth rate of the solution is
influenced by that of $\rho(\cdot)$.

The examination of SPDEs in the sublinear regime is also motivated by the
requirement for more realistic biological population models. As highlighted by
K\"onig in the Appendix of~\cite{konig:16:parabolic}, the \textit{parabolic
Anderson model} (PAM) (i.e., $\rho(u) = \lambda u$) utilized in population
dynamics, leads to excessive branching and killing rates that are not reflective
of real-world scenarios as it lacks any form of birth or death control. In
contrast, introducing a sublinear growth for $\rho$ provides an avenue for
developing models that may address this issue and provide a more accurate
depiction of the dynamics of biological populations.

Both the PAM and the SHE with additive noise, or simply the \textit{additive
SHE}, (i.e., $\rho\equiv 1$, or more generally, the case when $\rho$ is bounded)
have been extensively studied in the literature. They represent two extreme
cases where rich properties have been previously derived. For instance, when
$d=1$, the noise is space-time white noise, and the initial condition $\mu$ is
constant, Conus \textit{et al.}~\cite{conus.joseph.ea:13:on} showed that
\begin{align}\label{E:Conus-SpAsym}
  \begin{dcases}\medskip
    \phantom{\log}\sup_{|x|\le R} u(t,x) \asymp \left[\log R\right]^{1/2} & \text{SHE with bounded $\rho$ (Theorem 1.2 {\it ibid.});} \\
    \log \sup_{|x|\le R} u(t,x) \asymp \left[\log R\right]^{2/3}          & \text{PAM (Theorem 1.3 {\it ibid.}).}
  \end{dcases}
\end{align}
Here and in this paper, we use the notation $f(x) \lesssim g(x)$ to denote that
there exists a nonrandom constant $C>0$ such that
$\liminf_{x\to\infty}f(x)/g(x)\le C$; the notation $f(x)\gtrsim g(x)$ is defined
analogously. We also write $f(x)\asymp g(x)$ if both $f(x) \lesssim g(x)$ and
$f(x) \gtrsim g(x)$. The distinct behaviors exhibited in~\eqref{E:Conus-SpAsym}
naturally raise questions regarding the dynamics when $\rho$ is neither bounded
nor linear, but rather demonstrates sublinear growth, thereby providing a
potential interpolation between these two extreme scenarios. To address the
scenario of sublinear growth, a significantly different method must be
developed, which constitutes the primary contribution of this paper.\medskip

Sublinear examples of $\rho(u)$ (for $u>0$) typically include
\begin{gather}\label{E:Ex-LipRho}
  \rho(u) = \frac{u}{(r+u)^{1-\alpha}}, \quad \text{$\alpha\in[0,1)$ and $r\ge 0$}; \\
  \rho(u) = u^\alpha \left[\log(e + u^2)\right]^{-\beta}, \quad \text{with }
  \begin{cases}
    \text{$\alpha = 0$ and $\beta<0$}        & \text{Case (i)} , \\
    \text{$\alpha\in(0,1)$ and $\beta\in\R$} & \text{Case (ii)} ,\\
    \text{$\alpha = 1$ and $\beta>0$}        & \text{Case (iii)};
  \end{cases}
  \label{E:Ex-LogRho} \shortintertext{and}
  \rho (u) = u\exp\left( -\beta \left(\log \big(\log (e+u^2\,)\big)\right)^{\kappa} \right)\quad \text{with $\kappa > 0$ and $\beta>0$}.
  \label{E:Ex-VSV}
\end{gather}
Here, $e$ denotes the \textit{Euler constant}. In particular, letting $r=1$ and
$\alpha=0$ in~\eqref{E:Ex-LipRho}, we reduce to the case of bounded $\rho$
in~\eqref{E:Ex-BoundedRho}. The following important example is a special case
of~\eqref{E:Ex-LipRho} when $r=0$:
\begin{align}\label{E:Ex-AlphaRho}
  \rho(u) = u^\alpha, \quad \alpha\in(0,1).
\end{align}
It should be noted that SHEs with the diffusion coefficient given
in~\eqref{E:Ex-AlphaRho} are closely related to
superprocesses; see,
e.g.,~\cite{dawson:93:measure-valued,etheridge:00:introduction,perkins:02:dawson-watanabe}.
Among the three cases in~\eqref{E:Ex-LogRho}, cases~(i) and (iii) are more
interesting, as they are logarithmic perturbations of the additive SHE and the
PAM, respectively. When $\kappa =1$, $\rho$ in~\eqref{E:Ex-VSV} reduces to
case~(iii) of~\eqref{E:Ex-LogRho}. Figure~\ref{F:Examples} below illustrates the
hierarchy of these examples. One would expect properties, such as the moment
growth rates and spatial asymptotics, to transition from those of the additive
SHE to those of the PAM. For example, for the case~\eqref{E:Ex-LipRho}, one
would expect a polynomial growth in $t$ of moments, while the
case~\eqref{E:Ex-VSV} should lead to some exponential growth, but with a
sublinear dependence on $t$ in the exponent. Note that examples
in~\eqref{E:Ex-LipRho}--\eqref{E:Ex-VSV} are all for the case when
$u\ge 0$. If the solution is signed, one needs to replace $u$ by $|u|$.

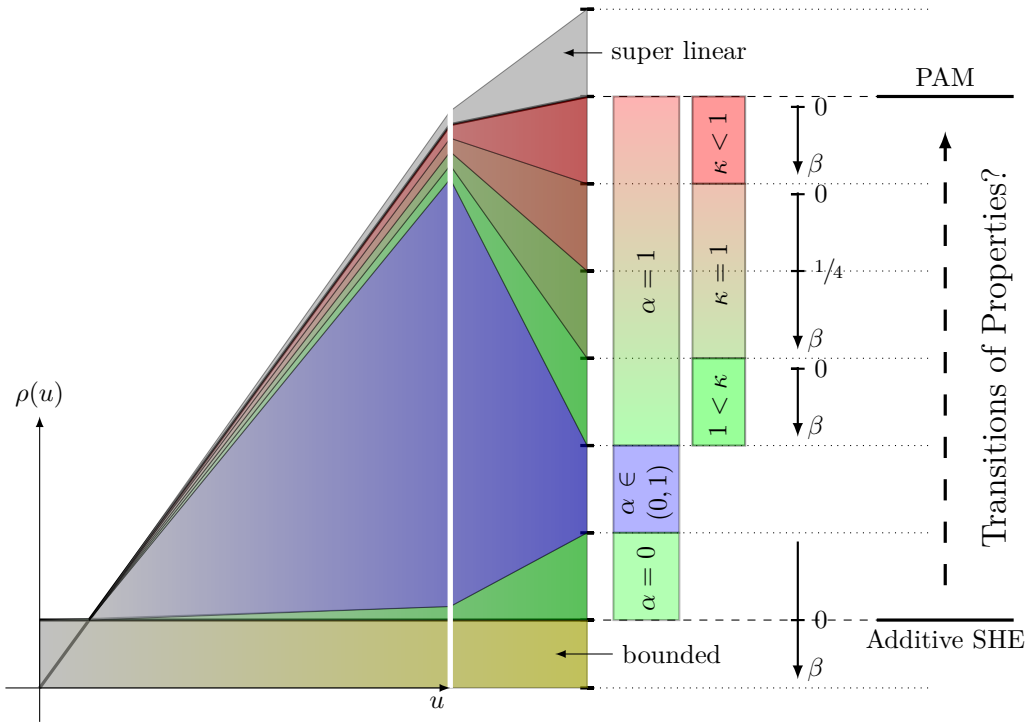
\begin{figure}[htpb!]
  \centering
  \begin{center}
    \begin{tikzpicture}[scale=0.9, transform shape]
      \tikzset{>=latex}
      \def\vd{0.2}
      \def\vx{6}
      \def\mag{2}
      \def\vy{9.0}
      \def\magbar{5.5}
      \newcommand{\MyColorH}{white}
      \newcommand{\MyColorG}{red}
      \newcommand{\MyColorF}{red!80!green}
      \newcommand{\MyColorE}{red!30!green}
      \newcommand{\MyColorD}{green}
      \newcommand{\MyColorC}{blue}
      \newcommand{\MyColorB}{green}
      \newcommand{\MyColorA}{yellow}

      \coordinate (o0) at (0,0);
      \coordinate (o1) at (0,1);           \coordinate (M0) at (\vx+\mag,0);
      \coordinate (R1) at (\vx,1.0);       \coordinate (M1) at (\vx+\mag,1);
      \coordinate (R2) at (\vx,1.0+1*\vd); \coordinate (M2) at (\vx+\mag,1+1*\vy/7);
      \coordinate (R3) at (\vx,8.5-5*\vd); \coordinate (M3) at (\vx+\mag,1+2*\vy/7);
      \coordinate (R4) at (\vx,8.5-4*\vd); \coordinate (M4) at (\vx+\mag,1+3*\vy/7);
      \coordinate (R5) at (\vx,8.5-3*\vd); \coordinate (M5) at (\vx+\mag,1+4*\vy/7);
      \coordinate (R6) at (\vx,8.5-2*\vd); \coordinate (M6) at (\vx+\mag,1+5*\vy/7);
      \coordinate (R7) at (\vx,8.5-1*\vd); \coordinate (M7) at (\vx+\mag,1+6*\vy/7);
      \coordinate (R8) at (\vx,8.5-0*\vd); \coordinate (M8) at (\vx+\mag,1+7*\vy/7);

      \begin{scope}[very thick]
        \draw [name path = Add] (0,1) -- (R1);
        \draw [name path = PAM] (0,0) -- (R7);
        \draw (R7) -- (M7);
        \draw (R1) -- (M1);
        \path[name intersections={of=Add and PAM}]
          (intersection-1) coordinate (o2);

        \draw (M0) --++ (0.1,0) --++(-0.2,0);
        \draw (M1) --++ (0.1,0) --++(-0.2,0);
        \draw (M2) --++ (0.1,0) --++(-0.2,0);
        \draw (M3) --++ (0.1,0) --++(-0.2,0);
        \draw (M4) --++ (0.1,0) --++(-0.2,0);
        \draw (M5) --++ (0.1,0) --++(-0.2,0);
        \draw (M6) --++ (0.1,0) --++(-0.2,0);
        \draw (M7) --++ (0.1,0) --++(-0.2,0);
        \draw (M8) --++ (0.1,0) --++(-0.2,0);

        \draw ([xshift = 11em]M7) --++ (2,0) node [midway,above]  (A) {PAM};
        \draw ([xshift = 11em]M1) --++ (2,0) node [midway, below] (B) {Additive SHE};
        \draw[shorten >= 1.2em, shorten <= 1.2em, ->, dash pattern=on 8pt off 8pt] (B) -- (A) node [midway, below, sloped, yshift = -1em] {\Large Transitions of Properties?};
      \end{scope}

      \begin{scope}[dotted]
        \def\myShift{5}
        \draw         (M0) --++ (\myShift,0);
        \draw[dashed] (M1) --++ (\myShift,0);
        \draw         (M2) --++ (\myShift,0);
        \draw         (M3) --++ (\myShift,0);
        \draw         (M4) --++ (\myShift,0);
        \draw         (M5) --++ (\myShift,0);
        \draw         (M6) --++ (\myShift,0);
        \draw[dashed] (M7) --++ (\myShift,0);
        \draw         (M8) --++ (\myShift,0);
      \end{scope}

      \begin{scope}[thick]
        \def\myShift{8em}

        \begin{scope}[shorten >=0.3em, shorten <=0.3em]
          \draw[->] ([xshift = \myShift]M6) -- ([xshift = \myShift]M4) node [right, yshift = 0.7em] {$\beta$};
          \draw[->] ([xshift = \myShift]M2) -- ([xshift = \myShift]M0) node [right, yshift = 0.7em] {$\beta$};
          \draw[->] ([xshift = \myShift]M4) -- ([xshift = \myShift]M3) node [right, yshift = 0.7em] {$\beta$};
          \draw[->] ([xshift = \myShift]M7) -- ([xshift = \myShift]M6) node [right, yshift = 0.7em] {$\beta$};
        \end{scope}

        \draw ([xshift = \myShift, yshift = -0.4em]M7) --++(-0.1,0) --++(0.2,0) node [right] {$0$};
        \draw ([xshift = \myShift, yshift = -0.4em]M6) --++(-0.1,0) --++(0.2,0) node [right] {$0$};
        \draw ([xshift = \myShift, yshift = -0.4em]M4) --++(-0.1,0) --++(0.2,0) node [right] {$0$};
        \draw ([xshift = \myShift]M5) --++(-0.1,0) --++(0.2,0) node [right] {$\sfrac{1}{4}$};
        \draw ([xshift = \myShift]M1) --++(-0.1,0) --++(0.2,0) node [right] {$0$};
      \end{scope}

      \begin{scope}[thick, every node/.style={right, rotate=90, black, opacity = 1, align = center}]
        \def\myShift{1.0em}
        \def\myShiftPlus{3.5em}
        \draw[top color = \MyColorG, bottom color = \MyColorD,opacity =0.3] ([xshift = \myShift]M3) rectangle node [xshift = -2.0em] {$\alpha=1$} ([xshift = \myShiftPlus]M7);
        \draw[fill = \MyColorC, opacity =0.3]                               ([xshift = \myShift]M2) rectangle node [xshift = -2.0em, text width =3em, ] {$\alpha\in$\\$(0,1)$} ([xshift = \myShiftPlus]M3);
        \draw[fill = \MyColorB, opacity =0.3]                               ([xshift = \myShift]M1) rectangle node [xshift = -1.7em] {$\alpha=0$} ([xshift = \myShiftPlus]M2);
      \end{scope}

      \begin{scope}[thick, every node/.style={right, rotate=90, black, opacity = 1, xshift = -1.7em}]
        \def\myShift{4em}
        \def\myShiftPlus{6em}
        \draw[fill = \MyColorG,opacity =0.4]                                 ([xshift = \myShift]M6) rectangle node {$\kappa<1$} ([xshift = \myShiftPlus]M7);
        \draw[top color = \MyColorF, bottom color = \MyColorE, opacity =0.3] ([xshift = \myShift]M4) rectangle node {$\kappa=1$} ([xshift = \myShiftPlus]M6);
        \draw[fill = \MyColorD,opacity =0.4]                                 ([xshift = \myShift]M3) rectangle node {$1<\kappa$} ([xshift = \myShiftPlus]M4);
      \end{scope}

      \begin{scope}[left color = white, opacity = 0.4, draw = none]
        \filldraw [right color = \MyColorA] (o0) -- (o1) -- (R1) -- (M1) -- (M0) -- (o0);
        \filldraw [right color = \MyColorB] (o1) -- (R1) -- (M1) -- (M2) -- (R2) -- (o1);
        \filldraw [right color = \MyColorC] (o2) -- (R2) -- (M2) -- (M3) -- (R3) -- (o2);
        \filldraw [right color = \MyColorD] (o2) -- (R3) -- (M3) -- (M4) -- (R4) -- (o2);
        \filldraw [right color = \MyColorE] (o2) -- (R4) -- (M4) -- (M5) -- (R5) -- (o2);
        \filldraw [right color = \MyColorF] (o2) -- (R5) -- (M5) -- (M6) -- (R6) -- (o2);
        \filldraw [right color = \MyColorG] (o2) -- (R6) -- (M6) -- (M7) -- (R7) -- (o2);
        \filldraw [right color = \MyColorH] (o2) -- (R7) -- (M7) -- (M8) -- (R8) -- (o2);
      \end{scope}

      \filldraw[white, line width = 2pt] (\vx,0) -- (\vx, 10);
      \draw [->] (-0.5,0) -- (\vx,0) node [below, xshift =-0.5em] {$u$};
      \draw [->] (0,-0.5) -- (0,  4) node [above] {$\rho(u)$};

      \begin{scope}[xshift =29em]
        \path (M0) -- (M1) node [midway, xshift = +3.2em] (A) {bounded};      \draw[->] (A) --++(-1.7,0);
        \path (M8) -- (M7) node [midway, xshift = +3.5em] (A) {super linear}; \draw[->] (A) --++(-1.7,0);
      \end{scope}

    \end{tikzpicture}
  \end{center}
  \caption{Examples of $\rho$ given in~\eqref{E:Ex-LipRho}--\eqref{E:Ex-VSV}.}
  \label{F:Examples}
\end{figure}

Let us proceed to set up the problem. The noise $\dot{W}$ in SHE~\eqref{E:SHE}
is a centered Gaussian noise that is white in time and homogeneously colored in
space. Its covariance structure is given by
\begin{align*}
	\E \left[\dot{W}(t,x) \dot{W}(s,y)\right] = \delta_0 (t-s) f(x-y).
\end{align*}
Here, $\delta_0$ denotes the Dirac delta measure at $0$ and $f$ is the
correlation (generalized) function on $\R^d$, which satisfies the following
hypothesis:

\begin{hypothesis}\label{H:corre}
  The correlation function $f:\R^d \to \R$ is a nonnegative and
  nonnegative-definite (generalized) function that is not identically zero, such
  that the following \textit{Dalang's condition}~\cite{dalang:99:extending} is
  satisfied:
	\begin{align}\label{E:Dalang}
		\int_{\R^d} \frac{\widehat{f}(\ud\xi)}{1 + |\xi|^2}<\infty,
	\end{align}
  where $\widehat{f}(\xi) = \int_{\R^d} f(x) e^{-i x\cdot \xi}\ud x$ is the
  Fourier transform of $f$.
\end{hypothesis}

By using the Fourier transform and the Plancherel theorem, it is easy to verify
that Dalang's condition~\eqref{E:Dalang} is equivalent to the following
condition (see, e.g.,~\cite[Formula (20)]{dalang:99:extending}):
\begin{align}\label{E:h}
		h(t)
    \coloneqq \int_0^t \ud s \iint_{\R^{2d}} \ud y \ud y'\: p_s(y) p_s(y') f(y-y')
    = \frac{1}{2}\int_0^{2t} \ud s \int_{\R^d} \ud z\: p_{s}(z) f(z)
    < \infty, \quad \forall t>0,
\end{align}
where $p_t(x) \coloneqq \left(2\pi\right)^{-d/2}\exp\left(-|x|^2/(2t)\right)$
refers to the heat kernel. The function $h$ plays an essential role in our main
result---Theorem~\ref{T:mup} below, whose asymptotic behaviors both at infinity
and around zero will be postponed to Appendix~\ref{S:app}. Note that the
nonnegativity assumption of $f$ in Hypothesis~\ref{H:corre} ensures that the
function $h(\cdot)$ is an increasing function on $\R_+$. \bigskip

To facilitate the analysis, we introduce the following hypothesis on $\rho$,
which covers all examples given in~\eqref{E:Ex-LipRho}--\eqref{E:Ex-VSV} with all $u$ replaced $|u|$:

\begin{hypothesis}\label{H:rho}
  The diffusion coefficient $\rho:\R\to \R$ satisfies the following properties:
  \begin{enumerate}[(i)]
    \item $\rho$ is a locally bounded function.
    \item $\displaystyle \lim_{x\to \pm \infty} \frac{\rho(x)}{x} = 0$;
    \item There exists a constant $M_0\geq 0$, such that the function $|\rho|$
      is concave separately on $(-\infty, -M_0]$ and $[M_0,\infty)$.
  \end{enumerate}
\end{hypothesis}

The initial condition $\mu$ in~\eqref{E:SHE} also plays an active role in
shaping the properties of the solution, for which we make the following
assumption:

\begin{hypothesis}\label{H:rough}
  The initial condition $\mu$ for SHE~\eqref{E:SHE} is a signed Borel
  measure\footnote{We follow the convention that when $\mu$ is absolutely
  continuous with the Lebesgue measure, it is identified as its Lebesgue
density.} on $\R^d$ such that
\begin{enumerate}[(i)]
 \item For any $(t,x) \in \R_+\times \R^d$, it holds that, in the sense of
   Lebesgue integral,
   \begin{align}\label{E:J0}
     \begin{aligned}
      & \calJ_0(t,x)   \coloneqq \int_{\R^d} \mu(\ud y)\: p_{t} (x-y) \in (-\infty,\infty), \quad \text{or equivalently,} \\
      & \calJ_+(t,x) \coloneqq \int_{\R^d} |\mu| (\ud y)\: p_{t} (x-y) < \infty,
     \end{aligned}
   \end{align}
   where $\mu = \mu_+ - \mu_-$ is the Hahn decomposition of $\mu$ and $|\mu| =
   \mu_+ + \mu_-$;
  \item Moreover, if $d\geq 2$, then for any $(t,x) \in \R_+\times \R^d$,
    \begin{align}\label{E:J1}
      \calJ_1(t,x)
      \coloneqq \int_0^t \ud s \iint_{\R^{2d}} \ud y\ud y'\:
      p_{t-s} (x-y) p_{t-s} (x-y') f(y - y') \calJ_0^2(s,y)
      <\infty.
    \end{align}
  \end{enumerate}
\end{hypothesis}

As usual, the solution to~\eqref{E:SHE} is understood as the mild solution to
the corresponding stochastic integral equation:
\begin{align}\label{E:Mild}
  u(t,x) = \mathcal{J}_0(t,x) + \int_0^t \int_{\R^{d}} p_{t-s}(x-y) \rho\left(u(s,y)\right) W(\ud s,\ud y), \quad t>0,\, x\in \R^d,
\end{align}
where the stochastic integral is understood in the sense of
Walsh~\cite{walsh:86:introduction, dalang:99:extending}. Now we are ready to
state our main result as follows:

\begin{theorem}\label{T:mup}
  Under Hypotheses~\ref{H:corre},~\ref{H:rho}, and~\ref{H:rough}, let $u(t,x)$
  be a solution to SHE~\eqref{E:SHE}. Then, for all $(t,x,p)\in \R_+ \times
  \R^d\times[2,\infty)$, it holds that
  \begin{align}\label{E:mup}
    \Norm{u(t, x)}_p^2 \leq
      & 2 \calJ_0^2(t,x)  + 2 (2\pi)^{d}\left(  h(t)^{-1} \calJ_1(t,x) + 4 K_M^2 p\: h(t)  + F^{-1}( 2 p\: h(t))   \right).
  \end{align}
  In~\eqref{E:mup}, $\calJ_0(\cdot,\circ)$, $\calJ_1(\cdot,\circ)$ and
  $h(\cdot)$ are defined in~\eqref{E:J0},~\eqref{E:J1} and~\eqref{E:h} above,
  respectively. The constant $K_M$ is defined as
  \begin{align}\label{E:KM}
    K_M \coloneqq \sup_{x\in (-M, M)} |\rho(x)|,
  \end{align}
  with $M$ given in part (iii) of Lemma~\ref{L:ccvtrho} below. The function
  $F(\cdot)$ and its inverse $F^{-1}(\cdot)$, both determined by $\rho$,
  are defined in~\eqref{E:F} and~\eqref{E:F^-1}, respectively. In particular,
  the following statements hold:
  \begin{enumerate}[(i)]
    \item If $d=1$, one can replace $4h(t)^{-1}\calJ_1(t,x)$ in~\eqref{E:mup} by
      $2^{7/2} \pi \calJ_+^2(t/2,x)$ (see~\eqref{E:J0} for notation);

    \item If $|\rho(\cdot)|$ is concave separately on $\R_+$ and $\R_-$, then
      one can take $M = K_M = 0$ in~\eqref{E:mup} to obtain
      \begin{align}\label{E:mup_M0}
        \Norm{u(t, x)}_p^2 \leq & 2\: \calJ_0^2(t,x) + 2 (2\pi)^{d} \left( h(t)^{-1} \calJ_1(t,x) + F^{-1}(2 p\: h(t)) \right);
      \end{align}

    \item If $\rho(\cdot)$ is not identically $0$ on $(-\infty, -M_0] \cup
      [M_0,\infty)$, then there exists a constant $C>0$ such that for all
      $(t,x,p) \in [1,\infty) \times \R^d \times [2,\infty)$,\footnote{Here $t$
      can be relaxed to $t > 0$, but the constant $C$ in~\eqref{E:mup_asy} will
    depend on $t$ when $t$ is close to $0$}
      \begin{align}\label{E:mup_asy}
        \Norm{u(t,x)}_p^2 \leq & 2\: \calJ_0^2(t,x) + C \left(h(t)^{-1} \calJ_1(t,x) + F^{-1}\big(C p\: h(t)\big)\right).
      \end{align}
      Moreover, if the initial condition is bounded, i.e., $u_0\in
      L^{\infty}(\R^d)$, the moment bound in~\eqref{E:mup_asy} can be simplified
      as follows:
      \begin{align}\label{E:mup_smp}
        \Norm{u(t,x)}_p^2 \leq C_* F^{-1}\big( C_* p\, h(t)\big)\quad \text{for some $C_* > 0$.}
      \end{align}
  \end{enumerate}
\end{theorem}

The general moment bounds in the above theorem demonstrate how different
components of the SPDE affect the moment growth of the solution (see
Figure~\ref{F:Components} below), which is a concrete manifestation of the
statement in Zeldovich \textit{et
al}~\cite{zel-dovich.molchanov.ea:87:self-excitation} (see
also~\cite[Section~8.9]{zel-dovich.ruzmauikin.ea:90:almighty}) that \textit{the
  behavior of nonlinear solutions depends radically on the time behavior of the
potential and on the form of the nonlinearity.} When considering specific
scenarios, in order to utilize the above theorem, we have to study the
corresponding $F^{-1}(\cdot)$ as detailed in Section~\ref{S:Rho}, and $h(t)$ as
outlined in Section~\ref{S:app}. Following these analyses, we summarize the
results and derive the specific moment bounds in Section~\ref{S:Application}
below.

\begin{figure}[htpb!]
  \begin{center}
    \begin{tikzpicture}[scale=1, transform shape, -{Stealth[scale=1.5]}, double distance =2pt]

    \node (0) at (0,0) {$\Norm{u(t,x)}_p^2$};
    \node [right of = 0, xshift = 1em] (1)  {$\leq 2\:$};
    \node [right of = 1] (2)  {$\calJ_0^2(t,x)$};
    \node [right of = 2, xshift = 0.5em] (3)  {$+\: C \bigg($};
    \node [right of = 3] (4)  {$h(t)^{-1}$};
    \node [right of = 4, xshift = 0.5em] (5)  {$\calJ_1(t,x)$};
    \node [right of = 5, xshift =-0.3em] (6)  {$+$};
    \node [right of = 6, xshift =-1.0em] (7)  {$F^{-1}$};
    \node [right of = 7, xshift =-0.8em] (8)  {$\big(C p\:$};
    \node [right of = 8, xshift =-0.9em] (9)  {$\:\,h(t)$};
    \node [right of = 9, xshift =-0.9em] (10) {$\big)\bigg)$};

    \node[above of = 7, yshift = 1em] (Rho) {Sublinear $\rho$};
    \draw[double] (Rho) -- (7);

    \node[below of = 6, yshift = -2em] (Noise) {Noise structure $f$};
    \draw[double] (Noise) -- (9.south);
    \draw[double] (Noise) -- (4.south);

    \node[above of = 4, yshift = +2em, xshift = -0.5em] (Init) {Initial data $\mu$};
    \draw[double] (Init) -- (2.north);
    \draw[double] (Init) -- (5.north);
    \end{tikzpicture}
  \end{center}

  \caption{Structure of the moment bounds in~\eqref{E:mup_M0}
  or~\eqref{E:mup_asy}.}

  \label{F:Components}
\end{figure}
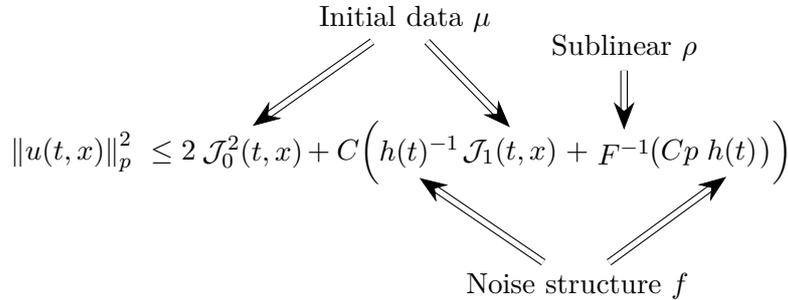

As some applications of the above moment bounds, we obtain the tail probability
of the solution, sample-path H\"older regularity, and further establish some
estimates for the sample-path asymptotics in the spatial variable, as shown
in~\eqref{E:asymspc} below. The sample-path spatial asymptotics have been
previously studied by Conus \textit{et al.}~\cite{conus.joseph.ea:13:on,
conus.joseph.ea:13:on*1} in two extreme cases, namely, the case when $\rho$ is
linear and that when $\rho$ is bounded. The result presented in
Theorem~\ref{T:Asymspc} below serves as an initial attempt to bridge the gap
between these two extreme cases by allowing $\rho$ to have sublinear growth.
Here, we have to admit that only the asymptotic upper bounds have been obtained,
while the more challenging lower bounds will be left for future investigation.

\begin{theorem}[Tail probability]\label{T:Tail}
  Assume that both Hypotheses~\ref{H:corre} and~\ref{H:rho} hold. Let $u(t,x)$
  be a solution to SHE~\eqref{E:SHE} with the initial condition $\mu \equiv 1$.
  Then for all $t \geq 1$, $x\in\R^d$ and $z \geq L_t$,
  \begin{align}\label{E:tail}
    \P\left(|u(t,x)| \geq z\right)
    \leq \exp \left(- (C_* h(t))^{-1} F \left(z^2\big/ \big(C_* e^2\big)\right)\right),
  \end{align}
  where the constant $C_* > 0$ is the same as that in~\eqref{E:mup_smp},
  $F(\cdot)$ and $F^{-1}(\cdot)$ are defined Definition~\ref{D:F},
  \begin{align}\label{E:Lt}
    L_t \coloneqq e\sqrt{2C_* M^2 + C_* F^{-1} \big(2 C_* h(t)\big)}\:,
  \end{align}
  and the constant $M$ in~\eqref{E:Lt} is the same as those in
  Theorem~\ref{T:mup}.
\end{theorem}

To state the next two results, we need to introduce the following
\textit{improved Dalang condition} for the spatial correlation function $f$:

\begin{hypothesis}\label{H:efcdalang}
  The correlation function $f:\R^d \to \R$ is a nonnegative and
  nonnegative-definite (generalized) function that is not identically zero such
  that,
  \begin{align}\label{E:ED}
    \int_{\R^d} \frac{\widehat{f}(\ud\xi)}{(1 + |\xi|^2)^{1-\eta}}<\infty,
    \quad \text{for some $\eta \in (0,1)$.}
  \end{align}
\end{hypothesis}

\begin{theorem}[H\"older regularity]\label{T:Holder}
  Under parts (i) and (ii) of Hypothesis~\ref{H:rho}, part (i) of
  Hypothesis~\ref{H:rough}, and Hypothesis~\ref{H:efcdalang}, the solution
  $u(t,x)$ to~\eqref{E:SHE} has a version which is a.s. $\eta_1$--H\"older
  continuous in time and $\eta_2$--H\"older continuous in space on $(0,\infty)
  \times \R^d$ for all $\eta_1\in (0,\eta /2)$ and $\eta_2\in (0,\eta)$, where
  $\eta$ is given in~\eqref{E:ED}.
\end{theorem}

\begin{theorem}[Spatial asymptotics]\label{T:Asymspc}
  Assume that both Hypotheses~\ref{H:rho} and~\ref{H:efcdalang} hold. Suppose
  that $\rho(\cdot)$ is not identically zero on $(- \infty, - M_0] \cup
  [M_0,\infty)$. Let $u(t,x)$ be a solution to SHE~\eqref{E:SHE} with the
  constant initial condition $\mu \equiv 1$. Then, there exists a positive
  constant $C$ such that for all $t > 1$,
  \begin{align}\label{E:asymspc}
    \sup_{|x|\le R} u(t,x) \lesssim \sqrt{F^{-1} \Big(C\: h(t) \log R\Big)}\:,
    \quad \text{a.s., as $R \to \infty$,}
  \end{align}
  where $F^{-1}(\cdot)$ is defined in~\eqref{E:F^-1}.
\end{theorem}

The paper is organized as follows. We begin by expanding the discussions of the
main results in Section~\ref{S:Remark}, including a presentation of the proof
strategy of Theorem~\ref{T:mup}. We then proceed to prove the main results in
Section~\ref{S:main}. In Section~\ref{SS:lms}, we introduce several technical
lemmas, and their proofs can be found in Section~\ref{SS:lmmpf}. The proofs of
Theorems~\ref{T:mup},~\ref{T:Tail},~\ref{T:Holder} and~\ref{T:Asymspc} are
provided in Sections~\ref{SS:prof_mup},~\ref{SS:Tail},~\ref{SS:Holder}
and~\ref{SS:asymspc}, respectively. In Section~\ref{S:Rho}, we derive $F(\cdot)$
and $F^{-1}(\cdot)$ for examples in~\eqref{E:Ex-LipRho}--\eqref{E:Ex-VSV}.
Section~\ref{S:Others} explores potential generalizations of Theorem~\ref{T:mup}
to other SPDEs. In Section~\ref{S:Application}, we present several concrete
examples to illustrate, for example, the property transitions from the additive
SHE to the PAM. Lastly, in Appendix~\ref{S:app}, we present asymptotic formulas
for the function $h(\cdot)$ in various cases, and in Appendix~\ref{S:app_exam},
we provide a supplementary proof for results in Section~\ref{S:Application}.

\section{Remarks and strategy of the proof}\label{S:Remark}

In the following, we make a few comments/discussions on our assumptions and
results, and present the strategy of the proof of Theorem~\ref{T:mup}.

\paragraph{Assumptions on initial conditions} Here are some comments on the
assumptions made for the initial conditions in Hypothesis~\ref{H:rough}:
\textbf{(i)} Following~\cite{chen.dalang:15:moments}, we call the initial
condition satisfying inequality~\eqref{E:J0} the \textit{rough initial
condition} for SHE~\eqref{E:SHE}. In particular, the Dirac delta measure is a
special case; see also~\cite{conus.joseph.ea:14:initial}. Note that the Dirac
delta measure plays an important role in studying the asymptotic properties of
the PAM on $\R$; see~\cite{
amir.corwin.ea:11:probability,corwin:12:kardar-parisi-zhang}. As an easy exercise, condition~\eqref{E:J0} is
equivalent to
\begin{align*}
  -\infty < \int_{\R^d} e^{-a |x|^2} \mu(\ud x) <\infty, \quad \text{for all $a>0$}.
\end{align*}
\noindent\textbf{(ii)} Condition~\eqref{E:J1} is only a technical assumption. We
believe that this assumption can be removed. This will be left for future
investigation. If $\mathcal{J}_0(s,y)^2$ in~\eqref{E:J1} is replaced by
$\mathcal{J}_0(s,y) \mathcal{J}_0(s,y')$, then due to~\cite[Lemma
2.7]{chen.kim:19:nonlinear}, the integral is finite under Dalang's
condition~\eqref{E:Dalang} for all rough initial conditions. This extra
condition comes from the application of the inequality $\mathcal{J}_0(s,y)
\mathcal{J}_0(s,y')\le \frac{1}{2}\left(\mathcal{J}_0^2(s,y)+
\mathcal{J}_0^2(s,y')\right)$ in the proof of Theorem~\ref{T:mup}.
\textbf{(iii)} In case $d = 1$, condition~\eqref{E:J1} is automatically
satisfied by Lemma~\ref{L:J1d=1} below. On the other hand, for $d\ge 2$,
condition~\eqref{E:J1} excludes the Dirac delta initial condition. This can be
seen by setting $f(\cdot) \equiv 1$ (the space-independent noise), then the
integral in~\eqref{E:J1} with delta initial condition reduces to $\int_0^t s^{-d
/2} \ud s = \infty$. Instead, condition~\eqref{E:J1} holds when $\mu(\ud x) =
|x|^{-\ell}\ud x$ for any $\ell\in (0,1)$. This is due to the bound
in~\eqref{E:J1J0} and the fact that for such initial data, $\calJ_0(t,x)
\lesssim t^{-\ell/2}$; see~\cite{chen.eisenberg:22:invariant}. Roughly speaking,
the Dirac delta measure corresponds to the case when $\ell = d$. Removing
part~(ii) of Hypothesis~\ref{H:rough} will be left for future investigation.

\paragraph{Existence and uniqueness: the global Lipschitz case} If $\rho$ is
globally Lipschitz continuous such as those
in~\eqref{E:Ex-LipRho}--\eqref{E:Ex-VSV}, the existence and uniqueness of a
random field solution to SHE~\eqref{E:SHE} under rough initial conditions is
guaranteed (see~\cite{chen.dalang:15:moments*1, chen.kim:19:nonlinear}).
Moreover, since $\rho(0)=0$ in all
examples~\eqref{E:Ex-LipRho}--\eqref{E:Ex-VSV}, if the initial condition $\mu$
is a nonnegative measure and is not vanishing, then the solution $u(t,x)$ is
strictly positive almost surely for all $t>0$ and $x\in\R^d$ thanks to the
\textit{sample-path comparison principle} (see~\cite{mueller:91:long,
shiga:94:two, chen.huang:19:comparison}) One can further apply the
\textit{stochastic (moment) comparison principle}
(see~\cite{joseph.khoshnevisan.ea:17:strong, chen.kim:20:stochastic}) to bound
the moments from above by those of the \textit{parabolic Anderson model}
(see~\cite{carmona.molchanov:94:parabolic}), which corresponds to the case when
$\rho(u)= \lambda u$. However, moment bounds obtained in this way are generally
too rough and do not exhibit the ``smoothing intermittency'' effect as observed
in Zeldovich {\it et al}~\cite{zel-dovich.ruzmauikin.ea:90:almighty}.

\paragraph{Existence and uniqueness: the non-Lipschitz case} In case when $\rho$ is
not Lipschitz continuous, such as the example in~\eqref{E:Ex-AlphaRho},
Theorem~\ref{T:mup} can be used to provide some a priori moment estimates. It is
known that establishing the uniqueness of solutions to SHE~\eqref{E:SHE} in this
case is a challenging problem. However, it is commonly known that the existence
of solutions to~\eqref{E:SHE} can usually be established using some common
strategy. One first mollifies the non-Lipschitz $\rho$ function into a sequence
of globally Lipschitz continuous functions, based on which a sequence of random
fields are constructed. Under some mild conditions, it is possible to prove that
these random fields are jointly H\"{o}lder continuous with a uniform constant.
By combining the H\"{o}lder continuity with a priori estimates for solutions,
such as those given in Theorem~\ref{T:mup}, one can apply the
\textit{Kolmogorov-Chentsov criterion} (see, e.g.,~\cite[Corollary
16.9]{kallenberg:02:foundations}) to conclude that there is a convergent
subsequence of the aforementioned random fields, whose limit gives rise to one
solution to~\eqref{E:SHE}. Since the existence and uniqueness of solution is not
the focus of the paper. We will not pursue this direction here. Finally, we
would like to point out that the strong or even the weak well-posedness problem
of SHEs with non-Lipschitz coefficients is highly involved, and only a few
results are known. One may consult~\cite{mytnik:98:weak,
mytnik.perkins.ea:06:on, mytnik.perkins:11:pathwise} for the sublinear case, and
the recent works in~\cite{dalang.khoshnevisan.ea:19:global, salins:22:global,
chen.huang:22:superlinear} for the superlinear growth case.

\paragraph{Generality versus sharpness} The moment bounds obtained in
Theorem~\ref{T:mup} strike a balance between their level of generality and their
sharpness. Obtaining sharp moment asymptotics in general can be extremely
challenging and is typically only possible in some specific settings. For
instance, in the case when $d=1$, $\dot{W}$ is the space-time white noise, and
$\rho(u) = \sqrt{u}$ (the super-Brownian motion) with $u(0,x) \equiv 1$, the
second author and his collaborators~\cite{hu.wang.ea:23:moment} recently find
the following exact moment asymptotics:
\begin{align*}
  \E \left[u(t,x)^p\right]\asymp K^p p! \left(1 + t^{(p-1)/2}\right), \quad \text{as $t\to\infty$,}
\end{align*}
which is valid for any positive integer $p$ and $x\in \R$. This asymptotic
should be compared with the upper bounds obtained by Theorem~\ref{T:mup}:
\begin{align*}
  \E \left[u(t,x)^p\right]\leq K^p p! \left(1 + t^{p/2}\right).
\end{align*}
The difference between these bounds is a factor of $\sqrt{t}$, showing that the
bound in Theorem~\ref{T:mup} is not sharp. Although we only have this case
showing that our moment bounds are not sharp, we believe that this is generally
the case. Nonetheless, the method used in the proof of Theorem~\ref{T:mup} are
quite robust and can be easily extended to more general settings and a broad
class of stochastic partial differential equations (SPDEs), including the
\textit{stochastic wave equation} (SWE) (see Section~\ref{SS:SWEd}), and SPDEs
with fractional differential operators (see Section~\ref{SS:frac}). Moreover,
despite being sub-optimal, the moment bounds obtained in Theorem~\ref{T:mup} are
sufficient to provide a quantitative description of the ``smoothing
intermittency'' phenomenon introduced by Zeldovich \textit{et
al.}~\cite{zel-dovich.ruzmauikin.ea:90:almighty}.

\paragraph{H\"older regularity} The H\"{o}lder continuity for solutions to SHEs
has been extensively investigated in the literature. Notably, Konno and
Shiga~\cite{konno.shiga:88:stochastic} and
Reimers~\cite{reimers:89:one-dimensional} examined the H\"{o}lder continuity for
super-Brownian motions and Fleming--Viot processes driven by space-time white
noise with function-valued initial conditions. Sanz-Sol\'e and
Sarr\'a~\cite{sanz-sole.sarra:02:holder}, on the other hand, examined the
scenario where the noise satisfies the improved Dalang condition
(Hypothesis~\ref{H:efcdalang}) and $\rho$ satisfies the golobal Lipschitz
condition. Recently, Chen and Daland~\cite{chen.dalang:14:holder-continuity} and
Chen and Huang~\cite{chen.huang:19:comparison} generalized results
in~\cite{sanz-sole.sarra:02:holder} by allowing rough initial conditions (part
(i) of Hypothesis~\ref{H:rough}). Theorem~\ref{T:Holder} makes a slight
extension of \cite[Theorem~1.8]{chen.huang:19:comparison} by allowing the
non-Lipschitz condition on $\rho$, where the global Lipschitz condition is
replaced by ``locally bounded + linear growth at infinity'' of $\rho$. As far as
we know, the H\"{o}lder continuity for SHEs with $\rho(u)=\sqrt{u}$ starting
from the Dirac delta measure (one instance of the rough initial condition) has
not been studied in the literature. For the classical initial condition, one may
consult, e.g., \cite[Theorem~1.4.6]{xiong:13:three}.

\paragraph{Spatial asymptotics and tail estimates} Theorem~\ref{T:Asymspc}
provides an almost sure asymptotic upper bound for solutions to~\eqref{E:SHE} in
space. The proof of this theorem closely follows the approach presented
in~\cite{conus.joseph.ea:13:on}, which relies on tail estimates
(Theorem~\ref{T:Tail}) and the Borel--Cantelli lemma. While the moment bounds
may not be particularly sharp, we find that the tail estimates are indeed sharp,
at least in the case of super Brownian motion; see Proposition~\ref{P:Alpha}
with $\alpha = 1/2$ and~\cite[Proposition 1.4]{hu.wang.ea:23:moment}. As a
result, we believe Theorem~\ref{T:Asymspc} provides a sharp bound for
super-Brownian motion, specifically, for any $t > 0$ fixed,
\begin{align*}
  \sup_{|x|\leq R} u(t,x) \asymp \sqrt{t} \log (R),
  \quad \text{almost surely, as $R\to \infty$.}
\end{align*}
In Proposition~\ref{P:Alpha}, by letting $\alpha = 0$, the spatial asymptotic
upper bounds coincide with the exact asymptotics of SHEs with the additive
space-time white noise as proved in~\cite[Theorem 1.2]{conus.joseph.ea:13:on},
and with additive spatial colored noises as proved in~\cite[Theorem
2.3]{conus.joseph.ea:13:on*1}. However, Theorem~\ref{T:Tail} provides a tail
estimates obtained by approximating the Legendre-type transform of $H$, as seen
in~\eqref{E:H}. This approximation is effective only when $\rho(u)$ grows at a
``significantly slower'' rate than $u$ as $u \to \infty$, such as $\rho(u) =
u^{\alpha}$ with $\alpha \in [0,1)$. However, if $\rho(u)$ is ``very close'' to
$u$ for large $u$, such as $\rho(u) = u\left[\log(e + u)\right]^{-\beta}$ with
$\beta\in (0, \sfrac{1}{4})$, part (b) of Proposition~\ref{P:Log} below suggests
that the results in Theorem~\ref{T:Tail} and consequently in
Theorem~\ref{T:Asymspc} can be improve.

\paragraph{Interaction with initial conditions} Due to the multiplicative nature
of the noise, the initial condition interacts with the noise, playing an active
role in shaping the solution. This interaction is evident when one writes out
the Picard iteration of the mild solution given by~\eqref{E:Mild}. For the PAM,
this interaction leads to the following moment bound (see~\cite[Theorem
1.7]{chen.huang:19:comparison}):
\begin{align}\label{E:PAM-Mom}
  \Norm{u(t,x)}_p^2 \le C \calJ_0^2(t,x)\, \Upsilon(t), \quad \text{for all $t>0$, $x\in\R^d$, and $p\ge 2$,}
\end{align}
where $\Upsilon(\cdot)$ is a function that represents the contribution of the
driven noise. The \textit{multiplicative interaction} of the initial conditions
and the driven noise, as shown in~\eqref{E:PAM-Mom}, naturally gives rise to the
property of propagation of tall peaks in some space-time cone $\left\{|x| \leq
\kappa\, t\right\}$, which was earlier observed in physical contexts, see, e.g.,
Zeldovich \textit{et al.}~\cite[Section
8.10]{zel-dovich.ruzmauikin.ea:90:almighty} and later rigorously formulated and
proved by Conus and Khoshnevisan~\cite{conus.khoshnevisan:12:on}.  Since then,
additional researches (see~\cite{chen.dalang:15:moments,
chen.kim:19:nonlinear, huang.le.ea:17:large*1}) has expanded upon this cone
property. In essence, the cone property states that there exists a space-time
cone of size $\kappa$ within which the moments of the PAM grow exponentially
fast, while outside of it, they decay exponentially rapidly. The precise size
$\kappa$ of the cone is known as the \textit{intermittency front}.

However, in the present paper, we obtain an \textit{additive interaction} of the
initial data and the noise as shown in Theorem~\ref{T:mup} and
Figure~\ref{F:Components}. This additive interaction arises from the way we
solve the inequality~\eqref{E:MotIneq} (see Figure~\ref{F:Motive}) or from the
application of Lemma~\ref{L:Gmm} in general, where the coefficient $b$
corresponds to the initial condition. By sending $b$ to zero
in~\eqref{E:MotIneq}, the linear equation $x = k x^{\alpha} + b$ transitions
from having one unique nonnegative solution to having two nonnegative solutions,
one of which is zero. Accordingly, assuming $b=0$, the inequality $x\le k
x^{\alpha}$ can only imply that $x\in [0,k^{1/(1-\alpha)}\,]$, but additional
information is needed to determine the exact value of $x$. In the context of
SPDEs, such additional information may be related to the uniqueness and
non-uniqueness of the solution. In fact, when $\rho(u) = |u|^{\alpha}$ with
$\alpha\in (0,3/4)$ and $\dot{W}$ is the one-dimensional space-time white noise,
Mueller \textit{et al}~\cite{mueller.mytnik.ea:14:nonuniqueness} constructed
nontrivial solutions starting from zero initial condition. Therefore, the
propagation of high peaks (of polynomial order in this case) will be much more
subtle and will be left for future research.

If the diffusion coefficient $\rho$ is globally Lipschitz, the moment comparison
theorem may be applied, and thus the moment bounds for the ``dominant PAM'' can
be used to determine an upper bound for the propagation to the solution
to~\eqref{E:SHE}. In case that $\rho$ is not globally Lipschitz, it is known
that solutions to~\eqref{E:SHE}, assuming $\rho(u) = u^{\alpha}$ with $\alpha
\in (0,1)$, is compactly supported provided the initial condition is a finite
measure, as demonstrated
in~\cite{dawson.iscoe.ea:89:super-brownian,mueller.perkins:92:compact}.
Additionally, precise analysis on front propagation (a related but distinct
property) for Fisher-KPP equations has been presented
in~\cite{mueller.mytnik.ea:11:effect, mueller.mytnik.ea:21:speed}. As a result,
it should be possible to obtain results about propagation of tall peaks, even
for non-Lipschitz cases with certain necessary restrictions on initial
conditions. We hope that this question can be resolved in the future.

\paragraph{Proof strategy of Theorem~\ref{T:mup}} The proof of
Theorem~\ref{T:mup} relies on the concavity of $\rho$. To highlight the strategy
of the proof, consider the case of space-time white noise ($f=\delta_0$) and
assume that $M_0 = 0$ in Hypothesis~\ref{H:rho}. Using standard arguments
involving Burkholder-Davis-Gundy's (see, e.g.,
\cite[Theorem~B.1]{khoshnevisan:14:analysis}), Minkowski's and Jensen's
inequalities, one can derive the following inequality:
\begin{align}\label{E:toy}
  X(t,x) \leq k\rho_2(X(t,x)) + b, \quad \text{with} \quad
  X(t,x) = \int_0^t\ud s\int_{\R^d}\ud y\, p_{t-s}^2(x-y) \Norm{u(s,y)}_p^2,
\end{align}
where $\rho_2(\cdot)$ is a sublinear function defined on $\R_+$, $k$ and $b$ are
some constants depending on $p$ and $t$. Since $\rho_2$ is sublinear, any
nonnegative solution to inequality~\eqref{E:toy} has to lie in a compact
interval, e.g., $[0,F]$. In other words, as $X$ satisfies~\eqref{E:toy}, we get
$X\leq F$; see Lemma~\ref{L:Gmm}.

To illustrate the idea of solving inequality~\eqref{E:toy}, consider the case
when $\rho(u) = u^a$, $u\ge 0$, with $a\in [0,1)$ fixed. In this case,
$\rho_2(u) = \rho(u)$ and the inequality in~\eqref{E:toy} becomes
\begin{align}\label{E:MotIneq}
  x \le k x^a + b, \quad \text{for $x\ge 0$ with $a\in (0,1)$, $b>0$ and $k>0$ being fixed}.
\end{align}
By the concavity of the function $x^a$, the corresponding equation $x = k x^a +
b$ has a unique positive solution, which is denoted by $x^*$. Hence,
inequality~\eqref{E:MotIneq} holds provided that $x\in [0,x^*]$, i.e., $x^*$ is
an upper bound estimate for $x$. Moreover, one can apply Newton's method for one
step, properly started, to obtain an upper bound for $x^*$ (see
Figure~\ref{F:Motive} for an intuitive display of this procedure):
\begin{align}\label{E:Xstar}
  x^*\le k^{1/(1-a)} + b/(1-a).
\end{align}

To handle more general cases, including those where $M_0\neq 0$ (i.e., $\rho$ is
only asymptotic concave), the noise is not white in space, and the initial
conditions are more general (not a constant), a more meticulous approach is
required. This ultimately leads to the establishment of an inequality in the
form of~\eqref{E:toy}, as shown in~\eqref{E:nottoy} below. Lemma~\ref{L:Gmm}
outlines the procedure for identifying an upper bound similar to the one
depicted in Figure~\ref{F:Motive}, which in turn yields a bound akin
to~\eqref{E:Xstar} for solutions to~\eqref{E:nottoy}. Overall, this constitutes
the general strategy behind the proof of Theorem~\ref{T:mup}.

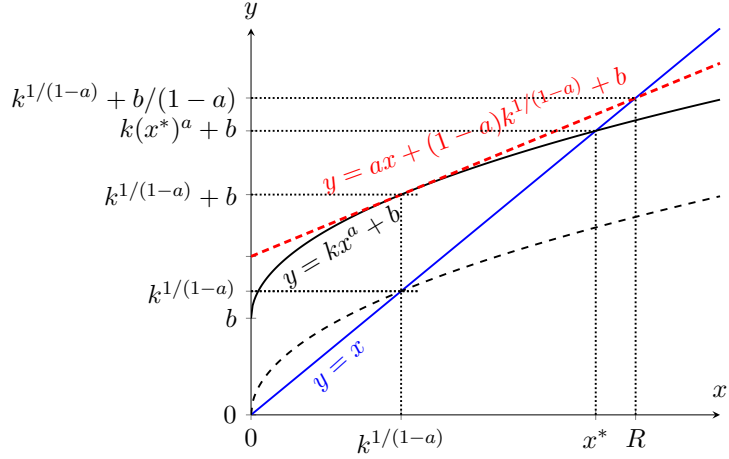
\begin{figure}[htpb]
  \begin{center}
    \begin{tikzpicture}[scale=0.9, transform shape]
      \begin{axis}[
        axis lines = left,
        xtick={0, 0.16, 0.367481, 0.41},
        xticklabels={0, $k^{1/(1-a)}$, $x^*$, $R$},
        ytick={0, 0.125, 0.205, 0.367481, 0.285, 0.41, 0.16},
        yticklabels={0, $b$, , $k (x^*)^a+b$, $k^{1/(1-a)}+b$, $k^{1/(1-a)} + b/(1-a)$, $k^{1/(1-a)}$},
        xlabel=$x$,
        ylabel=$y$,
        x label style={at={(axis description cs:1.0,+0.1)},anchor=north},
        y label style={at={(axis description cs:+0.0,1.0)},rotate=-90,anchor=south},
        ]
        \addplot[blue, domain=0:0.5, samples=100, thick]{x} node [near start, rotate = 42, below, xshift = -2em] {$y=x$};
        \addplot[black,domain=0:0.5, samples=300, thick]{ 0.125 + 0.4 * sqrt(x) } node [near start, below, rotate = 31, xshift = +0em] {$y=k x^a + b$};
        \addplot[black,domain=0:0.5, samples=300, thick, dashed]{ 0.4 * sqrt(x) } node [near start, below, rotate = 31, xshift = +0em] {};
        \addplot[red,domain=0:0.5, samples=100, very thick, densely dashed]{ 0.5 * x + 0.205}node [midway, above, rotate = 20, yshift = 0.3em]
          {$y=a x + (1-a) k^{1/(1-a)} + b$};

        \def\vertline{0.16}
        \def\horizonline{0.285}
        \addplot +[mark=none, black, densely dotted, thick] coordinates {(\vertline, 0) (\vertline, \horizonline)};
        \addplot +[mark=none, black, densely dotted, thick] coordinates {(0, \horizonline) (\vertline+0.02, \horizonline)};

        \def\vertline{0.16}
        \def\horizonline{0.16}
        \addplot +[mark=none, black, densely dotted, thick] coordinates {(0, \horizonline) (\vertline+0.02, \horizonline)};
        \addplot +[mark=none, black, densely dotted, thick] coordinates {(0, \horizonline) (\vertline+0.02, \horizonline)};

        \def\vertline{0.41}
        \def\horizonline{0.41}
        \addplot +[mark=none, black, densely dotted, thick] coordinates {(\vertline, 0) (\vertline, \horizonline)};
        \addplot +[mark=none, black, densely dotted, thick] coordinates {(0, \horizonline) (\vertline, \horizonline)};
        \def\vertline{0.367481}
        \def\horizonline{0.367481}
        \addplot +[mark=none, thick, black, densely dotted] coordinates {(\vertline, 0) (\vertline, \horizonline)};
        \addplot +[mark=none, thick, black, densely dotted] coordinates {(0, \horizonline) (\vertline, \horizonline)};

    \end{axis}
    \end{tikzpicture}
  \end{center}

  \caption{Applying Newton's method for one step starting from the point
  $(k^{1/(1-a)},k^{1/(1-a)}+b)$ on the graph to estimate (find an
upper bound of) the solution $x^*$ to the equation $x = k x^a +b$.}

  \label{F:Motive}
\end{figure}

The procedure outlined in this paper involves establishing
inequality~\eqref{E:nottoy} first and then deriving its upper bounds using
Lemma~\ref{L:Gmm}. This can be seen as a generalization of the one-variable
Bihari--LaSalle inequality~\cite{bihari:56:generalization,
lasalle:49:uniqueness} to multivariate or field cases, as evidenced
by~\eqref{E:utx}. However, there are some subtle differences between the two,
such as the convolutional form of the time variable in our setting, as seen
in~\eqref{E:utx}. While the Bihari--LaSalle inequality has been applied to
moment estimates for nonlinear stochastic differential equations (SDEs), such as
in~\cite{fang.zhang:05:study}, its application to nonlinear SPDEs requires
careful handling of the multivariate setting. One common approach is taking the
supreme norm on the spatial argument and obtaining a one-variable integral
inequality. However, this method is not always applicable, especially for
unbounded initial conditions or when a better understanding of how the initial
condition enters the iterations of the multiplicative type noise is needed. The
challenge of applying the Bihari--LaSalle inequality to nonlinear SPDEs
motivates us to formulate inequality~\eqref{E:nottoy} and establish
Lemma~\ref{L:Gmm}. These two steps constitute some key components of the proof
for Theorem~\ref{T:mup}.

\section{Proof of the main results}\label{S:main}
\subsection{Technical lemmas}\label{SS:lms}
In this part, we provide some technical lemmas that will be used in the
proof of Theorem~\ref{T:mup}. The proofs of these lemmas are postponed to
Section~\ref{SS:lmmpf}.
\begin{lemma}\label{L:rhononde}
  Let $\rho$ be a function satisfying Hypothesis~\ref{H:rho}. Then there exist
  two functions $g^+: [M_0,\infty)\to\R$ and $g^-: (-\infty,-M_0]\to\R$ that
  satisfy the following properties:
  \begin{enumerate}[(i)]
    \item $g^+$ is nonnegative, non-increasing, right-continuous, and it
      satisfies that
      \begin{align*}
        \lim_{x\to \infty} g^+ (x) = 0 \quad \text{and} \quad
        | \rho(x) | - | \rho(M_0) | = \int_{M_0}^x g^+(y) \ud y \quad \text{for all $x\ge M_0$};
      \end{align*}
    \item Similarly, $g^-$ is nonnegative, non-decreasing, left-continuous, and
      it satisfies that
      \begin{align*}
        \lim_{x\to -\infty} g^- (x) = 0 \quad \text{and} \quad
        | \rho(x) | - | \rho(-M_0) | = \int_{x}^{-M_0} g^-(y) \ud y \quad \text{for all $x\le - M_0$}.
      \end{align*}
  \end{enumerate}
\end{lemma}

Let $\theta_p:\R\to\R_+$ denote the power function $\theta_p(x) = |x|^p$ for
$p\in\R$. When $\rho(u) = |u|^\alpha$ with $u\ge 0$ and $\alpha\in (0,1]$, we
claim that
\begin{align*}
  \Norm{\rho(u)}_p^2 \le \rho\left(\Norm{u}_p^2\right) \quad \text{for all $p>0$}.
\end{align*}
Indeed, it is clear that
\begin{align*}
  \Norm{\rho(u)}_p^2
  & = \left(\theta_{\frac{2}{p}} \circ \E \circ \theta_p \circ |\rho|\right) (u)
    = \left(\theta_{\frac{2}{p}} \circ \E \circ \theta_p \circ \theta_{\alpha}\right) (u)
    = \left(\theta_{\frac{2}{p}} \circ \E \circ \theta_{\alpha} \circ \theta_{p}\right) (u)\\
  &\le\left(\theta_{\frac{2}{p}} \circ \theta_{\alpha} \circ \E \circ \theta_{p}\right) (u)
    = \left(\theta_{\alpha} \circ \theta_{\frac{2}{p}} \circ \E \circ \theta_{p}\right) (u)
    = \rho\left(\Norm{u}_p^2\right),
\end{align*}
where the inequality is due to the concavity of the function $\rho =
\theta_\alpha$ and we used twice the commutative property:
\begin{align*}
    \rho\circ\theta_p
  = \theta_\alpha \circ \theta_p
  = \theta_p \circ \theta_\alpha
  = \theta_p \circ \rho.
\end{align*}
However, for a general $\rho$ that satisfies Hypothesis~\ref{H:rho}, we need to
introduce $\rho_p$ for this purpose.

\begin{definition}\label{D:rhop}
  For any function $\rho:\R \to\R$ and any positive number $p$, let the
  functions $\rho_p^+$, $\rho_p^-$ and $\rho_p:\R_+\to\R_+$ be defined as
  follows: for all $x\in \R_+$,
  \begin{align}\label{E:rhop}
    \rho_p (x)   \coloneqq \rho_p^+ (x) + \rho_p^- (x) \quad\text{with}\quad
    \rho_p^\pm (x) \coloneqq \left|\rho \left(\pm\: x^{1/p}\right)\right|^p.
  \end{align}
  Or in other words, $\rho_p^{\pm}(\cdot)$ with $p>0$ are defined so that
  \begin{align*} 
    \left|\rho \left(\sign(x) |x|^{1/p}\right)\right|^p
    = &  \underbrace{\left(\theta_p \circ |\rho| \circ \theta_{1/p}\right) (x)}_{\displaystyle = \rho_p^+(x)}  1_{\{x \ge 0\}}
        +\underbrace{\left(\theta_p \circ |\rho| \circ r \circ \theta_{1/p} \right) (x)}_{\displaystyle = \rho_p^-(-x)} 1_{\{x  < 0\}} \le \rho_p(|x|),
  \end{align*}
  where $r(x) \coloneqq -x$ is the reflection function.
\end{definition}

The next lemma shows that $\rho_p(\cdot)$ and $\rho_p^\pm(\cdot)$ inherit the
properties from $\rho(\cdot)$.

\begin{lemma}\label{L:ccvtrho}
  Suppose that $\rho(\cdot)$ is a function satisfying Hypothesis~\ref{H:rho}.
  For any $p>0$, let $\rho_p(\cdot)$, $\rho_p^+(\cdot)$ and $\rho_p^-(\cdot)$ be
  given in Definition~\ref{D:rhop}. Then, the following properties hold:
  \begin{enumerate}[(i)]
    \item $\rho_p^+$ and $\rho_p^-$ admit the following representations: for any
      $x \ge M_0^p$,
      \begin{align}\label{E:gp+}
        \rho_p^+ (x) - \rho_p^+ (M_0^p) = \int_{M_0^p}^x g^+_p(y) \ud y \quad\text{with~} g_p^+ (x) \coloneqq \left\lvert\rho \left(+x^{1/p}\right)\right\rvert^{p-1} g^+ \left(+x^{1/p}\right) x^{-(p-1)/p} \shortintertext{and}\label{E:gp-}
        \rho_p^- (x) - \rho_p^- (M_0^p) = \int_{M_0^p}^x g^-_p(y) \ud y \quad\text{with~} g_p^- (x) \coloneqq \left\lvert\rho \left(-x^{1/p}\right)\right\rvert^{p-1} g^- \left(-x^{1/p}\right) x^{-(p-1)/p},
      \end{align}
      respectively, where $g^\pm(\cdot)$ are given in Lemma~\ref{L:rhononde};
    \item Both $\rho_p^+$ and $\rho_p^-$ are non-decreasing on $[M_0^p,\infty)$;
    \item There exists $M \geq M_0$, independent of $p$, such that all functions
      $\rho_p$, $\rho_p^+$, and $\rho_p^-$ are concave on $[M^p,\infty)$;
    \item If $M_0 = 0$ in part~(iii) of Hypothesis~\ref{H:rho}, then all
      functions $\rho_p^+$, $\rho_p^-$ and $\rho_p$ are concave on $\R_+$.
  \end{enumerate}
\end{lemma}

\begin{lemma}\label{L:Concave}
  For any $U\in L^p(\Omega)$ and $p>0$, it holds that
  \begin{align}\label{E:up3}
    \E \left[\rho_p^+ (|U|^p) \one_{[+M,+\infty)}(U)\right]
    \leq & \rho_p^+ \left(M^p + \Norm{U}_p^p \right) \shortintertext{and}
    \E \left[\rho_p^- (|U|^p) \one_{(-\infty, -M]}(U)\right]
    \leq & \rho_p^- \left(M^p + \Norm{U}_p^p \right),
    \label{E:up4}
  \end{align}
  where the constant $M$ is given in part (iii) of Lemma~\ref{L:ccvtrho}.
\end{lemma}

\begin{lemma}\label{L:rho-p}
  Suppose that $\rho$ satisfies Hypotheses~\ref{H:rho}. Let $M$ be the
  associated constant given in part (iii) of Lemma~\ref{L:ccvtrho}. Then for all
  $p\ge 2$ and any $U\in L^p(\Omega)$, it holds that
  \begin{align}\label{E:rho2-p}
    \Norm{\rho(U)}_p^2 \le K_M^2 + \rho_2\left(M^2 + \Norm{U}_p^2\right),
  \end{align}
  where $K_M$ is defined in~\eqref{E:KM}. In particular, if $M_0=0$ in
  Hypothesis~\ref{H:rho}, then one can take $M = K_M = 0$ in~\eqref{E:rho2-p}.
\end{lemma}

\begin{definition}\label{D:F}
  Suppose that $\rho$ satisfies Hypothesis~\ref{H:rho}. Let $\rho_2(\cdot)$ be
  the function given by~\eqref{E:rhop}. Define $F:[M^2,\infty) \to \R_+\cup
  \{\infty\}$ as
  \begin{gather}\label{E:F}
    F(x) \coloneqq \frac{x}{4\rho_2(x)}, \quad x\ge M^2,
     \shortintertext{with}
   \frac{x}{0} \coloneqq \infty \quad \text{if $x>0$;} \quad \text{and}  \quad F(0) \coloneqq \lim_{x\downarrow 0} \frac{x}{4\rho_2(x)}\in \R_+\cup\{\infty\},
    \quad \text{if $M=0$ and $\rho_2(0)=0$},\nonumber
  \end{gather}
  where $M > M_0$ is the same as in part (iii) of Lemma~\ref{L:ccvtrho}. We use
  $F^{-1}$ to denote the (right) inverse of $F$ restricted on $[2M^2,\infty)$ as
  follows,
  \begin{align}\label{E:F^-1}
    F^{-1}(x) \coloneqq \inf \left\{y\in [2M^2,+\infty) \colon F(y) \geq x \right\}.
  \end{align}
\end{definition}
\begin{remark}
  Here are some remarks on the functions $F$ and $F^{-1}$:
  \begin{enumerate}[(i)]
   \item As stated in Definition~\ref{D:F}, we allow $F(x) = \infty$ in case
     $x>0$ and $\rho_2(x) = 0$, see e.g., $\rho(x) = 1$ for $|x|<1$ and $\rho(x)
     = |x-1|^{\alpha}$ for $|x|\geq 1$. If $\rho\equiv 0$ on $[M_0,
     \infty)$, then $F \equiv \infty$ on $[M_0,\infty)$ as well. This implies that
     $F^{-1} \equiv 2M^2$ on $\R_+$. As a result, $F^{-1}(2 p h(t))$ as
     in~\eqref{E:mup} is uniformly bounded in $t$, this coincides with the SHE
     with additive noise.

    \item Under part (ii) of Hypothesis~\ref{H:rho}, the set in~\eqref{E:F^-1}
      is nonempty for any $x > 0$, and thus $F^{-1} (\cdot)$ is a real-valued
      function.

    \item From the definitions, and noticing that $F$ is continuous on
      $[M^2,\infty)$, it is easy to see that
      \begin{align}\label{E:FF-1}
        & F\circ F^{-1}(x) \geq x,   \quad x \geq 0, \shortintertext{and}
        & F^{-1} \circ F (x) \leq x, \quad x\geq 2M^2. \label{E:F-1F}
      \end{align}
  \end{enumerate}
\end{remark}

\begin{lemma}\label{L:Gmm}
  Suppose that the function $\rho(\cdot)$ satisfies Hypothesis~\ref{H:rho}. Let
  $\rho_2(\cdot)$ be defined as in~\eqref{E:rhop} with $p = 2$. Then, thanks to
  part (iii) of Lemma~\ref{L:ccvtrho}, $\rho_2$ is concave on $[M^p,\infty)$
  with some $M > M_0$. For any $k,\: b>0$, suppose $x\ge 0$ such that $x \le k
  \rho_2(x) + b$. Then,
  \begin{align}\label{E:BIalt}
    x\le 2 F^{-1} (k) + 2b < \infty.
  \end{align}
\end{lemma}

\begin{lemma}\label{L:J1d=1}
  Under Hypothesis~\ref{H:corre} and part~(i) of Hypothesis~\ref{H:rough}, for
  all $(t,x) \in \R_+\times \R$, $\calJ_1(t,x)$ defined in~\eqref{E:J1}
  satisfies that
  \begin{gather}\label{E:J1J0}
		\calJ_1(t,x) \leq \int_0^t \ud s\: k(t-s) \int_{\R^{d}} \ud y\: p_{t-s} (x-y) \calJ_0^2(s,y), \shortintertext{where}
    k(t)
    \coloneqq \int_{\R^d} \ud z \: p_{t}(z) f(z)
    = h' \Big(\frac{t}{2}\Big)
    < \infty.
    \label{E:k}
  \end{gather}
  In particular, when $d = 1$, it holds that
	\begin{align}\label{E:J1d=1}
    \calJ_1(t,x)
    \le 2^{3/2}\pi\: h(t)\:\calJ_+ \left(t/2, x\right)^2
    <\infty.
	\end{align}
\end{lemma}

\subsection{Proof of moment growth formulas---Theorem~\ref{T:mup}}\label{SS:prof_mup}
The proof of Theorem~\ref{T:mup} consists of four steps: \bigskip

\noindent\textbf{Step 1.~} In this step, we use the assumption of $\rho$ given
in Hypothesis~\ref{H:rho} to obtain the following nonlinear integral
inequality $\Norm{u(t,x)}_p$:
\begin{align}\label{E:utx}
  \begin{aligned}
    \Norm{ u(t,x)}_p^2
    \leq   2 \calJ_0^2(t,x) + 8 K_M^2 p\, h(t) + 8p
         & \int_0^t \ud s \iint_{\R^{2d}} \ud y \ud y' p_{t-s}(x-y) p_{t-s} (x-y') \\
         & \times f(y-y') \rho_2\left(M^2 + \Norm{u(s,y)}_p^2 \right).
  \end{aligned}
\end{align}

\noindent Indeed, by the Burkholder-Davis-Gundy and Minkowski's inequalities to
the mild form~\eqref{E:Mild},
\begin{align*}
  \MoveEqLeft \Norm{u(t,x)}_p^2 = \left(\E\left[\left(
    \calJ_0 (t,x) + \int_0^t\int_{\R^d} p_{t-s}(x-y) \rho(u(s,y)) W(\ud s,\ud y)
    \right)^p \right]\right)^{2/p} \\
  \leq & 2 \calJ_0^2(t,x)
         + 8p \Norm{\int_0^t \ud s \iint_{\R^{2d}} \ud y\ud y'\:
         p_{t-s}(x-y) p_{t-s} (x-y') f(y-y') \rho(u(s,y)) \rho (u(s,y'))}_{p/2}\\
  \leq & 2 \calJ_0^2(t,x) + 8p \int_0^t \ud s \iint_{\R^{2d}} \ud y\ud y'\:
         p_{t-s}(x-y) p_{t-s}(x-y') f(y-y') \Norm{\rho(u(s,y))\rho(u(s,y'))}_{p/2}.
\end{align*}
Then an application of the Cauchy-Schwarz inequality for the
$L^{p/2}(\Omega)$--norm yields that
\begin{align}\label{E:BDG}
  \begin{aligned}
    \Norm{u(t,x)}_p^2 \leq 2 \calJ_0^2(t,x) + 8p \int_0^t \ud s \iint_{\R^{2d}} \ud y\ud y'\: f(y-y')
    \:     p_{t-s}(x-y)  \Norm{\rho(u(s,y)) }_p &  \\
    \times p_{t-s}(x-y') \Norm{\rho(u(s,y'))}_p & \,.
  \end{aligned}
\end{align}

Taking account of the fact that $ab \le \frac{1}{2}(a^2+b^2)$ for all $a,b\in
\R$, we can further deduce that
\begin{align*}
  \Norm{u(t,x)}_p^2 \leq & 2 \calJ_0^2(t,x) + 8p \int_0^t \ud s \iint_{\R^{2d}} \ud y\ud y'\:
         p_{t-s}(x-y) p_{t-s}(x-y') f(y-y') \Norm{\rho(u(s,y))}_p^2.
\end{align*}
Next, applying Lemma~\ref{L:rho-p}, we get
\begin{align}\label{E:rhou}
  \Norm{\rho(u(s,y))}_p^2
  \le K_M^2 + \rho_2\left(M^2 + \Norm{u(s,y)}_p^2\right).
\end{align}
Plugging~\eqref{E:rhou} into the previous inequality proves~\eqref{E:utx}.

\bigskip\noindent\textbf{Step 2.~} In this step, we will solve the nonlinear
integral inequality~\eqref{E:utx}. Fix $t>0$ and $x\in\R^d$. Using the
function $h(t)$ in~\eqref{E:h} as a normalization constant and thanks to the
concavity of $\rho_2$, we can apply Jensen's inequality to the triple
integrals in~\eqref{E:utx} to write that
\begin{align}\label{E:utx2}
  \Norm{u(t,x)}_p^2
  \leq   2 \calJ_0^2(t,x) + 8 K_M^2 p\, h(t) + 8p h(t) \rho_2\left(X\right),
\end{align}
where
\begin{align*}
  X \coloneqq M^2 + h(t)^{-1} Y \quad \text{with} \quad
  Y \coloneqq \int_0^t \ud s \iint_{\R^{2d}} \ud y\ud y'\: p_{t-s}(x-y) p_{t-s} (x-y') f(y-y') \Norm{u(s,y)}_p^2.
\end{align*}
It reduces to find an upper bound for $X$. By using~\eqref{E:utx}, we deduce that
\begin{gather*}
  Y \leq 2 \calJ_1(t,x)	+ 8 K_M^2 p\, h(t)^2 + \calI
  \shortintertext{where} \notag
  \begin{aligned}
  \calI \coloneqq 8p & \int_0^t \ud r \iint_{\R^{2d}} \ud z \ud z' f(z-z') \rho_2 \big(M^2 + \Norm{u(r,z)}_p^2 \big) \\
       \times        & \int_r^t \ud s \iint_{\R^{2d}}dy \ud y'\:
                        p_{t-s}(x-y) p_{t-s}(x-y') p_{s-r}(y-z)
                        p_{s-r} (y'-z') f(y-y').
  \end{aligned}
\end{gather*}
Using the following formula,
\begin{align}\label{E:Factor}
    p_{t-s}(a)p_s(b)
  = p_{s(t-s)/t}\left(b-\frac{s}{t}\left(a+b\right)\right) p_t(a+b),
  \quad \text{for all $0<s<t$ and $a,b\in\R$,}
\end{align}
we see that
\begin{align}\label{E:2hsck}
 \MoveEqLeft \int_{r}^t \ud s \iint_{\R^{2d}} \ud y\ud y' p_{t-s} (x-y) p_{t-s} (x-y') p_{s-r}(y-z) p_{s-r}(y'-z') f(y-y')	\notag \\
  =    & p_{t-r} (x-z) p_{t-r} (x-z')\int_{r}^t \ud s \iint_{\R^{2d}} \ud y\ud y' f(y-y') \notag \\
       & \times p_{\frac{(s-r)(t-s)}{t-r}}\left(y-z - \frac{s-r}{t-r}(x-z) \right) p_{\frac{(s-r)(t-s)}{t-r}} \left(y'-z' - \frac{s-r}{t-r}(x-z') \right) \notag \\
  \leq & \left(2\pi\right)^{-d} p_{t-r} (x-z) p_{t-r} (x-z') h(t),
\end{align}
where the last inequality follows from~\cite[Lemma 2.6 and  inequalities
(2.21)--(2.23)]{chen.kim:19:nonlinear}. Thus,
\begin{align*}
  \calI \leq & 8 \left(2\pi\right)^{-d} p\, h(t)
  \int_0^t \ud r \iint_{\R^{2d}} \ud z \ud z'	\: p_{t-r} (x-z) p_{t-r} (x-z') f(z-z') \rho_2 \left(M^2 + \Norm{u(r,z)}_p^2\right) \\
        \leq    & 8 \left(2\pi\right)^{-d} p h^2(t) \rho_2(X),
\end{align*}
due to the concavity of $\rho_2(\cdot)$ (see
Lemma~\ref{L:ccvtrho}),~\eqref{E:h}, and Jensen's inequality. Therefore,
\begin{align*}
  Y \leq 2 \calJ_1(t,x)	+ 8 K_M^2 p\, h(t)^2 + 8\left(2\pi\right)^{-d} p h^2(t) \rho_2(X),
\end{align*}
or equivalently,
\begin{align}\label{E:nottoy}
  X \leq
  \underbrace{M^2 + 2h(t)^{-1} \calJ_1(t,x) + 8K_M^2 p\, h(t)}_{ \displaystyle \coloneqq b(p,t,x)} +
  \underbrace{8\left(2\pi\right)^{-d} p\, h(t)}_{ \displaystyle \coloneqq k(p,t)}\rho_2(X).
\end{align}
By Lemma~\ref{L:Gmm}, we have that
\begin{align*}
  X \le 2 F^{-1}(k(p,t)) + 2 b(p,t,x).
\end{align*}
Finally, thanks to the monotonicity of $\rho_2(\cdot)$ when $x\ge M^2$; see
Lemma~\ref{L:ccvtrho}, plugging the above moment bounds back to~\eqref{E:utx2}
proves the moment following bounds
\begin{align}\label{E_:mup}
  \begin{aligned}
    \Norm{u(t,x)}_p^2 \leq
    & 2\: \calJ_0^2(t,x) + 8 K_M^2 p \: h(t)  + 8 p\: h(t) \:\rho_2 \left(2b(p,t,x) + 2F^{-1}\big( k(p,t)\big) \right).
  \end{aligned}
\end{align}

\bigskip\noindent\textbf{Step 3.~} In this step, we will simplify the expression
in~\eqref{E_:mup} and prove inequality~\eqref{E:mup}. Recall the definition of
$F^{-1}$ in~\eqref{E:F^-1}, one can show that for
any $x \geq 2F^{-1}(k)$,
\begin{gather*}
  \frac{\rho_2(x)}{x} = \frac{\rho_2 (x) - \rho_2 (2 F^{-1}(k)) + \rho_2 (2F^{-1}(k))}{ x - 2F^{-1} (k) + 2F^{-1}(k)} \leq \frac{1}{2k},
  \shortintertext{because}
  \frac{\rho_2 \left(2F^{-1}(k)\right)}{2F^{-1}(k)}\leq \frac{1}{4k} \leq \frac{1}{2k} \quad \text{and}\quad \frac{\rho_2 (x) - \rho_2 \left(2F^{-1}(k)\right) }{ x - 2F^{-1} (k)} \leq g_2\left(2F^{-1}(k)\right) \leq \frac{1}{2k},
\end{gather*}
where the last inequality is proved in~\eqref{E:g2} below.
As a result, concerning the fact that $F^{-1} (k)\geq 2 M^2$ for
all $k>0$, and $8(2\pi)^{-d} \leq 2$ for all $d\geq 1$, we can write
\begin{align*}
  8p\,h(t) \:\rho_2 \left(2b(p,t,x)  + 2F^{-1}\big( k(p,t)\big) \right)
  \leq & \frac{4 p\,h(t)}{k(p,t)} \left(b(p,t,x) + F^{-1}\big( k(p,t)\big) \right) \\
  \leq & 2 (2\pi)^{d}\left( h(t)^{-1} \calJ_1(t,x) + 4 K_M^2 p\: h(t) + F^{-1}( 2 p\: h(t)) \right).
\end{align*}
Plugging the above upper bound back to~\eqref{E_:mup} proves~\eqref{E:mup}.

\bigskip\noindent\textbf{Step 4.~} The special case when $d=1$ is an application
of Lemma~\ref{L:J1d=1} and the case when $\rho_2(\cdot)$ is concave separately
on $\R_+$ and $\R_-$ is due to Lemma~\ref{L:rho-p}. This proves both parts~(i)
and~(ii). As for part~(iii), if $M = 0$, inequality~\eqref{E:mup_asy} follows
from part~(ii). Otherwise, if $M>0$, then~\eqref{E:mup_asy} holds provided that
one can verify that there exists $C>0$ such that for all $t \geq 1$,
\begin{align}\label{E:rhoasc}
  4 K_M^2 p \: h(t) +  F^{-1}\left( 2 p\: h(t)\right) \leq C F^{-1}(C ph(t)).
\end{align}
Indeed, the assumption of part~(iii) ensures that $\rho_2(x) \geq c > 0$ with
some uniform constant $c>0$ if $x > M$ is large enough, which implies that
$F^{-1}(k) \geq c k$ for large $k$ (see~\eqref{E:F}). Hence,
inequality~\eqref{E:rhoasc} holds by noticing that $h$ is a non-decreasing
function such that $h(t) > 0$ for all $t > 0$ under Hypothesis~\ref{H:corre}.
This proves~\eqref{E:mup_asy}.

Finally, if $\mu(\cdot)$ is bounded, then $\calJ_0(t,x)$ is bounded on
$\R_+\times\R^d$, and the same is true for $h(t)^{-1}\calJ_1(t,x)$ (see
Lemma~\ref{L:J1d=1}). Then, inequality~\eqref{E:mup_smp} is a consequence of the
fact that $F^{-1}(k)$ is bounded below by a positive constant for all $k$ large
enough because $\rho$ is not identically zero on $(-\infty, -M_0] \cup
[M_0,\infty)$. This completes the whole proof of Theorem~\ref{T:mup}. \myEnd

\subsection{Proof of tail probability---Theorem~\ref{T:Tail}}\label{SS:Tail}

We first prove a lemma:

\begin{lemma}\label{L:tail}
  Let $X$ be a random variable such that for some
  function $\alpha \colon [2,\infty) \to \R$,
  \begin{align*}
    \mathbb{E} \left[|X|^p\right] \leq \exp (\alpha (p)) < \infty,
    \quad \text{for all $p \geq 2$.}
  \end{align*}
  Then, for all $z > 0$, it holds that
  \begin{align}\label{E:lgd}
    \P (|X| \geq z) \leq \exp \big(-\alpha^*(\log(z))\big),
  \end{align}
  where $\alpha^*(\cdot)$ is the Legendre-type transform of $\alpha(\cdot)$ on
  $[2,\infty)$, namely,
  \begin{align}\label{E:Legendre}
    \alpha^* (x) \coloneqq \sup\bigg\{ xp - \alpha(p) \colon p\in [2,\infty) \bigg\} \in \R \cup\{\infty\},\qquad \text{for all $x\in\R$}.
  \end{align}
\end{lemma}
\begin{proof}
  The lemma is demonstrated using Chebyshev's inequality: for any $z > 0$ and
  $p\geq 2$,
  \begin{align*}
       \P \left(|X| \geq z\right)
     = \P \left(|X|^p \geq z^p\right) \leq z^{-p} \exp \left(\alpha(p)\right)
     = \exp \left(- \left[ p\log(z) - \alpha(p)\right ]\right).
  \end{align*}
  Then,~\eqref{E:lgd} follows immediately.
\end{proof}

\begin{proof}[Proof of Theorem~\ref{T:Tail}]
  For all $t \geq 1$ and $p\ge 2$, we apply Lemma~\ref{L:tail} with the moment
  bounds given in~\eqref{E:mup_smp} of Theorem~\ref{T:mup} to see that
  \begin{align}\label{E_:tail1}
    \P(u(t,x) \geq z) \leq \exp \left(- H^* (\log(z)) \right),
  \end{align}
  where $H^*\colon \R \to \R$ is the Legendre-type transform
  (see~\eqref{E:Legendre}) of $H\colon [2,\infty) \to \R$ given by
  \begin{align}\label{E:H}
    H(p) \coloneqq \frac{p}{2}\log \left(C_* F^{-1} (C_* p h(t))\right), \qquad \text{for all } p \geq 2,
  \end{align}
  with the constant $C_*$ given in~\eqref{E:mup_smp}. Next, notice that if
  \begin{align}\label{E_:Lt}
    y \geq \max\left\{ 2^{-1}\log \left(2 C_* M^2\right) + 1, 2^{-1}\log \left(C_* F^{-1} \big(2 C_* h(t) \big)\right) + 1\right\},
  \end{align}
  then, we have $e^{2(y-1)}/C_* \geq 2 M^2$, and (thanks to~\eqref{E:FF-1})
  \begin{align*}
    p^*(y) \coloneqq \left(C_* h(t) \right)^{-1} F \left(e^{2(y-1)} /C_* \right) \geq 2.
  \end{align*}
  Therefore,
  \begin{align*}
    H(p^*(y)) = & \left(2C_* h(t) \right)^{-1} F \left(e^{2(y-1)} /C_* \right) \log \left(C_* F^{-1} \circ F \left(e^{2(y-1)} /C_* \right) \right) \\
    \leq        & \left(2C_* h(t) \right)^{-1} F \left(e^{2(y-1)} /C_* \right) \log \left(e^{2(y-1)} \right) = p^*(y)(y-1),
  \end{align*}
  where we have used~\eqref{E:F-1F} for the inequality. This yields that
  \begin{align}\label{E:H*}
    H^* (y) \geq y p^*(y) - H(p^*(y))
    = y p^*(y) - p^*(y)(y-1)
    =  \left(C_* h(t) \right)^{-1} F \left(e^{2(y-1)} /C_* \right).
  \end{align}
  Therefore,~\eqref{E:tail} is justified by plugging~\eqref{E:H*}
  in~\eqref{E_:tail1} with $y$ replaced by $\log z$. Similarly, the expression
  of $L_t$ in~\eqref{E:Lt} can be obtained by replacing $y$ by $\log z$
  in~\eqref{E_:Lt}.
\end{proof}

\subsection{Proof of H\"older regularity---Theorem~\ref{T:Holder}}\label{SS:Holder}

\begin{proof}[Proof of Theorem~\ref{T:Holder}]
  The proof follows exactly the same arguments as those
  in~\cite[Theorem~1.8]{chen.huang:19:comparison} with the moment bounds
  obtained in Lemma~\ref{L:PAM-BD} below.
\end{proof}

\begin{lemma}\label{L:PAM-BD}
  Assume Hypothesis~\ref{H:corre}, parts (i) and (ii) of Hypothesis~\ref{H:rho},
  and part (i) of Hypothesis~\ref{H:rough}. Then the solution $u(t,x)$ to
  SHE~\eqref{E:SHE} satisfies that
  \begin{align}\label{E:PAM-Mom-UB}
    \Norm{u(t,x)}_p \le  \left(\mu' * p_t\right)(x)\: \left[H\left(t\,;\,32pK^2\right)\right]^{1/2}, \quad
    \text{for all $(t,x,p) \in \R_+\times\R^d \times [2,\infty)$},
  \end{align}
  where $\mu' \coloneqq \sqrt{2} + 2|\mu|$ and $H(t;\lambda)$ is a
  non-decreasing function of $t$ with a nonnegative parameter $\lambda>0$.
\end{lemma}

Referring to the precise definition of $H(t;\lambda)$, one can consult
~\cite[Formula (2.4)]{chen.huang:19:comparison}. It should be noted that
$H(0;\lambda)>0$, which means the function $H(t;\lambda)$ does not introduce any
singularity at $t=0$. In general, $H(t;\lambda)$ exhibits exponential growth
rate in $t$; see~\cite[Lemma~2.5]{chen.kim:19:nonlinear}. Certainly, the moment
bounds~\eqref{E:PAM-Mom-UB} are considerably less accurate compared to those as
in~\eqref{E:mup}, especially for large $p$ or $t$. This is a worthy trade-off
for removing part (ii) of Hypothesis~\ref{H:rough} in Lemma~\ref{L:PAM-BD},
which is sufficient for deducing the H\"{o}lder continuity of the solutions in
Theorem~\ref{T:Holder}. Notably, to achieve more precise tail estimates in
Theorem~\ref{T:Tail}, additional improved estimates for the moment increments,
as illustrated in Lemma~\ref{L:holder} below, are required.

\begin{proof}[Proof of Lemma~\ref{L:PAM-BD}]
  Parts (i) and (ii) of Hypothesis~\ref{H:rho}, namely, $\rho$ is locally
  bounded and asymptotically sublinear $ \displaystyle \lim_{x \to
  \pm\infty}x^{-1}\rho(x) = 0$, imply that $|\rho(x)| \leq
  K_{\rho} (1 + |x|)$ with some universal constant $K_\rho > 0$ for
  all $x\in \R$. Hence, from~\eqref{E:BDG}, we see that for all $(t,x,p)\in
  \R_+\times\R^d\times[2,\infty)$,
  \begin{align*}
    \Norm{u(t,x)}_p^2 \leq 2 \calJ_0^2(t,x) + 8p K_\rho^2 \int_0^t \ud s \iint_{\R^{2d}} \ud y\ud y'\: f(y-y')
    \:     p_{t-s}(x-y)  \left(1+\Norm{u(s,y) }_p\right) &  \\
    \times p_{t-s}(x-y') \left(1+\Norm{u(s,y')}_p\right) & \,.
  \end{align*}
  By setting $g(t,x)\coloneqq 1 + \Norm{u(t,x)}_p$ and denoting the above triple
  integral by $I(t,x)$, the above inequality implies that
  \begin{align*}
    g(t,x)^2
    \le 2 + 2 \Norm{u(t,x)}_p^2
    \le \left(\sqrt{2} + 2 \calJ_+(t,x)\right)^2 + 16 p K_\rho^2 I(t,x).
  \end{align*}
  Therefore, for $\mu' = \sqrt{2} K_\rho + |\mu|$, $g(t,x)$ satisfies the following
  integral inequality:
  \begin{align*}
    g(t,x)^2 \le \left[\left(\mu'*p_t\right)(x)\right]^2 + 16pK_\rho^2 \int_0^t \ud s \iint_{\R^{2d}} \ud y\ud y'\: f(y-y')
    \:     p_{t-s}(x-y)  g(s,y)  &  \\
    \times p_{t-s}(x-y') g(s,y') & \,.
  \end{align*}
  It is easy to see that $\mu'$ also satisifes part (i) of
  Hypothesis~\ref{H:rough}. Then, an application
  of~\cite[Lemma~2.2]{chen.huang:19:comparison} with the above $\mu'$ and
  $\lambda = 16 p K_\rho^2$ implies the moment bound~\eqref{E:PAM-Mom-UB}. The
  property of $H(t;\lambda)$ can be found
  in~\cite[Lemma~2.5]{chen.kim:19:nonlinear}.
\end{proof}

\subsection{Proof of spatial asymptotics---Theorem~\ref{T:Asymspc}}\label{SS:asymspc}

The proof of Theorem~\ref{T:Asymspc} is based on a tail estimate for the
solution to SHE~\eqref{E:SHE} given in Theorem~\ref{T:Tail}. We also need moment
increments in the space variable, in case of the constant initial condition,
with a sharper constant than the one implicitly appearing in Theorem~\ref{T:Holder}.

\begin{lemma}\label{L:holder}
  Assume Hypotheses~\ref{H:rho} and~\ref{H:efcdalang}. Let $u(t,x)$ be a
  solution to SHE~\eqref{E:SHE} with initial condition $u_0 \equiv 1$. Let $C_*
  > 0$ be the constant in~\eqref{E:mup_smp}. Then, the following statements are
  satisfied.
  \begin{enumerate}[(i)]
    \item For all $x, y \in \R^d$, $t \geq 1$, and $p \geq 2$, there exists a
      constant $C > 0$ such that
      \begin{align*}
        \Norm{u(t,x) - u(t,y)}_p^2 \leq C F^{-1}\left(C_* p h(t)\right) |x-y|^{2\eta};
      \end{align*}

    \item For any $x\in\R^d$, with $B_x$ denoting the unit ball centered at $x$,
      it holds that
      \begin{align*}
        \Norm{\sup_{y_1,y_2 \in B_x} \left|u(t,y_1) - u(t,y_2)\right| }_p^2  \leq C' F^{-1}\left(C_* p h(t)\right),
      \end{align*}
      where $C' > 0$ is another universal constant.
  \end{enumerate}
\end{lemma}
\begin{proof}
  Following the same lines as in the proof of~\cite[Theorem
  1.8]{chen.huang:19:comparison}, one can show that
  \begin{align*}
     \Norm{u(t,x) - u(t,y)}_p^2
     \leq & 4p \int_0^t \iint_{\R^{2d}} |p_{t-s} (x - z) - p_{t-s} (y - z)| |p_{t-s} (x - z') - p_{t-s} (y - z')| \\
          & \times  \Norm{\rho(u(s,z))}_p \Norm{\rho(u(s,z))}_p f(z - z').
  \end{align*}
  Thanks to~\eqref{E:rho2-p} and part~(iii) of Theorem~\ref{T:mup}, one can
  write
  \begin{align*}
    \sup_{s\in [0,t]}\sup_{z\in \R^d} \Norm{\rho(u(s,z))}_p^2
    \leq & K_M^2 + \rho_2 \left(M^2 + \sup_{s\in [0,t]}\sup_{x\in \R^d}\Norm{u(s,x)}_p^2\right) \\
    \leq & K_M^2 + \rho_2 \left(M^2 + C_* F^{-1} (C_* p h(t))\right),
  \end{align*}
  for all $t \ge 1$, where $C_*>0$ is the same as in~\eqref{E:mup_smp}. If
  $F^{-1} (C_* p h(t)) < 2 M^2$, one can apply part (ii) of
  Lemma~\ref{L:ccvtrho} to see that
  \begin{align*}
    I \coloneqq \rho_2 \left(M^2 + C_* F^{-1} (C_* p h(t))\right)
      \leq      \rho_2 \left( (1+2C_*) M^2 \right).
  \end{align*}
  Otherwise, with $g_2 (\cdot) \coloneqq g_2^+ (\cdot) + g_2^-(\cdot)$ defined
  as in~\eqref{E:gp+} and~\eqref{E:gp-}, respectively, we have,
  \begin{align*}
    I & =  \rho_2(M^2) + \int_{M^2}^{M^2 + C_* F^{-1} (C_* p h(t))}  g_2(x) \ud x \\
      & \leq  \rho_2(M^2) + \frac{C_* F^{-1}(C_* p h(t))}{ F^{-1}(C_* p h(t))  - M^2} \int_{M^2}^{ F^{-1} (C_* p h(t))}  g_2(x) \ud x
        \leq  (2C_* + 1) \rho_2 \left(F^{-1} (C_* p h(t))\right).
  \end{align*}
  Combining the above two cases, we have that
  \begin{align*}
    I \leq & K_1 + C_1 \rho_2 \left(F^{-1} (C_* p h(t))\right)
      =      K_1 + C_1 \frac{F^{-1} (C_* p h(t))}{F \circ F^{-1} (C p h(t))} \nonumber\\
      \leq & K_1 + C_1 \frac{F^{-1} (C_* p h(t))}{C_* p h(t)}.
  \end{align*}
  with some universal positive constants $C_1$ and $K_1$. Additionally,
  following the same idea as in the proof of~\eqref{E:rhoasc}, and noting
  $h(\cdot)$ is a non-decreasing function with $h(t)>0$ for $t>0$, we can
  further simplify above inequality as follows:
  \begin{align*}
    I \leq C_2 p^{-1} F^{-1} \left(C_* ph(t)\right),
  \end{align*}
  where $C_2>0$ is a universal constant. Hence,
  \begin{align*}
    \Norm{u(t,x) - u(t,y)}_p^2 \leq  4C_2 F^{-1} (C_* p h(t)) \int_0^t \ud s \iint_{\R^{2d}} \ud z \ud z'\: f(z - z')
            \left|p_{t-s} (x - z ) - p_{t-s} (y - z )\right| &  \\
    \times  \left|p_{t-s} (x - z') - p_{t-s} (y - z')\right| & .
  \end{align*}
  The rest proof of part (i) is the same as Step 1 in the proof of~\cite[Theorem
  1.8]{chen.huang:19:comparison}. The proof of part (ii) also follows from a
  classical argument in the proof of Kolmogorov's continuity criterion. Thus, we
  omit it and refer interested readers to, e.g.,~\cite[Theorem 4.3 on page
  10]{dalang.khoshnevisan.ea:09:minicourse} for more references.
\end{proof}

Now we are ready to present the proof of Theorem~\ref{T:Asymspc}.

\begin{proof}[Proof of Theorem~\ref{T:Asymspc}]
  Let $C_1$ and $C_2 > 0$ be two generic constants that will be fixed later.
  Denote
  \begin{align*}
    Q(R) \coloneqq C_1 \sqrt{ F^{-1} \Big( C_2 h(t) \log(R)\Big)}, \quad  R>0.
  \end{align*}
  It is clear that for any $R >0$ fixed, $Q(R)$ is increasing in both $C_1$ and
  $C_2$. Following the idea as in~\cite{conus.joseph.ea:13:on}, to apply the
  Borel--Cantelli lemma, we need to estimate
  \begin{gather*}
    \mathcal{T}_1 (R) \coloneqq \P\left\{ \max_{x\in\left\{y\in \mathbb{Z}^d \colon |y|\leq R\right\}} |u(t,x)| \geq Q(R) \right\} \shortintertext{and}
    \mathcal{T}_2 (R) \coloneqq \P\left\{ \max_{x\in\left\{y\in \mathbb{Z}^d \colon |y|\leq R\right\}} \sup_{y\in B_x} |u(t,y) - u(t,x)| \geq Q(R) \right\},
  \end{gather*}
  which come from the following inequality:
  \begin{gather*}
    \P\left\{ \sup_{|x|\leq R} |u(t,x)| \geq 2 Q(R)  \right\} \leq  \mathcal{T}_1 (R) + \mathcal{T}_2 (R),
  \end{gather*}
  for all positive integer $R$ such that $Q(R) \geq L_t$; see
  Theorem~\ref{T:Tail}.

  Let $C_*$ be the constant in~\eqref{E:mup_smp}. Assume $C_1 = \sqrt{C_*} e $ and
  $C_2 = (2 + 2d) C_*$. Then, due to Theorem~\ref{T:Tail}, we have
  \begin{align*}
    \mathcal{T}_1 (R)
    \leq & \sum_{x\in\{y\in \mathbb{Z}^d \colon |y|\leq R\}} \P\left\{  |u(t,x)| \geq Q(R)  \right\} \\
    \leq & (2R)^d \exp \left(- (C_* h(t))^{-1} F \left(Q(R)^2\big/ \big(C_* e^2\big)\right)\right) = R^{-2}.
  \end{align*}

  The estimate for $\mathcal{T}_2$ is quite similar. By using the same argument
  as in Theorem~\ref{T:Tail} and taking account of part (ii) of
  Lemma~\ref{L:holder}, one can show that with some $L_t' > 0$ and the same
  constants $C_*$ and $C'$ as in part (ii) of Lemma~\ref{L:holder}, for all $z
  \geq L_t'$,
  \begin{align*}
         \P \left(\sup_{y_1, y_2 \in B_x}|u(t,y_1) - u(t,y_2)| \geq z\right)
    \leq \exp \left( - (C_* h(t))^{-1} F\left(z^2 \big/ (C'e^2)\right) \right).
  \end{align*}
  Then, with $C_1 = \sqrt{C'} e $ and $C_2 = (2 + 2d) C_*$, we have
  $\mathcal{T}_1 (R) \leq R^{-2}$ as well. Therefore, with appropriate $C_1 ,C_2
  >0$ and $L_t > 0$ the same as in~\eqref{E:Lt}, the next inequality holds,
  \begin{align*}
    \sum_{R = 1}^{\infty} \P\left\{ \sup_{|x|\leq R} |u(t,x)| \geq 2 Q(R)  \right\}
    \leq & \sum_{ R = 1}^{\lfloor L_t \vee L_t'\rfloor } \P\left\{ \sup_{|x|\leq R} |u(t,x)| \geq 2 Q(R)  \right\} \\
         & + \sum_{R = \lfloor L_t \vee L_t'\rfloor + 1 }^{\infty}\left(\mathcal{T}_1 (R) + \mathcal{T}_2 (R) \right)
         < \infty.
  \end{align*}
  An application of the first Borel--Cantelli lemma completes the proof of
  Theorem~\ref{T:Asymspc}.
\end{proof}

\subsection{Proofs of technical lemmas}\label{SS:lmmpf}

\begin{proof}[Proof of Lemma~\ref{L:rhononde}]
  Since $|\rho|$ is concave on $[M_0,\infty)$, we can find a non-increasing and
  right-continuous function $g^+$ on $[M_0,\infty)$ such that $|\rho(x)| -|\rho
  (M_0)| = \int_{M_0}^x g^+ (y) \ud y$ holds for all $x\in [M_0,\infty)$;
  see~\cite[Theorems 1.4.2 and 1.5.2]{niculescu.persson:18:convex}. We claim
  that $g^+\geq 0$. Because otherwise, due to part~(i) of Hypothesis~\ref{H:rho}
  and the monotonicity of $g^+$, we have $\lim_{x\to\infty} \int_M^x g^+ (y) =
  -\infty$, and thus $\lim_{x\to\infty} |\rho(x)| = -\infty$, which is
  impossible. Therefore, $g^+\geq 0$. Regarding the limit $\lim_{x\to \infty}
  g^+ (x)=0$, if it is not true, then $\lim_{x\to\infty}g^+ (x)>0$,
  which implies that $\lim_{x\to\infty} \rho(x) = \infty$. Thus, by
  L'H\^{o}pital's rule, $\lim_{x\to\infty}\frac{\rho(x)}{x} = \lim_{x\to\infty}
  g^+ (x) > 0$, which contradicts part~(ii) of Hypothesis~\ref{H:rho}. Hence,
  $\lim_{x\to \infty} g^+ (x)=0$. This proves all properties related to $g^+$.
  The case for $g^-(\cdot)$ can be proved in the same way.
\end{proof}

\begin{proof}[Proof of Lemma~\ref{L:ccvtrho}]
  The representations in both~\eqref{E:gp+} and~\eqref{E:gp-} are direct
  consequences of the definitions of $\rho_p^\pm$ in~\eqref{E:rhop} and
  Lemma~\ref{L:rhononde}. Part (ii) is an immediate consequence of part~(i). Now
  we prove part~(iii). It suffices to show the case for $\rho_p^+(\cdot)$ since
  the case for $\rho_p^-(\cdot)$ can be proved in the same way and the case for
  $\rho_p$ follows from those two cases. Let $g^+$ be given in~\eqref{E:gp+}. We
  need to show that $g_p^+$ is non-increasing on $[M^p,\infty)$ with some $M
  \geq M_0$. We know that $g^+$ is non-increasing on $(M_0,\infty)$, it suffices
  to show that $\varphi(x) = \frac{|\rho(x)|}{x}$ is non-increasing for $x$ large
  enough. To show this property, we write
  \begin{align*}
    |\rho(x)| = \int_{M_0}^x g^+ (y)dy + |\rho (M_0)|.
  \end{align*}
  Thus for almost every $x\in (M_0,\infty)$,
  \begin{align*}
    \varphi'(x) = x^{-2} \left[g^+ (x) \left(x - M_0\right) - \int_{M_0}^x g^+ (y) \ud y + g^+ (x) M_0 - |\rho(M_0)| \right]
    \leq g^+ (x) M_0 - |\rho(M_0)|,
  \end{align*}
  where the last inequality follows from the fact that $g^+$ is non-increasing
  on $(M_0,\infty)$. Since $g^+(x)\to 0$ as $x\to\infty$ (see
  Lemma~\ref{L:rhononde}), we conclude that $g^+(x) M_0 - |\rho (M_0)| < 0$ for
  $x$ large enough. In other words, there exists $M\geq M_0$, such that
  $\varphi$ is non-increasing on $[M, \infty)$. Therefore, $\rho_p^+$ is concave
  on $[M^p, \infty)$. For the same reason, we can show that $\rho_p^-$, and thus
  $\rho_p$, are also concave on $[M^p,\infty)$ with a possibly different $M\geq
  M_0$. This proves part~(ii). Finally, from the above argument, we see that
  when $M_0=0$, then $g^+(x) M_0 - |\rho(M_0)| = -|\rho(0)|\le 0$ for all $x\ge
  0$. Hence, part~(iv) follows. This completes the proof of
  Lemma~\ref{L:ccvtrho}.
\end{proof}

\begin{proof}[Proof of Lemma~\ref{L:Concave}]
   In the following, we will prove~\eqref{E:up3} only. The proof
   of~\eqref{E:up4} is similar. Let $p>0$ and $U\in L^p(\Omega)$.  Set $\alpha
   \coloneqq \P(U \geq M)$. It is clear that when $\alpha =0$, the
   inequality~\eqref{E:up3} is trivially true. So, we may assume that
   $\alpha>0$. Since $\rho_p^+(\cdot)$ is concave on $[M^p,\infty)$ (see
   Lemma~\ref{L:ccvtrho}), we can apply Jensen's inequality to see that
  \begin{align*}
    \E \left[\rho_p^+ (|U|^p) \one_{[M,\infty)}(U)  \right]
    \leq \alpha  \rho_p^+ \left(\frac{1}{\alpha} \E\left[|U|^p\one_{[M,\infty)}(U)\right] \right)
    \leq \alpha  \rho_p^+ \left(\frac{1}{\alpha} \E\left[|U|^p\right] \right),
  \end{align*}
  where the second inequality is due to the monotonicity of $\rho_p^+(\cdot)$;
  see Lemma~\ref{L:ccvtrho}. Denote $x = \E[|U|^p]$. Since $x/\alpha\ge M^p$,
  another application of the monotonicity and the concavity of
  $\rho_p^+(\cdot)$ shows that
  \begin{align*}
    \rho_p^+ (\alpha^{-1}x) \leq & \rho_p^+ (M^p +
    \alpha^{-1}x)
    \leq  \rho_p^+ (M^p) + \alpha^{-1} \big(
    \rho_p^+ (M^p + x) - \rho_p^+ (M^p)
  \big) \leq \alpha^{-1} \rho_p^+ (M^p + x).
  \end{align*}
  This implies that
  \begin{align*}
    \E \big[\rho_p^+ (|U|^p) \one_{[M,\infty)}(U)\big]
    \leq & \rho_p^+ \left(M^p + \mathbb{E} \left[|U|^p \right] \right)
    = \rho_p^+ \left(M^p + \Norm{U}_p^p \right),
  \end{align*}
  which completes the proof of Lemma~\ref{L:Concave}.
\end{proof}

\begin{proof}[Proof of Lemma~\ref{L:rho-p}]
  Fix an arbitrary $p\ge 2$. By the subadditivity of $\theta_{2/p}(\cdot)$,
  \begin{align*}
    \Norm{\rho(U)}_p^2
    \le & \left(\E\left[|\rho(U)|^p 1_{\{|U|\le  M\}}\right]\right)^{2/p}
        + \left(\E\left[|\rho(U)|^p 1_{\{ u \ge  M\}}\right]\right)^{2/p}
        + \left(\E\left[|\rho(U)|^p 1_{\{ u \le -M\}}\right]\right)^{2/p}\\
    \le & K_M^2
        + \left(\theta_{2/p} \circ \E_{\ge} \circ \theta_p \circ |\rho|\right) (U)
        + \left(\theta_{2/p} \circ \E_{\le} \circ \theta_p \circ |\rho|\right) (U)\\
    =   & K_M^2
        + \left(\theta_{2/p} \circ \E_{\ge} \circ \theta_p \circ |\rho| \circ \theta_{1/p} \circ \theta_p \right) (U)
        + \left(\theta_{2/p} \circ \E_{\le} \circ \theta_p \circ |\rho| \circ r \circ \theta_{1/p} \circ \theta_p \right) (U)\\
    =   & K_M^2
        + \left(\theta_{2/p} \circ \E_{\ge} \circ \rho_p^+ \circ \theta_p \right) (U)
        + \left(\theta_{2/p} \circ \E_{\le} \circ \rho_p^- \circ \theta_p \right) (U),
  \end{align*}
  where $\E_\ge$ denote the expectation under $\{u \ge M\}$, namely,
  $\E_\ge\left(f(U)\right) = \E\left(f(U)1_{\{ u\ge M\} }\right)$, and $\E_\le$
  is defined in the same way. Set $y\coloneqq M^p+\Norm{U}_p^p$. By
  Lemma~\ref{L:Concave},
  \begin{align*}
    \Norm{\rho(U)}_p^2
    \le & K_M^2
        + \left(\theta_{2/p} \circ \rho_p^+ \right) (y)
        + \left(\theta_{2/p} \circ \rho_p^- \right) (y)\\
    =   & K_M^2
        + \left(\theta_{2/p} \circ \theta_p \circ |\rho| \circ \theta_{1/p} \right) (y)
        + \left(\theta_{2/p} \circ \theta_p \circ |\rho| \circ r \circ \theta_{1/p} \right) (y)\\
    =   & K_M^2
        + \left(\theta_2 \circ |\rho| \circ \theta_{1/2} \circ \theta_{2/p} \right) (y)
        + \left(\theta_2 \circ |\rho| \circ r \circ \theta_{1/2} \circ \theta_{2/p} \right) (y)\\
    =   & K_M^2
        + \left(\rho_2^+ \circ \theta_{2/p} \right) (y)
        + \left(\rho_2^- \circ \theta_{2/p} \right) (y)\\
    \le & K_M^2
        + \rho_2^+ \left(M^2 + \Norm{U}_p^2\right)
        + \rho_2^- \left(M^2 + \Norm{U}_p^2\right),
  \end{align*}
  where the last step is due to the subadditivity of $\theta_{2/p}(\cdot)$ and
  the monotonicity of $\rho_2^\pm(\cdot)$.
\end{proof}

\begin{proof}[Proof of Lemma~\ref{L:Gmm}]
  The finiteness of $F^{-1}(k) + b$ is a consequence of part~(ii) of
  Hypothesis~\ref{H:rho}. Thus it suffices to show the first inequality
  in~\eqref{E:BIalt}. To this end, we first prove it by verifying $x \leq
  \gamma_0 (k) + 2 b$, with
  \begin{align*}
    \gamma_0 (k)  \coloneqq \inf \left\{ x\in (M^2,\infty)  \colon \frac{\rho_2 (x)}{x} \leq \frac{1}{k}\ \mathrm{and}\ g_2(x) \leq \frac{1}{2k} \right\}.
  \end{align*}
   Here, $g_2(\cdot)\coloneqq g_2^+(\cdot) + g_2^-(\cdot)$ with $g_2^+(\cdot)$
   and $g_2^-(\cdot)$ defined in~\eqref{E:gp+} and~\eqref{E:gp-}, respectively.
   Note that $\gamma_0(k)\ge M^2$ and $\rho_2(\cdot)$ is concave on
   $[M^2,\infty)$. Thus, if $x \ge \gamma_0(k) + 2b$, we have
  \begin{align*}
    k\rho_2(x) + b
    =   & k \rho_2 (\gamma_0(k)) + k\int_{\gamma_0(k)}^{x} g_2(y) \ud y + b \\
    \le & \gamma_0(k) + \frac{1}{2} \left(x - \gamma_0(k)\right) + b        \\
    \le & \gamma_0(k) + \frac{1}{2} \left(x - \gamma_0(k)\right) + \frac{1}{2} (x-\gamma_0(k))
    = x.
  \end{align*}
  In the next step, we can show that
  \begin{align}\label{E:gmmgmm0}
    \gamma_0(k) \leq 2 F^{-1}(k),\qquad  \text{for all } k > 0.
  \end{align}
  First if $M = 0$, then
  \begin{align*}
   \frac{2}{F(x)} = \frac{x}{2\rho_2(x)} \leq  \frac{x}{\rho_2(x)} = x \left(\int_0^x\ud y g_2(y)\right)^{-1} \leq \frac{1}{g_2(x)},
  \end{align*}
  by the non-increasing property of $g_2 (\cdot)$, and~\eqref{E:gmmgmm0} follows
  immediately from the definition of $F^{-1}$; see~\eqref{E:F^-1}. On the other
  hand, if $M > 0$. Then, for any $x \geq 2 F^{-1}(k) \geq 3 F^{-1}(k)/2$, it
  holds that
  \begin{align*}
    g_2(x) \leq \frac{\rho_2(x) - \rho_2(M^2)}{x - M^2} \leq \frac{\rho_2 (3F^{-1}(k)/2) -  \rho_2(M^2)}{3 F^{-1}(k)/2 - M^2} .
  \end{align*}
  Notice that $F^{-1}(k) \geq 2M^2$ implies $3F^{-1}(k)/2 - M^2 \geq F^{-1}(k)$
  and
  \begin{align*}
    \rho_2 (3F^{-1}(k)/2) -  \rho_2(M^2) = \int_{M^2}^{3F^{-1}(k)/2} \ud y g_2(y)
    \leq & \frac{3F^{-1}(k)/2- M^2}{F^{-1}(k) - M^2} \int_{M^2}^{F^{-1}(k)} \ud y g_2(y) \\
    \leq & 2\rho_2 \left(F^{-1}(k)\right).
  \end{align*}
  Combining above inequalities and~\eqref{E:FF-1}, we get
  \begin{align}\label{E:g2}
    g_2(x)
    \leq \frac{2\rho_2 \left(F^{-1}(k)\right)}{F^{-1}(k)}
    =    \frac{1}{2 F\circ F^{-1}(k)}
    \leq \frac{1}{2k}.
  \end{align}
  This implies that~\eqref{E:gmmgmm0}, and thus completes the proof of
  Lemma~\ref{L:Gmm}.
\end{proof}

\begin{proof}[Proof of Lemma~\ref{L:J1d=1}]
  Without loss generality, we assume that $\mu$ is nonnegative. Otherwise, one
  simply replaces $\mu$ by $|\mu|$. By the Plancherel theorem and the fact that
  $f$ is nonnegative-definite, we can write
  \begin{align*}
    \sup_{x\in \R^d} \int_{\R^d} \ud y\: p_{t}  (x-y) f(y)
    = & \sup_{x\in \R^d} (2\pi)^{-d} \int_{\R^d} \ud\xi\: e^{- ix \cdot \xi - \frac{t |\xi|^2}{2}} \widehat{f}(\xi) \\
    = & (2\pi)^{-d} \int_{\R^d} \ud\xi\: e^{- \frac{t |\xi|^2}{2}} \widehat{f}(\xi)
    = \int_{\R^d} \ud z\: p_t (z) f(z)
    = k(t),
  \end{align*}
  from which one proves~\eqref{E:J1J0}. As for~\eqref{E:J1d=1}, we proceed with
  a general $d\ge 1$ and make the restriction to $d = 1$ when some integrability
  issue comes up. By the Cauchy-Schwarz inequality, we see that
  \begin{align*}
    \MoveEqLeft \int_{\R^d} \ud y\: p_{t-s} (x-y)   \calJ_0^2(s,y)        \\
    & = \int_{\R^d} \ud y \iint_{\R^{2d}} \mu( \ud z) \mu( \ud z')\: p_{t-s} (x-y) p_s(y-z) p_s(y-z') \\
    & \leq \int_{\R^{2d}} |\mu|( \ud z) |\mu|( \ud z')
      \left(\int_{\R^{d}} \ud y\: p_{t-s} (x-y) p_s(y-z)^2 \right)^{1/2}
      \left(\int_{\R^{d}} \ud y\: p_{t-s} (x-y) p_s(y-z')^2 \right)^{1/2} \\
    & = \left[\int_{\R^{d}} |\mu|( \ud z) \left(\int_{\R^{d}} \ud y\: p_{t-s} (x-y) p_s(y-z)^2 \right)^{1/2}\right]^2.
  \end{align*}
  Because $p_t^2 (x) = (4\pi t)^{-d/2} p_{t /2}(x)\leq (2\pi t)^{-d/2} p_t(x)$
  for all $(t,x) \in \R_+\times \R^d$, we can write
  \begin{align*}
    \int_{\R^{d}} \ud y\: p_{t-s} (x-y) p_s(y-z)^2 \leq (2\pi s)^{-d/2}
    \int_{\R^{d}} \ud y\: p_{t-s} (x-y) p_s(y-z) = (2\pi s)^{-d/2} p_t(x-z),
  \end{align*}
  which implies that
  \begin{align*}
    \calJ_1(t,x) \leq & (2t)^{d/2}
    \int_0^t \ud s\: s^{-d/2} k(t-s)
    \left(\int_{\R^{d}} |\mu| ( \ud z) p_{t/2}(x-z)\right)^2 \\
    = & (2t)^{d/2}
     \calJ_+ \left(t/2, x \right)^2
     \int_0^t \ud s\: s^{-d/2}\: k(t-s).
  \end{align*}
  Notice that the integral against the time argument in the last expression is
  finite for all $t>0$, if and only if $d = 1$. Moreover, in case $d = 1$, we
  can deduce that
  \begin{align*}
    \int_0^t \ud s\: s^{-1/2} k(t-s)
    \leq (2\pi)^{-1/2} \int_0^t\ud s\: s^{-1/2} (t-s)^{-1/2}
    \int_{\R} \ud z\: e^{-\frac{x^2}{2 t}} f(z)
    = \pi t^{1/2} k(t).
  \end{align*}
  Therefore,
  \begin{align} \label{E_:J1d=1}
    \calJ_1(t,x)
    \le \sqrt{2}\pi\: t \:  k(t)\:\calJ_+ \left(t/2, x\right)^2
    <\infty.
  \end{align}
  Finally, notice that $k(t)$ is a nonincreasing function of $t$ because $k(t) =
  (2\pi)^{-d}\int_{\R^d} \widehat{f}(\ud\xi) e^{-t|\xi|^2/2}$. From~\eqref{E:h},
  we see that
  \begin{align*}
    2 h(t) \ge 2 h(t/2) = \int_0^t \ud s \, k(s)
    \ge \int_0^{t} \ud s \,k(t)  = t k(t).
  \end{align*}
  Plugging this inequality back to~\eqref{E_:J1d=1} proves~\eqref{E:J1d=1}. This
  completes the proof of Lemma~\ref{L:J1d=1}.
\end{proof}

\section{Examples on various sublinear diffusion coefficients \texorpdfstring{$\rho$}{}}\label{S:Rho}

In this section, we derive the explicit expressions for $F(\cdot)$ and its
inverse $F^{-1}(\cdot)$ for the sublinear diffusion coefficients $\rho$ given in
examples~\eqref{E:Ex-LipRho}--\eqref{E:Ex-VSV}. The results are summarized in
Figure~\ref{F:F-Inv} at the end of this session.

\begin{proposition}\label{P:Rho-Alpha}
  The following diffusion coefficient
  \begin{align}\label{E:Rho-Alpha-R}
    \rho(u) =  \frac{|u|}{ (r + |u|)^{1-\alpha}} \qquad
    \text{for $u\in\R$, with $r\ge 0$ and $\alpha\in [0,1)$,}
  \end{align}
  which reduces to
  \begin{align*}
    \rho(u) =
    \begin{dcases}
      \quad |u|^{\alpha} \quad \text{with $\alpha\in [0,1)$} & \text{if $r=0$,}        \\
      \frac{|u|}{1+|u|}                                      & \text{if $\alpha = 0$,} \\
    \end{dcases}
  \end{align*}
  satisfies the following properties:
  \begin{enumerate}
    \item $\rho$ is globally Lipschitz provided that $r>0$, in which case,
      $\rho'(0) = r^{-(1-\alpha)}<\infty$;
    \item $\rho$ satisfies Hypothesis~\ref{H:rho} with $M_0=0$;
    \item the corresponding $F$ and $F^{-1}$ take the following explicit forms:
      \begin{align}
        F(u)      = \frac{1}{8}\left(r + u^{1/2}\right)^{2(1-\alpha)}, \quad \text{for all $u\ge 0$,}\nonumber \shortintertext{and}
        F^{-1}(x) = \left((8x)^{1/(2(1-\alpha))}-r\right)^{2} \one_{\left\{x\ge 8^{-1}r^{2(1-\alpha)}\right\}}.
        \label{E:F-Inv-Alpha}
      \end{align}
      In particular, $F^{-1}(x)\le (8x)^{1/(1-\alpha)}$ for all $x\ge 0$.
  \end{enumerate}
\end{proposition}
\begin{proof}
  The diffusion coefficient $\rho$ given in~\eqref{E:Rho-Alpha-R} is
  clearly locally bounded, and has sublinear growth:
  \begin{align*}
    \rho(u) \asymp |u|^{\alpha} \quad \text{as $|u|\to \infty$.}
  \end{align*}
  Notice that for $u\ge 0$,
  \begin{align*}
    \rho'(u) = \left(r+u\right)^{\alpha-2} \left(r+\alpha  u \right)
    \to r^{-(1-\alpha)} \quad \text{as $u\to 0$.}
  \end{align*}
  $\rho(\cdot)$ is concave separately on $\R_+$ and $\R_-$ since
  \begin{align*}
    \rho''(u)
    = -(1-\alpha)
      \left(r+|u|\right)^{\alpha-3}
      \left( 2r + \alpha |u| \right)\le 0, \quad \text{for all $u\in\R$.}
  \end{align*}
  Therefore, $\rho$ satisfies Hypothesis~\ref{H:rho} with $M_0=0$.

  The form of $\rho$ (see also~\eqref{E:Ex-LipRho}) makes it easy to compute
  $\rho_2$, $F$ and $F^{-1}$. Indeed, in this case, for $u\ge 0$,
  \begin{align*}
    \rho_2(u) = \frac{2u}{ \left(r + u^{1/2}\right)^{2(1-\alpha)}}  \quad \text{and} \quad
    F(u)      = \frac{1}{8}\left(r + u^{1/2}\right)^{2(1-\alpha)}, \quad \text{for all $u\ge 0$.}
  \end{align*}
  The expression of $F^{-1}$ is readily to be derived from the above formula for
  $F$.
\end{proof}

The diffusion term $\rho$ in the following examples does not exhibit global
concavity, which motivates us to propose the asymptotic concave condition in
Hypothesis~\ref{H:rho}.

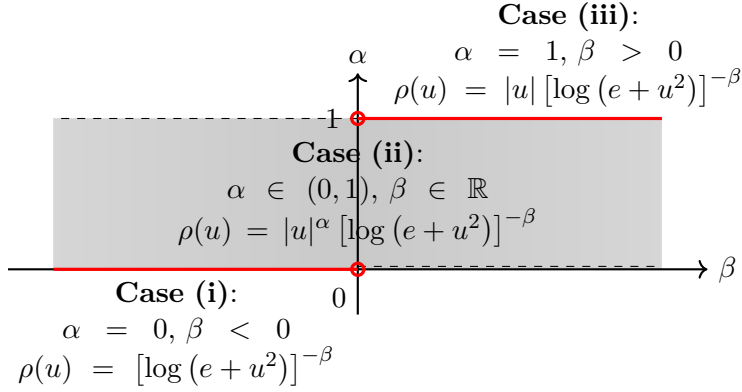
\begin{figure}[htb]
  \centering
  \begin{tikzpicture}[x=2cm,y=2cm]

    \filldraw [shading=radial, left color=gray!50!white,  right color=white, draw = white, opacity =0.2] (0,0) rectangle (2,1);
    \filldraw [shading=radial, left color=white, right color=gray!50!white,  draw = white, opacity =0.2] (0,0) rectangle (-2,1);
    \draw[dashed] (0,1) --++(-2,0);
    \draw[dashed,yshift=+0.1em] (0,0) --++(+2,0);

    \draw[->,thick] (-2.3,0) -- (2.3,0) node[right] {$\beta$};
    \draw[->,thick] (0,-0.3) -- (0,1.3) node[above] {$\alpha$};
    \draw (+0.05,1) --++(-0.1,0) node[left] {$1$};

    \draw[very thick, red] (0,0) node[left, yshift = -1em,  black] {$0$} -- (-2,0)
      node[midway, black, yshift = -2.2em, text width = 16em, xshift = -1em, align = center]
      { \textbf{Case (i)}: \\
        $\alpha=0$,  $\beta<0$ \\
        $\rho(u) = \left[\log\left(e+u^2\right)\right]^{-\beta}$
      };
    \draw[very thick, red] (0,0) circle (0.04);

    \draw[very thick, red] (0,1) -- (+2,1)
      node[midway, black, yshift = +2.2em, xshift = +2em, text width = 16em, align = center]
      { \textbf{Case (iii)}: \\
        $\alpha=1$, $\beta>0$ \\
        $\rho(u) = |u| \left[\log\left(e+u^2\right)\right]^{-\beta}$
      };
    \draw[very thick, red] (0,1) circle (0.04);

    \node [text width = 15em, align = center] at (-0.0,0.5)
      { \textbf{Case (ii)}:\\
        $\alpha\in(0,1)$, $\beta\in\R$ \\
        $\rho (u) = |u|^{\alpha} \left[ \log\left(e+u^2\right)\right]^{-\beta}$
      };

  \end{tikzpicture}

  \caption{Three cases for $\rho$ given in~\eqref{E:Ex-LogRho} and
  Proposition~\ref{P:Rho-Log}.}

  \label{F:Log}
\end{figure}

\begin{proposition}\label{P:Rho-Log}
  Suppose that
  \begin{align*}
    \rho (u) = |u|^{\alpha} \left[ \log\left(e+u^2\right)\right]^{-\beta},
  \end{align*}
  with $\alpha, \beta$ fulfilling one of the three conditions given as
  in~\eqref{E:Ex-LogRho}; see Figure~\ref{F:Log}. Then we have
  \begin{enumerate}
    \item $\rho$ is globally Lipschitz only when $\alpha = 1$ (for all
    $\beta\in\R)$, in which case, $\rho'(0) = 1$;
    \item $\rho$ satisfies Hypothesis~\ref{H:rho} with some $M_0>0$;
    \item the corresponding $F$ and $F^{-1}$ are given below:
      \begin{align}\label{E:F-Log}
        F(u) = \frac{1}{8} u^{1 - \alpha} \left[\log(e + u)\right]^{2\beta}, \quad u>0,
      \end{align}
      and
      \begin{align}\label{E:F-Inv-Log}
        \begin{dcases}
          F^{-1}(x) \asymp (8x)^{1/(1-\alpha)} \left[\frac{1}{1-\alpha}\log\left(e+x\right)\right]^{2\beta/(1-\alpha)}, \quad x\to\infty, & \text{Cases (i) and (ii)}, \\
          F^{-1}(x) = \left(\exp\left(2^{3/(2\beta)} x^{1/(2\beta)} \right) - e\right)\one_{\{x>8^{-1}\}},                                & \text{Case (iii).}
        \end{dcases}
      \end{align}
      In particular, in Cases (i) and (ii),
      \begin{align*}
        F^{-1}(x) \lesssim \left[x \left(\log x\right)^{-2\beta}\right]^{1/(1-\alpha)},
        \quad \text{as $x\to\infty$.}
      \end{align*}
  \end{enumerate}
\end{proposition}
\begin{proof}
  It is clear that $\rho$ is locally bounded and has sublinear growth at
  infinity. For the asymptotic concavity, it is elementary, though tedious, to
  show that for $u\ge 0$,
  \begin{align*}
    \rho'(u)
    = u^{\alpha-1} \left[\log\left(e+u^2\right)\right]^{-\beta} \left(\alpha -\frac{2 \beta u^2}{\left(e+u^2\right) \log \left(e+u^2\right)}\right)
    \quad \text{and} \quad
    \rho''(u) = q_1(u) q_2(u)
  \end{align*}
  with
  \begin{align*}
    q_1(u) \coloneqq & (e+u^2)^{-2} u^{\alpha - 2} \left[\log\left(e+u^2\right)\right]^{-\beta -2}, \\
    q_2(u) \coloneqq & 4\beta (\beta + 1) u^4 - 2\beta u^2 \big((2\alpha - 1) u^2 +2\alpha e + e\big)\log(e+u^2) \\
                     & + \alpha (\alpha - 1) (e+u^2)^2 \left[\log\left(e+u^2\right)\right]^2.
  \end{align*}
  From the expression of $\rho'(u)$, we see that $\rho$ is globally Lipschitz
  only when $\alpha = 1$ and in that case $\rho'(0) = 1$. From the expression of
  $\rho''(u)$, we see that $q_1(u)\geq 0$ for $u > 0$, but $q_2(u)$ may take
  negative values when $u$ is small. By considering the three cases given
  in~\eqref{E:Ex-LogRho} (see Figure~\ref{F:Log}), we see that
  \begin{align*}
    q_2(u) \asymp \overline{q}_2(u) < 0, \quad |u|\to\infty,
    \quad \text{with~~}
    \overline{q}_2 (u) \coloneqq
    \begin{dcases}
      2\beta u^4\log(e+u^2),                        & \text{Case (i)},  \\
      \alpha (\alpha - 1) (e+u^2)^2 \log^2 (e+u^2), & \text{Case (ii)}, \\
      -2\beta u^4\log(e+u^2),                       & \text{Case (iii)}.
    \end{dcases}
  \end{align*}
  Hence, $\rho$ is asymptotically concave separately on $[M_0,\infty)$ and
  $(-\infty,-M_0]$ for some constant $M_0>0$. This proves that $\rho$ satisfies
  Hypothesis~\ref{H:rho}. \bigskip

  It is straightforward to derive the expression for $F$:
  \begin{align*}
    \rho_2(u) = 2 u^\alpha \left[\log\left(e+u\right)\right]^{-2\beta} \quad \text{and} \quad
    F(u) = \frac{u}{4\rho_2(u)}
         = \frac{1}{8} u^{1 - \alpha} \left[\log(e + u)\right]^{2\beta}, \quad u>0,
  \end{align*}
  It remains to prove~\eqref{E:F-Inv-Log}. In Case (iii), namely, $\alpha = 1$
  and $\beta>0$, $F^{-1}(x)$ can be obtained explicitly as given in the
  statement of the proposition. However, in Cases (i) and (ii), we have $\alpha
  \in [0,1)$, and it seems impossible to find the explicit formula for
  $F^{-1}(x)$. Instead, one can find the exact asymptotics. Set
  \begin{align*}
    \Theta(x) \coloneqq x^{1/(1-\alpha)} \left[\log(e +x)\right]^{-2\beta/(1-\alpha)}, \quad (x>0).
  \end{align*}
  From~\eqref{E:F-Log}, we see that, for $\lambda>0$ and $x>0$,
  \begin{align*}
    F\left(\lambda \Theta(x)\right)
    = \frac{\lambda \Theta(x)}{4 \rho_2(\lambda \Theta(x))}
    = & \frac{\lambda^{1-\alpha} x}{8} \left[\frac{\log\left(e+x\right)}{\log\left( e + \lambda x^{1/(1-\alpha)} \left[\log(e+x)\right ]^{-2\beta/(1-\alpha)} \right)}\right]^{-2\beta} \\
    \asymp & 8^{-1} \lambda^{1-\alpha} (1-\alpha)^{-2\beta} x,  \quad \text{as $x\to\infty$.}
  \end{align*}
  Hence, by choosing $\lambda_0 \coloneqq 8^{1/(1-\alpha)}
  (1-\alpha)^{2\beta/(1-\alpha)}$, we see that
  \begin{align*}
    F\left(\lambda_0 \Theta(x)\right) \asymp x, \quad \text{or equivalently} \quad
    F^{-1}(x) \asymp \lambda_0 \Theta(x), \quad \text{as $x\to\infty$}.
  \end{align*}
  This completes the proof of Proposition~\ref{P:Rho-Log}.
\end{proof}

\begin{proposition}\label{P:Rho-VSV}
  The diffusion coefficient
  \begin{align*}
    \rho (u) = |u|\exp\left( -\beta \Big(\log \log \left(e+u^2\right)\Big)^{\kappa} \right)
    \quad \text{with $\kappa > 0$ and $\beta>0$}
  \end{align*}
  have the following properties:
  \begin{enumerate}
    \item $\rho$ is globally Lipschitz continuous with $\rho'(0) = 1$;
    \item $\rho$ satisfies Hypothesis~\ref{H:rho} with some $M_0>0$;
    \item the corresponding $F$ and $F^{-1}$ take the following explicit forms:
      \begin{gather}
        F(u)      = \frac{1}{8} \exp\left(2\beta \left(\log\log(e+u)\right)^\kappa\right), \quad \text{for all $u>0$,}\nonumber \shortintertext{and}
        F^{-1}(x) = \left(\exp\left\{\exp\left(\left((2\beta)^{-1} \log(8x) \right)^{1/\kappa}\right)\right\} - e \right)\one_{\{x>1/8\}}.\label{E:F-Inv-VSV}
      \end{gather}
  \end{enumerate}
\end{proposition}
\begin{proof}
  It is easy to see that $\rho$ is locally bounded. Since $\beta>0$, we see that
  $\rho$ has sublinear growth at infinity. Thus, $\rho$ satisfies both parts
  (i) and (ii) of Hypothesis~\ref{H:rho}. Notice that
  \begin{align*}
    \rho'(u) = e^{-\beta \log ^{\kappa}\left(\log\left(e+u^2\right)\right)}\left(1-\frac{2 \kappa\beta  u^2 \log ^{\kappa-1}\left(\log\left(e+u^2\right)\right)}{\left(e+u^2\right) \log\left(e+u^2\right)}\right).
  \end{align*}
  Using the fact that $\log(a+u) \approx \log(a) + u/a$ as $u\to 0_+$ with
  $a>0$, we see that
  \begin{align*}
    \rho'(u) \approx e^{-\beta e^{-\kappa} u^{2\kappa}} \left(1 - 2\kappa\beta e^{-\kappa} u^{2\kappa}\right)
    \to 1 \quad \text{as $u\to 0_+$, provided that $\kappa>0$.}
  \end{align*}
  Hence, $\rho$ is globally Lipschitz with $\rho'(0)=1$. As for the asymptotic
  concavity, we have that $\rho''(u) = q_1(u) q_2(u)$ for $u > 0$ with
  \begin{align*}
     q_1(u)
     =       & 2 \kappa \beta u e^{-\beta\log^{\kappa} \left(\log\left(e+u^2\right)\right)}
               \frac{\log ^{\kappa -2}\left(\log\left(e+u^2\right)\right)}{\left(e+u^2\right)^2 \log ^2\left(e+u^2\right)}
    \ge        0, \shortintertext{and}
     q_2(u)
     =       & \quad  2 \kappa  \beta  u^2 \log ^{\kappa }\left(\log\left(e+u^2\right)\right)-2 (\kappa -1) u^2         \\
             & -\left(\left(u^2+3 e\right) \log\left(e+u^2\right)-2 u^2\right) \log \left(\log\left(e+u^2\right)\right) \\
     \asymp  & - u^2 \log\left(e+u^2\right)\log\left(\log\left(e+u^2\right)\right) \le 0,  \quad \text{as $u\to\infty$.}
  \end{align*}
  Therefore, $\rho$ is asymptotically concave on $\R_+$. This proves that $\rho$
  satisfies Hypothesis~\ref{H:rho}. Finally, it is routine to compute $F$ and
  $F^{-1}$. This completes the proof of Proposition~\ref{P:Rho-VSV}.
\end{proof}

\begin{remark}
   When $M_0\ne 0$, $\rho$ is only asymptotically concave (separately) on $\R_+$
   and $\R_-$. Hence, when using the formulas in~\eqref{E:F-Inv-Log}
   and~\eqref{E:F-Inv-VSV}, we require the argument of $F^{-1}$ to be
   sufficiently large.
\end{remark}

\begin{figure}[htpb!]
  \centering
  \begin{center}
    \begin{tikzpicture}[scale=0.9, transform shape]
      \tikzset{>=latex}
      \def\vd{0.2}
      \def\vx{1}
      \def\mag{1}
      \def\vy{9.0}
      \def\magbar{5.5}
      \newcommand{\MyColorH}{white}
      \newcommand{\MyColorG}{red}
      \newcommand{\MyColorF}{red!80!green}
      \newcommand{\MyColorE}{red!30!green}
      \newcommand{\MyColorD}{green}
      \newcommand{\MyColorC}{blue}
      \newcommand{\MyColorB}{green}
      \newcommand{\MyColorA}{yellow}

      \coordinate (o0) at (0,0);
      \coordinate (o1) at (0,1);           \coordinate (M0) at (\vx+\mag,0);
      \coordinate (R1) at (\vx,1.0);       \coordinate (M1) at (\vx+\mag,1);
      \coordinate (R2) at (\vx,1.0+1*\vd); \coordinate (M2) at (\vx+\mag,1+1*\vy/7);
      \coordinate (R3) at (\vx,8.5-5*\vd); \coordinate (M3) at (\vx+\mag,1+2*\vy/7);
      \coordinate (R4) at (\vx,8.5-4*\vd); \coordinate (M4) at (\vx+\mag,1+3*\vy/7);
      \coordinate (R5) at (\vx,8.5-3*\vd); \coordinate (M5) at (\vx+\mag,1+4*\vy/7);
      \coordinate (R6) at (\vx,8.5-2*\vd); \coordinate (M6) at (\vx+\mag,1+5*\vy/7);
      \coordinate (R7) at (\vx,8.5-1*\vd); \coordinate (M7) at (\vx+\mag,1+6*\vy/7);
      \coordinate (R8) at (\vx,8.5-0*\vd); \coordinate (M8) at (\vx+\mag,1+7*\vy/7);

      \begin{scope}[very thick]
        \draw [name path = Add] (0,1) -- (R1);
        \draw [name path = PAM] (0,0) -- (R7);
        \draw (R7) -- (M7);
        \draw (R1) -- (M1);
        \path[name intersections={of=Add and PAM}]
          (intersection-1) coordinate (o2);

        \draw (M0) --++ (0.1,0) --++(-0.2,0);
        \draw (M1) --++ (0.1,0) --++(-0.2,0);
        \draw (M2) --++ (0.1,0) --++(-0.2,0);
        \draw (M3) --++ (0.1,0) --++(-0.2,0);
        \draw (M4) --++ (0.1,0) --++(-0.2,0);
        \draw (M5) --++ (0.1,0) --++(-0.2,0);
        \draw (M6) --++ (0.1,0) --++(-0.2,0);
        \draw (M7) --++ (0.1,0) --++(-0.2,0);
        \draw (M8) --++ (0.1,0) --++(-0.2,0);

      \end{scope}

      \begin{scope}[dotted]
        \def\myShift{4}
        \draw         (M0) --++ (\myShift,0);
        \draw[dashed] (M1) --++ (\myShift,0);
        \draw         (M2) --++ (\myShift,0);
        \draw         (M3) --++ (\myShift,0);
        \draw         (M4) --++ (\myShift,0);
        \draw         (M5) --++ (\myShift,0);
        \draw         (M6) --++ (\myShift,0);
        \draw[dashed] (M7) --++ (\myShift,0);
        \draw         (M8) --++ (\myShift,0);
      \end{scope}

      \begin{scope}[thick]
        \def\myShift{8em}

        \begin{scope}[shorten >=0.3em, shorten <=0.3em]
          \draw[->] ([xshift = \myShift]M6) -- ([xshift = \myShift]M4) node [right, yshift = 0.7em] {$\beta$};
          \draw[->] ([xshift = \myShift]M2) -- ([xshift = \myShift]M0) node [right, yshift = 0.7em] {$\beta$};
          \draw[->] ([xshift = \myShift]M4) -- ([xshift = \myShift]M3) node [right, yshift = 0.7em] {$\beta$};
          \draw[->] ([xshift = \myShift]M7) -- ([xshift = \myShift]M6) node [right, yshift = 0.7em] {$\beta$};
        \end{scope}

        \draw ([xshift = \myShift, yshift = -0.4em]M7) --++(-0.1,0) --++(0.2,0) node [right] {$0$};
        \draw ([xshift = \myShift, yshift = -0.4em]M6) --++(-0.1,0) --++(0.2,0) node [right] {$0$};
        \draw ([xshift = \myShift, yshift = -0.4em]M4) --++(-0.1,0) --++(0.2,0) node [right] {$0$};
        \draw ([xshift = \myShift]M5) --++(-0.1,0) --++(0.2,0) node [right] {$\sfrac{1}{4}$};
        \draw ([xshift = \myShift]M1) --++(-0.1,0) --++(0.2,0) node [right] {$0$};
      \end{scope}

      \begin{scope}[thick, every node/.style={right, rotate=90, black, opacity = 1, align = center}]
        \def\myShift{1.0em}
        \def\myShiftPlus{3.5em}
        \draw[top color = \MyColorG, bottom color = \MyColorD,opacity =0.3] ([xshift = \myShift]M3) rectangle node [xshift = -2.0em] {$\alpha=1$} ([xshift = \myShiftPlus]M7);
        \draw[fill = \MyColorC, opacity =0.3]                               ([xshift = \myShift]M2) rectangle node [xshift = -2.0em, text width =3em, ] {$\alpha\in$\\$(0,1)$} ([xshift = \myShiftPlus]M3);
        \draw[fill = \MyColorB, opacity =0.3]                               ([xshift = \myShift]M1) rectangle node [xshift = -1.7em] {$\alpha=0$} ([xshift = \myShiftPlus]M2);
      \end{scope}

      \begin{scope}[thick, every node/.style={right, rotate=90, black, opacity = 1, xshift = -1.7em}]
        \def\myShift{4em}
        \def\myShiftPlus{6em}
        \draw[fill = \MyColorG,opacity =0.4]                                 ([xshift = \myShift]M6) rectangle node {$\kappa<1$} ([xshift = \myShiftPlus]M7);
        \draw[top color = \MyColorF, bottom color = \MyColorE, opacity =0.3] ([xshift = \myShift]M4) rectangle node {$\kappa=1$} ([xshift = \myShiftPlus]M6);
        \draw[fill = \MyColorD,opacity =0.4]                                 ([xshift = \myShift]M3) rectangle node {$1<\kappa$} ([xshift = \myShiftPlus]M4);
      \end{scope}

      \begin{scope}[left color = white, opacity = 0.4, draw = none]
        \filldraw [right color = \MyColorA] (o0) -- (o1) -- (R1) -- (M1) -- (M0) -- (o0);
        \filldraw [right color = \MyColorB] (o1) -- (R1) -- (M1) -- (M2) -- (R2) -- (o1);
        \filldraw [right color = \MyColorC] (o2) -- (R2) -- (M2) -- (M3) -- (R3) -- (o2);
        \filldraw [right color = \MyColorD] (o2) -- (R3) -- (M3) -- (M4) -- (R4) -- (o2);
        \filldraw [right color = \MyColorE] (o2) -- (R4) -- (M4) -- (M5) -- (R5) -- (o2);
        \filldraw [right color = \MyColorF] (o2) -- (R5) -- (M5) -- (M6) -- (R6) -- (o2);
        \filldraw [right color = \MyColorG] (o2) -- (R6) -- (M6) -- (M7) -- (R7) -- (o2);
        \filldraw [right color = \MyColorH] (o2) -- (R7) -- (M7) -- (M8) -- (R8) -- (o2);
      \end{scope}

      \filldraw[white, line width = 2pt] (\vx,0) -- (\vx, 10);
      \draw [->] (-0.5,0) -- (\vx,0) node [below, xshift =-0.5em] {$u$};
      \draw [->] (0,-0.5) -- (0,4) node [above] {$\rho(u)$};

      \begin{scope}[xshift =29em]
        \path (M0) -- (M1) node [midway, xshift = +3.2em] (A) {bounded};      \draw[->] (A) --++(-1.7,0);
        \path (M8) -- (M7) node [midway, xshift = +3.5em] (A) {super linear}; \draw[->] (A) --++(-1.6,0);
      \end{scope}

      \begin{scope}[every node/.style={opacity = 1, black}, rounded corners = 10pt]
        \def\myShift{11em}
        \def\myShiftPlus{9em}
        \draw[top color = \MyColorG, bottom color = \MyColorD, opacity =0.2] ([xshift = \myShift]M3) rectangle node {$u e^{-\beta \left(\log\log\left(e+u^2\right)\right)^\kappa}$} ([xshift = \myShift + \myShiftPlus]M7);
        \draw[fill = \MyColorC, opacity =0.2] ([xshift = \myShift]M2) rectangle node {$\dfrac{u}{(r+u)^{1-\alpha}}$} ([xshift = \myShift + \myShiftPlus]M3);
        \draw[fill = \MyColorB, opacity =0.2] ([xshift = \myShift]M1) rectangle node {$\left[\log\left(e+u^2\right)\right]^{-\beta}$} ([xshift = \myShift + \myShiftPlus]M2);
        \path ([xshift = \myShift]M8) -- ([xshift = \myShift]M8) node [yshift = -2em, xshift = 4em] {$\rho(u) =$};

        \def\myShift{22em}
        \def\myShiftPlus{14em}
        \draw[fill = \MyColorG, opacity =0.2] ([xshift = \myShift]M6) rectangle node {$\exp\left(\exp\left(\left[\frac{1}{2\beta}\log(8x)\right]^{1/\kappa}\right)\right)$} ([xshift = \myShift + \myShiftPlus]M7);
        \draw[top color = \MyColorF, bottom color = \MyColorE, opacity =0.2] ([xshift = \myShift]M4) rectangle node {$\exp\left((8x)^{1/(2\beta)}\right)$} ([xshift = \myShift + \myShiftPlus]M6);
        \draw[fill = \MyColorD, opacity =0.2] ([xshift = \myShift]M3) rectangle node {$\exp\left(\exp\left(\left[\frac{1}{2\beta}\log(8x)\right]^{1/\kappa}\right)\right)$} ([xshift = \myShift + \myShiftPlus]M4);
        \draw[fill = \MyColorC, opacity =0.2] ([xshift = \myShift]M2) rectangle node {$(8x)^{1/(2(1-\alpha))}$} ([xshift = \myShift + \myShiftPlus]M3);
        \draw[fill = \MyColorB, opacity =0.2] ([xshift = \myShift]M1) rectangle node {$(8x)\left[\log x\right]^{2\beta}$} ([xshift = \myShift + \myShiftPlus]M2);
        \path ([xshift = \myShift]M8) -- ([xshift = \myShift]M8) node [yshift = -2em, xshift = 7em] {$F^{-1}(x) \asymp$};

      \end{scope}
    \end{tikzpicture}
  \end{center}

  \caption{Summary of the asymptotics of $F^{-1}$ for $\rho$ in
  Propositions~\ref{P:Rho-Alpha},~\ref{P:Rho-Log} and~\ref{P:Rho-VSV}.}

  \label{F:F-Inv}
\end{figure}
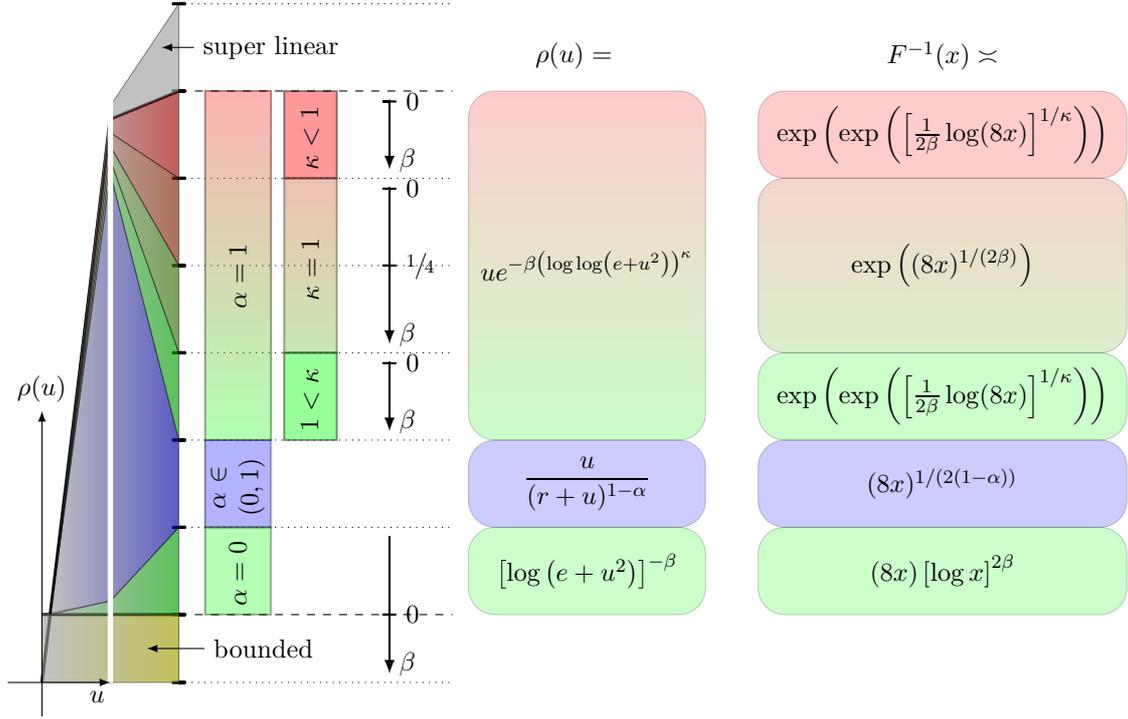

\section{Moment growth for other SPDEs in the sublinear regime}\label{S:Others}

In this section, we will explore the moment upper bounds for other SPDEs in the
sublinear regime.

\subsection{Nonlinear stochastic space-time fractional diffusion
equations}\label{SS:frac}

In this part, we explore the moment upper bounds for the following nonlinear
fractional stochastic partial differential equations (\textit{fractional SPDE})
driven by the multiplicative space-time white noise:
\begin{align}\label{E:fracdiff}
  \left[ \left(\frac{\partial}{\partial t}\right)^{b} + \frac{1}{2} (-\Delta)^{a/2} \right] u(t,x)
  = I_t^{\gamma} \big[\rho(u(t,x)) \dot{W} (t,x)\big],
\end{align}
where $\left(\frac{\partial}{\partial t}\right)^{b}$,
$(-\Delta)^{a/2}$, and $I^{\gamma}_t$ denote the \textit{Caputo fractional
differential operator} with $b\in(0,2)$, fractional Laplacian operator with
$a \in (0, 2]$, and the \textit{Riemann-Liouville operator} with $\gamma
>0$, respectively. We assume that equation~\eqref{E:fracdiff} starts from the
following initial condition(s):
\begin{align}\label{E:InitFrac}
  \begin{dcases}
    u (0,\cdot) = \mu_0,                                                                  & b \in ( 1/2, 1 ], \\
    u (0,\cdot) = \mu_0,\quad \frac{\partial}{\partial t} u(t,\cdot)\bigg|_{t=0} = \mu_1, & b\in (1,2),
  \end{dcases}
\end{align}
where both $\mu_0$ and $\mu_1$ are nonnegative measures. For a more detailed
account of this equation, we refer interested readers
to~\cite{chen.hu.ea:19:nonlinear,chen.guo.ea:22:moments,guo.song.ea:23:stochastic}.

\begin{definition}\label{D:fracdiff}
  A random field $u = \{u(t,x) \colon (t,x) \in \R_+ \times \R^d \}$ is a
  solution to equation~\eqref{E:fracdiff}, if it satisfies the following mild
  formulation:
  \begin{align*}
    u(t,x) = \calJ_0 (t,x) + \int_0^t \int_{\R^d} \: Y (t-s, x-y) \rho \left(u(s,y)\right) W(\ud s,\ud y),
  \end{align*}
  where
  \begin{align}\label{E:J0fracdiff}
    \calJ_0 (t,x) =\begin{dcases}
     \int_{\R^d} \mu_0(\ud y) Z (t,x-y),                                          & b\in(0,1], \\
     \int_{\R^d} \mu_0(\ud y) Z^* (t,x-y) + \int_{\R^d} \mu_1(\ud y)\, Z (t,x-y), & b\in(1,2),
   \end{dcases}
   \end{align}
    $\{(Z(t,x),Z^*(t,x),Y(t,x)) \colon (t,x) \in \R_+ \times \R^d \}$ denote the
    fundamental functions of equation~\eqref{E:fracdiff}, which can be
    formulated explicitly in terms of some special functions; see~\cite[Section
    4]{chen.hu.ea:19:nonlinear}.
\end{definition}

\begin{hypothesis}\label{H:subsg}
  \begin{enumerate}[(i)]
    \item Dalang's condition holds for~\eqref{E:fracdiff}, namely, $b +
      \gamma > \frac{1}{2} (1 + \frac{d b}{a})$ and $2a > d$.
    \item Either $\gamma = 0$ or $a > d = 1$.
  \end{enumerate}
\end{hypothesis}

\begin{definition}\label{D:SubSemigroup}
  Let $G: \R_+\times \R^d \to \R$ be a Borel function. We say that
  $G^2(\cdot,\circ)$ satisfies a \textit{sub-semigroup property} if there exist
  a positive constant $C^*$, a reference kernel $\calG: \R_+\times \R^d \to
  \R_+$, and a locally integrable function $\ell: \R_+\mapsto \R_+$, such that
  for all $(s,t,x) \in \R_+^2 \times \R^d$,
  \begin{align*}
    \int_{\R^d}\calG(t,y) \ud y = 1, \quad
    G(t,x)^2 \leq \ell(t)\, \calG(t,x) \quad \text{and} \quad
    \int_{\R^d}\calG(t,x-y) \calG(s,y) \ud y \leq C^*\, \calG(t+s, x).
  \end{align*}
\end{definition}

\begin{lemma}[Proposition~5.10 of~\cite{chen.hu.ea:19:nonlinear}]\label{L:subsg}
  Let $Y(\cdot,\circ)$ be as in Definition~\ref{D:fracdiff}. Under
  Hypotheses~\ref{H:subsg}, for some constant $C_0>0$, $Y^2(\cdot,\circ)$
  satisfies the sub-semigroup property (see Definition~\ref{D:SubSemigroup})
  with
  \begin{gather}\label{E:subsg}
    \ell(t) = C_0 t^{-\sigma}
    \quad \text{and} \quad
    \calG (t,x) = \begin{dcases}
      k_{ b} t^{-\frac{b d}{2}} \exp \Big(
      -\frac{1}{2}(t^{-\frac{b}{2}} |x|)^{\lfloor b \rfloor +1 } \Big),
      & a =2, \\
      \frac{ k_{a, b} t^{ b/a}}{
      \left( t^{\frac{2b}{a}  } + |x|^2 \right)^{(d+1)/2}},
      & a \in (0, 2),
    \end{dcases}
  \end{gather}
  where $k_{b}$ (resp. $k_{a,b}$) is a uniform constant, depending on $b$ (resp.
  $(a,b)$) and
  \begin{align}\label{E:sigma}
    \sigma \coloneqq 2(1 - b - \gamma) + (b d)/a < 1.
  \end{align}
  (Note that $\sigma<1$ in~\eqref{E:sigma} is due to part~(i) of
  Hypothesis~\ref{H:subsg}.)
\end{lemma}

Let $\calJ_0(t,x)$ be defined as in~\eqref{E:J0fracdiff} and let $\ell(t)$ and
$\calG(t,x)$ be given as in~\eqref{E:subsg}. We define the next expression
\begin{align}\label{E:J1fracdiff}
  \calJ_1(t,x) \coloneqq \int_0^t\ud s\, \ell(t - s) \int_{\R^d} \ud y\, \calG(t-s,x-y)
      \calJ_0^2(s,y),\quad \text{for all } (t,x) \in \R_+\times \R^d.
\end{align}
Then, thanks to~\cite[Lemma 5.12]{chen.hu.ea:19:nonlinear}, it holds that
$\calJ_1(t,x) <\infty$ for all $t>0$ and $x\in\R^d$. As a result of
Lemma~\ref{L:subsg}, we can follow the same idea as in Section~\ref{S:main} to
obtain the following result:

\begin{theorem}\label{T:mupfracdiff}
  Let $u(t,x)$ be a solution to the fractional SPDE~\eqref{E:fracdiff} driven by
  the space-time white noise. Assume that
  \begin{enumerate}
    \item the initial conditions given in~\eqref{E:InitFrac} satisfy
      $\mathcal{J}_0(t,x)<\infty$ for all $t>0$ and $x\in\R^d$;
    \item the diffusion coefficient $\rho$ satisfies Hypothesis~\ref{H:rho};
    \item the parameters $(d, a, b, \gamma)$ satisfy Hypothesis~\ref{H:subsg}.
  \end{enumerate}
  Then for all $(t,x,p) \in \R_+\times \R^d \times[2,\infty)$, it holds that
  \begin{align}\label{E:mupfracdiff}
    \begin{aligned}
      \Norm{u(t,x)}_p^2 \leq\:
      & 2 \calJ_0^2(t,x) + K \bigg( t^{-(1-\sigma)}\calJ_1(t, x) + K p\, t^{1-\sigma} + F^{-1}\left(K p\, t^{1-\sigma}\right) \bigg),
    \end{aligned}
  \end{align}
  where $F^{-1}$, $\sigma$, $\calJ_0$, and $\calJ_1$ are defined as
  in~\eqref{E:F^-1},~\eqref{E:sigma},~\eqref{E:J0fracdiff}
  and~\eqref{E:J1fracdiff}, respectively, and $K$ is a positive constant not
  depending on $p$ and $t$. In particular,
  \begin{enumerate}[(i)]
    \item if $|\rho(\cdot)|$ is concave separately on $\R_+$ and $\R_-$, then
      one can take $K_2=0$ in~\eqref{E:mupfracdiff} to obtain
      \begin{align*} 
        \Norm{u(t,x)}_p^2 \leq 2 \calJ_0^2(t,x) + C \bigg( t^{-(1-\sigma)}\calJ_1(t, x) + F^{-1} \left(C p\, t^{1-\sigma} \right)\bigg);
      \end{align*}
    \item if the initial conditions $\mu_0$ and $\mu_1$ are such that
      \begin{align*}
        \sup_{(t,x)\in \R_+\times\R^d} \calJ_0(t,x) <\infty,
      \end{align*}
      then there exist constant $C_*>0$, such that for all $(t,x,p)\in
      [1,\infty) \times\R^d\times[2,\infty)$,
      \begin{align}\label{E:mupfracdiff_smp}
        \Norm{u(t,x)}_p^2 \leq C_* F^{-1} \left(C_* p\, t^{1-\sigma} \right) ;
      \end{align}

    \item if $\mu_0 \equiv 1$ and $\mu_1 \equiv 0$, then the solution $u$ is
      a.s. $\eta_1$--H\"older continuous in time and $\eta_2$--H\"older
      continuous in space on $(0,\infty)\times\R^d$ for all
      \begin{align*}
        \eta_1 \in \left(0, \frac{1 - \sigma}{2}\right) \quad \text{and} \quad
        \eta_2 \in \left(0, 1 \wedge \frac{\theta}{2}\right),
      \end{align*}
      where $\theta \coloneqq 2a + \frac{a}{b} \min\{2\gamma - 1,0\}$;
    \item if $\mu_0 \equiv 1$ and $\mu_1 \equiv 0$, then with the same $C_*$ as
      in part (ii), and $L_t$ as in~\eqref{E:Lt} with $h(t) = t^{1 - \sigma}$,
      for all $(t,x) \in [1,\infty) \times \R^d$ and $z \geq L_t$,
        \begin{gather*}
          \P\left(|u(t,x)| \geq z\right) \leq \exp \left(- C_*^{-1} t^{- (1 - \sigma)} F \left(z^2\big/ \big(C_* e^2\big)\right)\right), \shortintertext{and with some universal constant $C' > 0$,}
          \sup_{|x|\le R} u(t,x) \lesssim \sqrt{F^{-1} \Big(C'\: t^{1 - \sigma} \log R\Big)}\:, \quad \text{a.s., as $R \to \infty$.}
        \end{gather*}
  \end{enumerate}
\end{theorem}

\begin{proof}
  The proof of parts (i) and (ii) is similar to the proof of
  Theorem~\ref{T:mup}. We only need to point out some differences.
  In this case, the function $h(t)$ should be computed as follows (thanks to
  Lemma~\ref{L:subsg}):
  \begin{align*}
    h (t) \coloneqq \int_0^t \ud s \int_{\R^d} \ud y \, Y^2(s,y)
    \le \int_0^t \ud s\: \ell(s) = C t^{1-\sigma}.
  \end{align*}
  Now, we apply the same argument as in the proof of Theorem~\ref{T:mup} to
  obtain the following inequality analogous to~\eqref{E:utx}:
  \begin{align*}
    \Norm{u(t,x)}_p^2
      & \leq 2 \calJ_0^2(t,x) + 8p K_M^2 \int_0^t \ud s \int_{\R^d} \ud y\, Y (t-s, x-y)^2 \\
      & \quad + 8 p \int_0^t \ud s \int_{\R^d} \ud y\, Y (t-s, x-y)^2 \rho_2 \left(M^2 + \Norm{u(s,y)}_p^2\right).
  \end{align*}
  where $K_M$ is defined in~\eqref{E:KM}. Using Lemma~\ref{L:subsg}, we can
  further deduce that
  \begin{align}\label{E:utxfracdiff}
    \begin{aligned}
      \Norm{ u(t,x)}_p^2
      & \leq 2 \calJ_0^2(t,x) + \frac{8 C_0 p K_M^2 t^{1-\sigma}}{1-\sigma} \\
      & + 8 C_0 p \int_0^t\ud s\, (t-s)^{-\sigma} \int_{\R^d} \ud y\, \calG (t-s, x-y) \rho_2 \left(M^2 + \Norm{u(s,y)}_p^2 \right).
    \end{aligned}
  \end{align}
  Performing the convolution on both sides of~\eqref{E:utxfracdiff} with $
  s^{-\sigma}\calG(s,z)$ on $[0,t]\times \R^d$, and using the sub-semigroup
  property of $\calG$, we deduce that
  \begin{align*}
    \MoveEqLeft \int_0^t \ud s\, (t-s)^{-\sigma}\int_{\R^d} \ud z\, \calG(t-s, x-z) \Norm{u(s,z)}_p^2 \\
    \leq & 2 \int_0^t \ud s\, (t-s)^{-\sigma}
             \int_{\R^d} \ud z\, \calG(t-s,x-z)
             \calJ_0 (s,z)^2 + \frac{8 C_0 K_M^2 B(1-\sigma,2-\sigma) p\, t^{2-2\sigma}}{1-\sigma} \\
         & + 8 C_0 C_1 p\, t^{1-\sigma}
             \int_0^t \ud s\: (t-s)^{-\sigma}
             \int_{\R^d} \ud z\, \calG (t-s,x-z) \rho_2\left(M^2 + \Norm{u(v,z)}_p^2\right),
  \end{align*}
  where $B(\cdot,\cdot)$ is the beta function. As a result, we can write the
  next inequality analogous to~\eqref{E:nottoy},
  \begin{gather*}
    X \leq M^2 + 2 t^{-(1-\sigma)} \mathcal{J}_1 (t,x) + \frac{8 C_0 K_M^2 B(1-\sigma,2-\sigma)}{1-\sigma} p\, t^{1-\sigma} + 8 C_0 C_1 p\, t^{1-\sigma} \rho_2 (X), \shortintertext{with}
    X \coloneqq t^{-(1-\sigma)} \int_0^t \ud s\, (t-s)^{-\sigma}\int_{\R^d} \ud z\, \calG(t-s, x-z) \Norm{u(s,z)}_p^2.
  \end{gather*}
  The rest proof for parts (i) and (ii) of the theorem employs the same arguments as in Theorem~\ref{T:mup}.

  Suppose that $\mu_0 \equiv 1$ and $\mu_1 \equiv 0$, then for any $x,y \in
  \R^d$, and $p \geq 2$,
  \begin{align}\label{E:spcincfrac}
    \Norm{u(t,x) - u(t,y)}_p^2 \leq &
     8 p \int_0^t \ud s \int_{\R^d} \ud z\, |Y (t-s, x-z) - Y(t-s, y-z)|^2 \rho_2 \left(M^2 + \Norm{u(s,z)}_p^2\right)\nonumber\\
      \leq & C p \int_0^t \ud s \int_{\R^d} \ud z\, |Y (t-s, x-z) - Y(t-s, y-z)|^2 \left(1 + F^{-1}(C_* p s^{1-\sigma}) \right)\nonumber\\
      \leq & C_{p,t} \left(1 + F^{-1}(C_* p t^{1-\sigma}) \right) |x-y|^{\theta},
  \end{align}
  where the last two inequalities are consequence of~\eqref{E:mupfracdiff},
  parts (i) and (ii) of Hypothesis~\ref{H:rho}, and~\cite[Proposition
  5.4]{chen.hu.ea:19:nonlinear}, and $C_{t,p} > 0$ depending on $p$ and $t$.
  This proves the H\"{o}lder continuity of $u$ in space. The H\"{o}lder
  continuity in time can be shown similarly, which is skipped for conciseness.
  This completes the proof of part (iii).

  Finally, with~\eqref{E:mupfracdiff_smp} and~\eqref{E:spcincfrac} and
  observation that $1 + F^{-1}(C_* p t^{1-\sigma}) \le C F^{-1}(C_* p
  t^{1-\sigma})$ for all $t \geq 1$ with some universal constant $C > 0$, the
  proof for part (iv) follows the same lines presented as in the proofs of
  Theorems~\ref{T:Tail} and~\ref{T:Asymspc}.
  The proof of Theorem~\ref{T:mupfracdiff} is complete.
\end{proof}

\begin{remark}\label{R:mupfracdiff}
  Here are two comments: (1) From the proof of Theorem~\ref{T:mupfracdiff}, it
  is evident that the key to the conclusion lies in the fact that the
  fundamental solution $Y$ satisfies Assumption~\ref{D:SubSemigroup}, which is
  ensured by Lemma~\ref{L:subsg}. There may be other equations whose mild
  formulation can be written as in Definition~\ref{D:fracdiff}, with $Y$
  fulfilling Assumption~\ref{D:SubSemigroup}, and for which a moment bound can
  be deduced using the same approach. (2) For spatially colored noise, it is not
  clear to us how to formulate a similar factorization formula for the reference
  kernel. Thus, directly applying our proof becomes non-trivial. We conjecture
  that the moment upper bound can also be established and leave it as an open
  problem for interested readers to explore.
\end{remark}

\subsection{One-dimensional stochastic wave equations}\label{SS:SWEd}

In this part, we consider the one-dimensional \textit{stochastic wave equation}
(SWE) driven by the spatial homogeneous noise that is white in time:
\begin{align}\label{E:SWE}
  \begin{dcases}
    \frac{\partial^2}{\partial t^2} u(t,x) = \Delta u(t,x) + \rho(u(t,x)) \dot{W} (t,x), & (t,x)\in \R_+\times \R, \\
    u(0,\cdot) = \mu_0, \quad  \left.\frac{\partial}{\partial t} u(t,x)\right|_{t = 0} = \mu_1;
  \end{dcases}
\end{align}
see, e.g.,~\cite{dalang.mueller:09:intermittency, chen.dalang:15:moment,
balan.conus:16:intermittency, hu.wang:21:intermittency}, for related results
about solutions to SWEs in the linear regime. The solution to~\eqref{E:SWE} is
formulated in the mild form:
\begin{gather}
  u(t,x) = \mathcal{J}_0 (t,x) + \int_0^t\int_{\R} G(t-s,x-y) \rho(u(s,y)) W(\ud s,\ud y), \notag  \shortintertext{with}
  \mathcal{J}_0 (t,x) \coloneqq \frac{1}{2} \left[ \mu_0(x+t) + \mu_0 (x-t) \right] + \int_{\R} \mu_1(\ud y) G(t,x-y),
  \label{E:J0-SWE}
\end{gather}
where, throughout the rest part of this section, $G(t,x) = \frac{1}{2}
\one_{[-t,t]}(x)$ refers to the wave kernel on $\R$. Our arguments for the main
results Theorem~\ref{T:mup} can be extended to this case too:

\begin{theorem}\label{T:mupwav}
  Let $u(t,x)$ be a solution to the SWE~\eqref{E:SWE} with initial position
  $\mu_0\in L_{loc}^2(\R)$ and initial velocity $\mu_1$ which is a locally
  finite Borel measure on $\R$. Under Hypothesis~\ref{H:corre} (with $d=1$) and
  Hypothesis~\ref{H:rho}, it holds that
  \begin{align}\label{E:mupwav}
    \begin{aligned}
      \Norm{u(t,x)}_p^2 \leq \:
      & 2 \calJ_0^2(t,x) + K_1 \left( h(t)^{-1} \calJ_1(t,x) + K_2 p\: h(t)  + F^{-1}( 2 p\: h(t)) \right),
    \end{aligned}
  \end{align}
   for all $(t,x,p)\in\R_+\times \R \times[2,\infty)$, where $\calJ_0$ and
   $F^{-1}$ are given in~\eqref{E:J0-SWE} and~\eqref{E:F^-1}, respectively,
  \begin{gather}\label{E:J1-SWE}
    \calJ_1(t,x) \coloneqq \int_0^t \ud s \iint_{\R^2} \ud y \ud y'\: G\left(t-s,x-y \right) G\left(t-s,x-y'\right) f(y - y') \calJ_0^2(s,y),
    \shortintertext{and}
    h (t) \coloneqq \int_0^t \ud s \iint_{\R^2} \ud y \ud y'\, G(s,y) G(s,y') f(y-y').
    \label{E:SWE-h}
  \end{gather}
  In~\eqref{E:mupwav}, $K_1$ and $K_2$ are some positive constants not depending
  on $t$ and $p$. In particular,
  \begin{enumerate}[(i)]
    \item if $|\rho(\cdot)|$ is concave separately on $\R_+$ and $\R_-$, then
      one can take $K_2=0$ in~\eqref{E:mupwav};
    \item there exist some constants $C > 0$ such that the moment
      bound~\eqref{E:mup_asy} holds true, where $\mathcal{J}_0$, $\calJ_1$, and
      $h(t)$ are given in~\eqref{E:J0-SWE},~\eqref{E:J1-SWE},
      and~\eqref{E:SWE-h}, respectively. Moreover, if $\mu_0$ is a bounded
      function and $|\mu_1|(\R)<\infty$, then the moment bound
      in~\eqref{E:mup_asy} can be simplified to~\eqref{E:mup_smp} with $h(t)$
      given in~\eqref{E:SWE-h};
   \item assuming that $\mu_0 \equiv 1$ and $\mu_1 \equiv 0$, the solution
     $u(t,x)$ to~\eqref{E:SWE} has a version which is a.s. $\eta_1$--H\"older
     continuous in time and $\eta_2$--H\"older continuous in space on
     $(0,\infty) \times \R$ for all $\eta_1,\eta_2\in (0,\eta)$, where $\eta$ is
     given in~\eqref{E:ED};
  \item assuming that $\mu_0 \equiv 1$ and $|\mu_1|\equiv 0$, with the same
    $C_*$ as in part (ii) and $L_t$ given in ~\eqref{E:Lt} but with $h(t)$
    replaced by~\eqref{E:SWE-h}, then for all $(t,x) \in [T,\infty) \times \R^d$
    and $z \geq L_t$,
    \begin{gather*}
      \P\left(|u(t,x)| \geq z\right) \leq \exp \left(- (C_* h(t))^{-1} F \left(z^2\big/ \big(C_* e^2\big)\right)\right), \shortintertext{and under Hypothesis~\ref{H:efcdalang}, with some universal constant $C' > 0$,}
      \sup_{|x|\le R} u(t,x) \lesssim \sqrt{F^{-1} \Big(C'\: h(t) \log R\Big)}\:, \quad \text{a.s., as $R \to \infty$.}
    \end{gather*}
  \end{enumerate}
\end{theorem}

We first prove two lemmas.

\begin{lemma}\label{L:WaveKernel}
  The wave kernel function $G(\cdot,\circ)$ satisfies the following properties:
  \begin{enumerate}[(i)]
    \item For all $t, s \in \R_+$ and $x, y\in\R$,
      \begin{align}\label{E:WaveFactor}
        G(t,x-y) G(s,y) = 2 G(t,x-y) G(s,y) G(t+s,x)\le  G(s,y) G(t+s,x);
      \end{align}
    \item For all $t>0$ and $x\in\R$,
      \begin{align}\label{E:wavcvl1}
        t G(t,x) \leq G(2t,x) (2t-|x|) = 2
        \left(G(t,\cdot) * G(t,\cdot)\right)(x)
        \leq 2t G(2t,x);
      \end{align}
    \item If the correlation function $f:\R\to\R_+$ satisfies
      Hypothesis~\ref{H:corre}, then for all $t>0$,
      \begin{align}\label{E:wavcvl}
        \sup_{(s,x)\in [0,t]\times \R} \int_{\R} \ud z \: G\left(s,x + z\right) f(z)
        \leq 2 k(2t) < \infty,
      \end{align}
      where $k(t)$ is defined as in~\eqref{E:k} with $p_t(x)$ replaced by
      $G(t,x)$;
    \item $G(\cdot, \circ)^2$ satisfies the sub-semigroup property (see
      Definition~\ref{D:SubSemigroup}) with
      \begin{align*}
        \mathcal{G} (t,x) = \frac{1}{2 t} \one_{[-t,t]} (x) = \frac{1}{t} G(t,x), \quad
        \ell(t) = \frac{t}{2}, \quad \text{and} \quad
        C^*= 2.
      \end{align*}
  \end{enumerate}
\end{lemma}
\begin{proof}
  Part (i) is due to the following identity: for any $s,t\in \R_+$ and $x,y\in
  \R$,
  \begin{align*}
    \one_{[-t,t]} (x-y) \one_{[-s,s]} (y)
    = \one_{[-t,t]} (x-y) \one_{[-s,s]} (y)
    \one_{[-(s+t),t+s]} (x). 
  \end{align*}
  Part (ii) can be obtained by direct computation. As for part~(iii), by the
  nonnegativity of $f$ and non-decreasing property of $t\to G(t, x)$ with
  $x\in\R$ fixed, we see that
  \begin{align*}
    \sup_{(s,x)\in [0,t]\times \R} \int_{\R} \ud z \: G\left(s,x + z\right) f(z) \leq \sup_{x \in\R} \int_{\R} \ud z \: G\left(t,x + z\right) f(z)
  \end{align*}
  Next, thanks to the first inequality in~\eqref{E:wavcvl1}, for any $(t,x)\in
  \R_+\times \R$, we have that
  \begin{align*}
    \int_{\R} \ud z \: G \left(t,x + z\right) f(z)
    \leq & t^{-1} \int_{\R} \ud z \: G(2t,x + z) (2t-|x +z|) f(z)                           \\
    =    & \frac{2}{t}\iint_{\R^2}\ud y \ud z\: G(t, x + y) G (t, y - z) f(z),              \\
    =    & \frac{1}{\pi t} \int_{\R} e^{ix}\left|\frac{\sin (t\xi)}{\xi}\right|^2 \hat{f}(\ud\xi)
    \leq   \frac{1}{\pi t} \int_{\R} \left|\frac{\sin (t\xi)}{\xi}\right|^2 \hat{f}(\ud\xi) \\
    =    & \frac{2}{t} \iint_{\R^2}\ud y \ud y'\: G(t, y) G (t, y') f(y - y').
  \end{align*}
  Then apply the second inequality in~\eqref{E:wavcvl1} to the above upper bound
  to see that
  \begin{align*}
    \int_{\R} \ud z \: G \left(t,x + z\right) f(z)
    \le & 2 \int_{\R} \ud z \: G(2t, z) f(z) = 2 k(2t).
  \end{align*}
  Notice that thanks to Dalang's condition~\eqref{E:Dalang},
  \begin{align*}
    \frac{1}{\pi t} \int_{\R} \left|\frac{\sin (t\xi)}{\xi}\right|^2 \hat{f}(\ud\xi)
    \le C_t \int_\R \frac{\widehat{f}(\ud\xi)}{1+|\xi|^2}<\infty, \quad \text{for all $t>0$.}
  \end{align*}
  This proves part~(iii).

  Finally, as for part~(iv), it is straightforward to verify that
  $\int_{\R}\mathcal{G} (t,x)\ud x = 1$ and $G(t,x)^2 = \ell(t) \mathcal{G}
  (t,x)$. Moreover,
  \begin{align*}
    \MoveEqLeft
    \int_{\R} \mathcal{G} (t,x-y) \mathcal{G} (s,y) \ud y
    = \frac{\one_{[-(t+s), t+s]} (x)}{4st} \int_{\R} \one_{[-t,t]} (x - y) \one_{[-s,s]} (y) \ud y      \\
    =    & \frac{\one_{[-(t+s), t+s]} (x)}{4st} \times \mathrm{Leb} \left([-t+x, t+x]\cap [-s,s]\right) \\
    \leq & \frac{2 (s\wedge t) \one_{[-(t+s), t+s]} (x)}{4st}
    = \frac{\one_{[-(t+s), t+s]} (x)}{2 (s\vee t)}
    \leq \frac{\one_{[-(t+s), t+s]} (x)}{s + t}
    = 2 \mathcal{G}(t+s, x).
  \end{align*}
  Thus, the sub-semigroup relation in~\eqref{E:subsg} holds with $C^*=2$. This
  proves part~(iii).
\end{proof}

Thanks to part~(iv) of Lemma~\ref{L:WaveKernel}, we can obtain the moment bounds
using the same arguments as those in Theorem~\ref{T:mupfracdiff}. It is
important to note, however, that as discussed in Remark~\ref{R:mupfracdiff}, the
method in Theorem~\ref{T:mupfracdiff} is limited to the space-time white noise.
To get the desired result for the spatial colored noises, we need the following
result, which serves as the analog of Lemma~\ref{L:J1d=1}.

\begin{lemma}\label{L:J1d=1wav}
  Under the same setting as Theorem~\ref{T:mupwav}, $\mathcal{J}_1(t,x)$ defined
  in~\eqref{E:J1-SWE} is finite for all $(t,x)\in \R_+\times\R$. In particular,
  for every $(t,x) \in \R_+\times \R$,
  \begin{align*}
    \mathcal{J}_1 (t,x)
    \le 4 t k(2t) \left(t\int_{\R} G(t,x - y) \mu_0^2 (y)\ud y + 2 \left(t \int_{\R} |\mu_1|(\ud y) G(t,x - y)\right)^2\right)
    < \infty,
  \end{align*}
  where $k(t)$ is defined in~\eqref{E:k} with $p_t(x)$ replaced by $G(t,x)$.
\end{lemma}
\begin{proof}
  We decompose $\mathcal{J}_0(t,x)$ in~\eqref{E:J0-SWE} into two parts:
  \begin{align*}
    \calJ_{0,1} (t,x) \coloneqq \frac{1}{2} [\mu_0(x+t) + \mu_0(x-t) ] \quad \mathrm{and} \quad
    \calJ_{0,2} (t,x) \coloneqq \int_{\R} \mu_1 (\ud y) G(t,x-y).
  \end{align*}
  Then, by change of variable,
  \begin{align*}
  \calJ_{1,1}(t,x)
  \coloneqq & \int_0^t \ud s \iint_{\R^2} \ud y \ud y'\: G\left(t-s,x-y\right) G\left(t-s,x-y'\right) f(y - y') \calJ_{0,1}(s,y)^2 \nonumber \\
    =       & \int_0^t \ud s \iint_{\R^2} \ud y \ud z\:  G\left(t-s,x-y\right) G\left(t-s,x - y + z\right) f(z) \calJ_{0,1}(s,y)^2.
  \end{align*}

  \noindent Using inequality~\eqref{E:wavcvl}, we have
  \begin{align*}
    \mathcal{J}_{1,1}(t,x)
    \le 2 k(2t) \int_0^t \ud s \int_{\R} \ud y\: G\left(t-s,x - y\right) \calJ_{0,1}(s,y)^2
    \eqqcolon 2 k(2t) \, \Theta_1(t).
  \end{align*}
  Notice that
  \begin{align*}
    \Theta_1(t)
    & \le \int_0^t \ud s \int_{\R}\ud  y\, G\left(t-s,x - y \right) \left(\mu_0^2(y+s)+ \mu_0^2(y-s)\right)\\
    & =   \int_{\R}\ud z \,\mu_0^2(z) \int_0^t \ud s\, \left[G\left(t-s,x - z + s \right) + G\left(t-s,x - z -s \right)\right].
  \end{align*}
  Because $s - |x-z|\le |x-z\pm s|\le t-s$, we see that the above $\ud
  s$-integral is bounded 
  $t\one_{\{|x-z|\le t\}} = 2t G(t,x-z)$. Hence, the condition $\mu_0\in
  L_{\text{loc}}^2\left(\R\right)$ implies that
  \begin{align*}
    \mathcal{J}_{1,1}(t,x)
    \le 2 t k(2t) \int_{x-t}^{x+t}\mu_0^2(z)\ud z
    =   4 t\, k(2t) \int_{\R} \ud z\, G (t,x-z) \mu_0^2(z) <\infty.
  \end{align*}

  It remains to show that
  \begin{align*}
    \calJ_{1,2}(t,x) \coloneqq \int_0^t \ud s \int_{\R^2} \ud y \ud y'\:
    G\left(t-s,x-y\right) G\left(t-s,x-y'\right) f(y - y') \calJ_{0,2}(s,y)^2 <
    \infty.
  \end{align*}
  Thanks to~\eqref{E:wavcvl}, we see that
  \begin{align*}
    \mathcal{J}_{1,2}(t,x)
    \le 2k(2t) \int_0^t \ud s\int_{\R} \ud y\, G\left(t-s,x - y\right) \calJ_{0,2}(s,y)^2
    \eqqcolon 2k(2t)\, \Theta_2(t).
  \end{align*}
  By the Minkowski inequality with respect to the $\ud y$-integral, we can write
  \begin{align*}
    \Theta_2(t) = 2 \int_{\R} \left[G(t-s, x-y) \calJ_{0,2} (s,y)\right]^2 \ud y
             \leq 2 \left[ \int_{\R} \left(\int_{\R} G(t-s, x-y) G(s,y-z)^2 \ud y \right)^{1/2} |\mu_1| (\ud z) \right]^2.
  \end{align*}
  Due to~\eqref{E:WaveFactor}, we see that $\int_{\R} G(t-s, x-y) G(s,y-z)^2 \ud
  y \le \frac{s}{4} \one_{[-t,t]} (x-z)$. Therefore, $\Theta_2(t) \le
  \frac{s}{2} \left(\int_{x-t}^{x+t} |\mu_1|(\ud z)\right)^2$ and
  \begin{gather*}
    \mathcal{J}_{1,2}(t,x)
    \le 2t^2\, k(2t)\left(\int_{x-t}^{x+t} |\mu_1|(\ud z)\right)^2
    = 8t^2\, k(2t)\left(\int_{\R} |\mu_1|(\ud z)  G(t, x - z)\right)^2 < \infty.
  \end{gather*}
  This completes the proof of Lemma~\ref{L:J1d=1wav}.
\end{proof}

Now we are ready to prove Theorem~\ref{T:mupwav}.

\begin{proof}[Proof of Theorem~\ref{T:mupwav}]
   It suffices to provide the proof for part (i), which follows a similar
   approach to that of Theorem~\ref{T:mup}, with one notable difference arising
   from the absence of the corresponding formula~\eqref{E:Factor} for the heat
   equation case. However, we can leverage~\eqref{E:WaveFactor} to overcome this
   obstacle and conclude
  \begin{align*}
    \MoveEqLeft \int_r^t \ud s \iint_{\R^2} \ud y \ud y'\, G(t-s, x-y) G(t-s, x'-y') G(s-r, y) G(s-r, y') f(y-y') \\
    \leq & G(t-r, x) G(t-r, x') \int_0^{t-r} \ud s \iint_{\R^2} \ud y \ud y'\,  G(s, y) G(s, y') f(y-y') \\
    =    & G(t-r, x) G(t-r, x') h(t-r)
    \le    G(t-r, x) G(t-r, x') h(t).
  \end{align*}
  The last sequence of inequalities play the same role as those
  in~\eqref{E:2hsck}. Following the same idea as in the proof of
  Theorem~\ref{T:mup}, we can show that
  \begin{align*}
    X \leq M^2 + 2h(t)^{-1} \mathcal{J}_1 (t,x) + 8K_M^2  p\, h(t) + 8p h(t) \rho_2
    (X),
  \end{align*}
  where $K_M$ is given in~\eqref{E:KM} and
  \begin{align*}
    X \coloneqq M^2 + h(t)^{-1} \int_0^t\ud s \iint_{\R^2} \ud y \ud y' G(t-s,x-y) G(t-s,x-y')
    f(y-y') \Norm{u(s,y)}_p^2.
  \end{align*}
  The remaining proof follows the same line of reasoning as presented in
  Theorem~\ref{T:mup},~\ref{T:Tail} and~\ref{T:Asymspc}, if one can show the
  H\"{o}lder continuity for solution(s) to~\eqref{E:SWE} under the improved
  Dalang condition. In the following, we only outline the proof for spatial
  continuity, and the time continuity can be verified similarly. Under the
  assumption that $\mu_0 \equiv 1$ and $\mu_1 \equiv 0$, we can write
  \begin{align*}
    \Norm{u(t,x) - u(t,y)}_p^2 \leq & C \int_0^t \ud s \iint_{\R^2} \ud z \ud z' f(z - z') |G(t-s, x - z) - G(t-s, y -z)|\\
    & \times |G(t-s, x - z') - G(t-s, y -z')| \Norm{\rho(u(s,z))}_p \Norm{\rho (u(s,z'))}_p\\
    \leq & C \int_0^t \ud s\iint_{\R^2} \ud z \ud z' f(z - z') |G(t-s, x - z) - G(t-s, y -z)|\\
    & \times |G(t-s, x - z') - G(t-s, y -z')|  \left(1 + F^{-1}(C_* p h(s))\right)^2.
  \end{align*}
  Then, by using the same arguments as in~\cite[Theorem
  5]{dalang.sanz-sole:05:regularity}, one finds that $u$ is H\"{o}lder
  continuous in space with and exponent $\alpha < \eta$ and constant of the form
  $C F^{-1}(C_* p h(t))$. The proof of Theorem~\ref{T:mupwav} is complete.
\end{proof}

\begin{remark}
  Parts (iii) and (iv) of Theorem~\ref{T:mupwav} requires that the initial
  conditions $\mu_0 \equiv 1$ and $\mu_1 \equiv 0$, which is not necessary. They
  are used for simplification of the proof of the H\"{o}lder continuity of the
  solution. It can be relaxed to e.g. $\mu_0\in H_2^{\eta}(\R)$ and $\mu_1 \in
  C(\R)$ with $|\mu_1|(\R) < \infty$; see~\cite{dalang.sanz-sole:05:regularity}.
\end{remark}

\section{Applications}\label{S:Application}

In this session, we will integrate the specific $\rho$ studied in
Section~\ref{S:Rho} and the noise structure, as detailed in
Appendix~\ref{S:app}, to derive more precise moment bounds, as well as, tail
probabilities and spatial asymptotics. In particular, we will demonstrate the
transitions of properties from the additive SHE to the PAM in
Section~\ref{SS:App-SHE}.

\subsection{Moment growth, tail probability and spatial asymptotics for SHE}\label{SS:App-SHE}

In this part, we showcase examples of the diffusion coefficient $\rho$ and apply
Theorems~\ref{T:mup},~\ref{T:Tail}, and~\ref{T:Asymspc} to derive moment bounds,
tail probabilities, and spatial asymptotics for SHE~\eqref{E:SHE}. We will use
$C$ to denote a generic constant that may change its value at each appearance,
but does not depend on either $p$ or $t$.

\paragraph{I.~} We first study the case when $\rho$ is given in
Proposition~\ref{P:Rho-Alpha}. We start with a simple example:

\begin{proposition}
  Under the setting of Theorem~\ref{T:mup} but with the following diffusion
  coefficient
  \begin{align*}
    \rho(u) = \frac{|u|}{ 1 + |u|},
  \end{align*}
  there exists a unique solution $u$ to SHE~\eqref{E:SHE} such that for all
  $t>0$, $x\in\R^d$ and $p\ge 2$,
  \begin{align}\label{E:SHEbdd}
    \Norm{u(t,x)}_p^2 \le 2\calJ_0^2(t,x) + 8p h(t).
  \end{align}
\end{proposition}

Note that the uniqueness comes from the fact that $\rho$ satisfies Lipschitz
condition and the moment bound in~\eqref{E:SHEbdd} coincides with the moment
bounds for the SHE with additive noise, the proof of which is straightforward by
noticing that
\begin{align*}
  \Norm{u(t,x)}_p^2
  \leq & 2 \calJ_0^2(t,x) + 8p \int_0^t \ud s \int_{\R^{2d}} \ud y \ud y' p_{t-s}(x-y)p_{t-s}(x-y') f(y-y').
\end{align*}
However, in case of the bounded initial condition and $t \geq 1$, one can also
obtain the moment bound~\eqref{E:SHEbdd} in the form of~\eqref{E:mup_smp} by
using $F^{-1}$ in~\eqref{E:F-Inv-Alpha} with $\alpha=0$: $F^{-1}(x) \le 8 x$.
\medskip

The next proposition examines the $\rho$ given in Proposition~\ref{P:Rho-Alpha}
in detail. We particularly focus on the impacts on the initial conditions on
the growth of moments.

\begin{proposition}\label{P:Alpha}
  Let $u$ be a solution to SHE~\eqref{E:SHE} under the same setting as
  Theorem~\ref{T:mup}, but with the diffusion coefficient $\rho$ given in
  Proposition~\ref{P:Rho-Alpha}. Then, for all
  $(p,t,x)\in[2,\infty)\times\R_+\times\R^d$, it holds that
  \begin{align*}
    \Norm{u(t,x)}_p^2 \leq C \left[
        \calJ_0^2(t,x)
        + p h(t)
        + \left(p h(t)\right)^{1/(1-\alpha)}
        + h(t)^{-1} \calJ_1(t,x)
      \right].
  \end{align*}
  Furthermore, thanks to Proposition~\ref{P:SHE-h}, we have the following
  special cases:
  \begin{enumerate}[(i)]
  \item If $d = 1$ and $f = \delta$, then
   \begin{align}\label{E:SHEalpha}
      \Norm{u(t,x)}_p^2
        \leq C \left(
          \calJ_0^2(t,x)
          + p\,\sqrt{t}
          + \left(p\, \sqrt{t}\right)^{1/(1-\alpha)}
          + \calJ_+^2 ( t/2, x )
        \right).
    \end{align}
    Moreover, if $\mu$ is a bounded function, then,
    \begin{align*}
     \Norm{u(t,x)}_p^2
        \leq C  \left(p\, \sqrt{t}\right)^{1/(1-\alpha)}, \quad t \geq 1.
    \end{align*}

    \item If $d = 1$ and $f(x) = |x|^{-\beta}$ with $\beta\in (0, 1)$, then
      \begin{align*}
        \Norm{u(t,x)}_p^2 \leq C \bigg(
          \calJ_0^2(t,x)
          + p\, t^{1 - \beta/2}
          + \Big(p\, t^{1 - \beta/2} \Big)^{1/(1-\alpha)}
          + \calJ_+^2 (t/2,x)
        \bigg).
      \end{align*}
      Again, if $\mu$ is bounded, then,
      \begin{align*}
        \Norm{u(t,x)}_p^2 \leq C  \left(p\, t^{1 - \beta/2}\right)^{1/(1-\alpha)}, \quad t \geq 1.
      \end{align*}

    \item If $f = |x|^{- \beta}$ with $\beta \in (0, 2 \wedge d)$, then for all $d\ge 1$
      it holds that
      \begin{align}\label{E:alpha-Riesz-corr}
        \Norm{u(t,x)}_p^2
        \leq C \left(
          \calJ_0^2(t,x)
          + p\, t^{1 - \beta/2}
          + \left(p\, t^{1 - \beta/2} \right)^{1/(1-\alpha)}
          + t^{-1+\beta/2}\, \calJ_1 ( t, x )
        \right).
      \end{align}
      Moreover,
      \begin{enumerate}[(a)]
        \item if $\mu(x) = |x|^{-\ell}$ with $\ell \in (0,2 \wedge d)$, then we
          have
          \begin{align}\label{E:alpha-Riesz}
            \Norm{u(t,x)}_p^2
            \leq C \left(
              t^{-\ell}
              + \left(p\, t^{1 - \beta/2}\right)^{1/(1-\alpha)} +   p\, t^{1 - \beta/2}
            \right);
          \end{align}
        \item and if $\mu(x) = e^{\ell |x|}$ with $\ell \in\R$, then with some
          universal constants $C_1$ and $C_2 >0$ it holds that
          \begin{align}\label{E:alpha-ExpInit}
            \begin{aligned}
            \Norm{u(t,x)}_p^2
            \leq C & \left(
              e^{ C_1 \ell^2 t + C_2 \ell |x|}
              + p\, t^{1 - \beta/2}
              + \left(p\, t^{1 - \beta/2}\right)^{1/(1-\alpha)}\right).
            \end{aligned}
          \end{align}
        \end{enumerate}
    \end{enumerate}
\end{proposition}

The proof of this proposition is essentially an application of
Proposition~\ref{P:Rho-Alpha} together with
Theorems~\ref{T:mup},~\ref{T:Tail},~\ref{T:Asymspc}, and
Proposition~\ref{P:SHE-h}. Some more details regarding the role of initial
conditions are given in Appendix~\ref{S:app_exam}.

\begin{remark}
  (1) It is clear that in Case (iii)--(b) of Proposition~\ref{P:Alpha}, if $\ell
  > 0$, as $t\to\infty$, $\Norm{u(t,x)}_p^2$ is bounded by an exponent function of
  $t$, which is dominated by the explosion rate of $\mu$ at infinity, and the
  stochastic fluctuation makes fewer contributions in the moment bounds. (2)
  When $r>0$, $\rho$ is globally Lipschitz continuous (see part~1 of
  Proposition~\ref{P:Rho-Alpha}) and the solution is unique.
\end{remark}

An application of Proposition~\ref{P:Rho-Alpha} together with
Theorems~\ref{T:mup},~\ref{T:Tail}, and~\ref{T:Asymspc} yields the following:

\begin{proposition}\label{P:AlphaCst}
  Under the same setting as Proposition~\ref{P:Alpha}, but with bounded the
  initial condition. Then for all $p\ge 2$, $t \geq 1$ and $x\in\R^d$,
  the following statements hold:
  \begin{align}\label{E:alphaCst-mnt}
    \Norm{u(t,x)}_p^2                    & \lesssim \left(p h(t)\right)^{1 /(1-\alpha)},         &  & \text{as $p \vee t\to\infty$,}                          \\
    \log \P \left(|u(t,x)| \geq z\right) & \lesssim -\frac{z^{2(1-\alpha)}}{h(t)},               &  & \text{as $z\to \infty$,} \label{E:alphaCst-Tail} \\
    \sup_{|x|\leq R} u(t,x)              & \lesssim \left[h(t) \log R \right]^{1/(2(1-\alpha))}, &  & \text{a.s. as $R\to\infty$.} \label{E:alphaCst-x}
  \end{align}
\end{proposition}

\begin{remark}
  In comparison to the tail estimates for super Brownian motions found
  in~\cite[Proposition 1.4]{hu.wang.ea:23:moment}, the tail estimate given
  by~\eqref{E:alphaCst-Tail} is sharp for the case when $\alpha = 1/2$. Since
  when $\alpha = 0$, the bound in~\eqref{E:alphaCst-x} coincides with the exact
  asymptotics demonstrated in~\cite[Theorem 1.2]{conus.joseph.ea:13:on}
  and~\cite[Theorem 2.3]{conus.joseph.ea:13:on*1}, we believe that the bound
  in~\eqref{E:alphaCst-x} is sharp for all $\alpha\in [0,1)$.
\end{remark}

\paragraph{II.~} Now we study the case when $\rho$ is given in
Proposition~\ref{P:Rho-Log}. An application of Proposition~\ref{P:Rho-Log}
together with Theorems~\ref{T:mup},~\ref{T:Tail},~\ref{T:Asymspc}, and
Proposition~\ref{P:SHE-h} ($h(t)=\sqrt{t/\pi}$) yields part (a) of the following
proposition:

\begin{proposition}\label{P:Log}
  Let $u$ be a solution to SHE~\eqref{E:SHE} under the setting that $d =1$, $\mu
  \equiv 1$, $f = \delta$, and the diffusion coefficient $\rho$ given in
  Proposition~\ref{P:Rho-Log}. Then, for all $p\ge 2$, $t \geq 1$, and $x\in\R$,
  the following statements hold:
  \begin{enumerate}[(a)]
    \item In cases (i) and (ii) in~\eqref{E:Ex-LogRho}, namely,
      $(\alpha,\beta)\in(0,1)\times\R$ or $(\alpha,\beta)\in
      \{0\}\times(-\infty,0)$,
      \begin{align}
         \Norm{u(t,x)}_p^2                   & \lesssim \left(p\sqrt{t}\right)^{1/(1-\alpha)} \left[\log\left( p \sqrt{t}\:\right)\right]^{-2\beta/(1-\alpha)},\label{E_:log_1}         &  & \text{as $p \vee t\to\infty$,}          \\
         \log \P\left(|u(t,x)| \geq z\right) & \lesssim - z^{2(1-\alpha)}t^{-1/2}\left[\log z\right]^{2\beta},                                                                          &  & \text{as $z\to\infty$},\label{E_:log_2} \\
         \sup_{|x|\leq R} u(t,x)             & \lesssim \left[\sqrt{t}\:\log R\right]^{\frac{1}{2(1-\alpha)}} \left[\log\left(\sqrt{t}\:\log R\right)\right]^{-\frac{\beta}{1-\alpha}}, &  & \text{a.s. as $R\to \infty$.}\label{E_:log_3}
      \end{align}
    \item In case (iii) in~\eqref{E:Ex-LogRho}, namely, $\alpha = 1$ and
      $\beta>0$, by setting $\beta^*\coloneqq \beta\wedge (1/4)$,
      \begin{align}
        \Norm{u(t,x)}_p^2                 & \lesssim \exp\left(C\left(p^2 t\right)^{1/(4\beta^*)}\right), \label{E_:exp_1} &  & \text{as $p \vee t\to\infty$,}          \\
        \log \P\left(u(t,x) \geq z\right) & \lesssim - t^{-1/2}\left[\log z\right]^{1 + 2\beta^*},                         &  & \text{as $z\to\infty$},\label{E_:exp_2} \\
        \sup_{|x|\leq R} u(t,x)           & \lesssim \exp\left(C \left(\sqrt{t} \log R\right)^{1/(1 + 2\beta^*)}\right),   &  & \text{a.s. as $R\to \infty$.}\label{E_:exp_3}
      \end{align}
  \end{enumerate}
\end{proposition}

\begin{proof}[Proof of part (b) of Proposition~\ref{P:Log}]
  For~\eqref{E_:exp_1}, if we apply $F^{-1}(\cdot)$ given in~\eqref{E:F-Inv-Log}
  to the moment formula~\eqref{E:mup_smp}, we obtain
  \begin{align*}
    \Norm{u(t,x)}_p^2 \leq C\exp\left( C \left(p^2 t\right)^{1/(4\beta)} \right).
  \end{align*}
  But by invoking the moment comparison principle
  (see~\cite{joseph.khoshnevisan.ea:17:strong, chen.kim:20:stochastic}), we find
  that the exponent of $t$ cannot exceed that of the parabolic Anderson model.
  Therefore, we use $\beta^*$ in the exponent of $t$ in~\eqref{E_:exp_1}.

  The tail estimate in~\eqref{E_:exp_2} is better than that obtained by setting
  $\alpha=1$ in~\eqref{E_:log_2}, the latter of which produces of the power
  $2\beta^*$ (instead of $2\beta^*+1$), namely,
  \begin{align*}
    \log \P(|u(t,x)| \geq z) \lesssim - t^{-1/2}\left(\log z\right)^{2\beta^*},\quad \text{for all } z>0 \text{ large enough}.
  \end{align*}
  Indeed, the tail probability in~\eqref{E_:exp_2} can be obtained by directly
  computing $H(p)$ in~\eqref{E:H}
  (with $CF^{-1}(Cph(t))$ replaced by the
  sharper moment upper bound in~\eqref{E_:exp_1})
   and the corresponding
  Legendre-type transform $H^*(y)$ in~\eqref{E:H*}:
  \begin{align*}
    H(p)   = C p^{\frac{1 + 2\beta^*}{2\beta^*}} t^{1/(4\beta^*)} \quad \Rightarrow \quad
    H^*(y) = \frac{C}{\sqrt{t}} y^{1+2\beta^*},\quad \text{for large $y>0$,}
  \end{align*}
  and then plugging $H^*(y)$ back to~\eqref{E_:tail1}. Finally,
  applying~\eqref{E_:exp_2} in the proof of Theorem~\ref{T:Asymspc}
  proves~\eqref{E_:exp_3}.
\end{proof}

\paragraph{III.~} Here we study the case when $\rho$ is given in
Proposition~\ref{P:Rho-VSV}. An application of Proposition~\ref{P:Rho-VSV}
together with Theorems~\ref{T:mup},~\ref{T:Tail},~\ref{T:Asymspc}, and
Proposition~\ref{P:SHE-h} ($h(t)=\sqrt{t/\pi}$) yields the following:

\begin{proposition}\label{P:VSV}
  Suppose that $d =1$, $\mu \equiv 1$, $f = \delta$, and the diffusion
  coefficient $\rho$ given in Proposition~\ref{P:Rho-VSV} with $\kappa>1$. Then
  there exists a unique solution $u$ to SHE~\eqref{E:SHE} such that for all
  $p\ge 2$, $x\in\R$, and $t>1$, the following statements hold:
  \begin{align*}
    \Norm{u(t,x)}_p^2                  & \lesssim \exp\left(\exp\left\{ \left[(2\beta)^{-1}\log\left(C p\: \sqrt{t}\right)\right]^{1/\kappa}\right\}\right),              &  & \text{as $p \vee t\to\infty$,} \\
    \log \P\left(u(t,x) \geq z\right)  & \lesssim - \frac{1}{\sqrt{t}} \exp\Big\{2\beta \left[ \log (2) + \log \big(\log z\big)\right]^\kappa\Big\},                      &  & \text{as $z\to\infty$},
                                      \\
    \sup_{|x|\leq R} u(t,x)            & \lesssim \exp\left(\frac{1}{2}\exp\left\{ \left[(2\beta)^{-1}\log\left(C \sqrt{t}\log R\right)\right]^{1/\kappa}\right\}\right), &  & \text{a.s. as $R\to \infty$.}
  \end{align*}
\end{proposition}

We summarize the moment growth rates and the spatial asymptotics in
Figure~\ref{F:Summary} below.

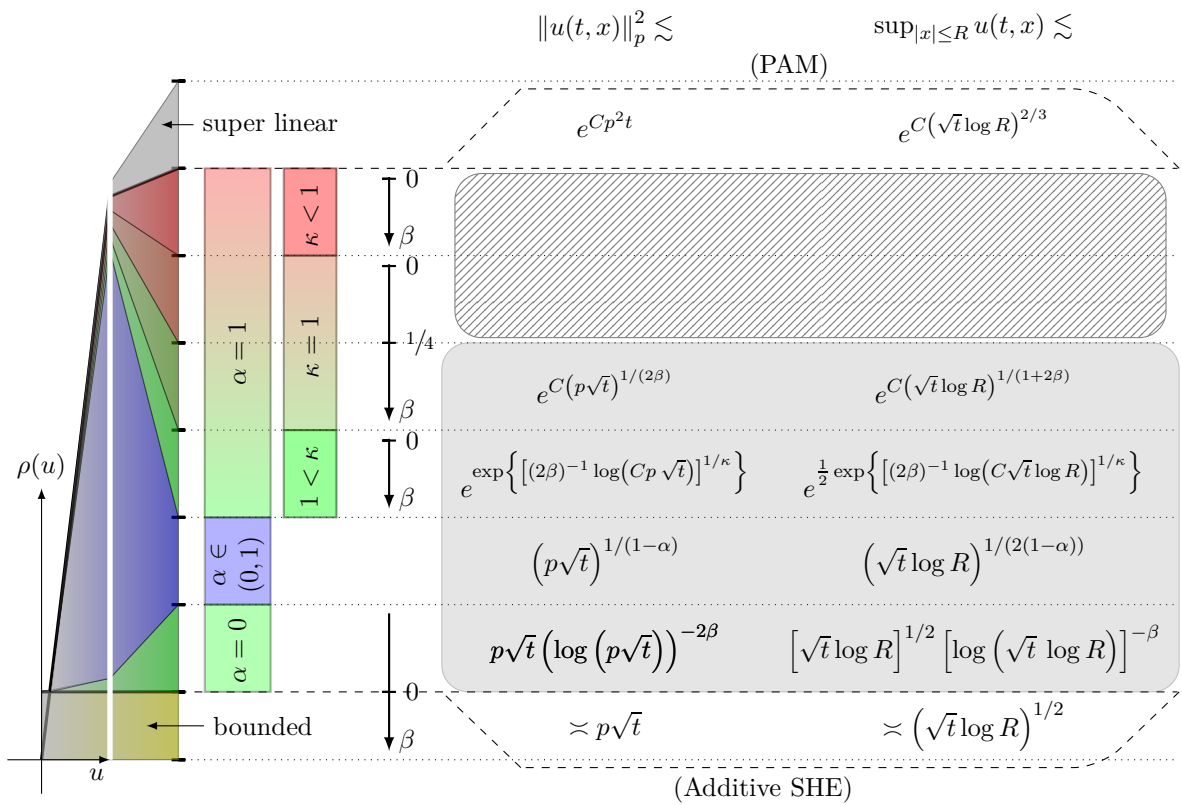
\begin{figure}[htpb!]
  \centering
  \begin{center}
    \begin{tikzpicture}[scale=0.9, transform shape]
      \tikzset{>=latex}
      \def\vd{0.2}
      \def\vx{1}
      \def\mag{1}
      \def\vy{9.0}
      \def\magbar{5.5}
      \newcommand{\MyColorH}{white}
      \newcommand{\MyColorG}{red}
      \newcommand{\MyColorF}{red!80!green}
      \newcommand{\MyColorE}{red!30!green}
      \newcommand{\MyColorD}{green}
      \newcommand{\MyColorC}{blue}
      \newcommand{\MyColorB}{green}
      \newcommand{\MyColorA}{yellow}

      \coordinate (o0) at (0,0);
      \coordinate (o1) at (0,1);           \coordinate (M0) at (\vx+\mag,0);
      \coordinate (R1) at (\vx,1.0);       \coordinate (M1) at (\vx+\mag,1);
      \coordinate (R2) at (\vx,1.0+1*\vd); \coordinate (M2) at (\vx+\mag,1+1*\vy/7);
      \coordinate (R3) at (\vx,8.5-5*\vd); \coordinate (M3) at (\vx+\mag,1+2*\vy/7);
      \coordinate (R4) at (\vx,8.5-4*\vd); \coordinate (M4) at (\vx+\mag,1+3*\vy/7);
      \coordinate (R5) at (\vx,8.5-3*\vd); \coordinate (M5) at (\vx+\mag,1+4*\vy/7);
      \coordinate (R6) at (\vx,8.5-2*\vd); \coordinate (M6) at (\vx+\mag,1+5*\vy/7);
      \coordinate (R7) at (\vx,8.5-1*\vd); \coordinate (M7) at (\vx+\mag,1+6*\vy/7);
      \coordinate (R8) at (\vx,8.5-0*\vd); \coordinate (M8) at (\vx+\mag,1+7*\vy/7);

      \begin{scope}[very thick]
        \draw [name path = Add] (0,1) -- (R1);
        \draw [name path = PAM] (0,0) -- (R7);
        \draw (R7) -- (M7);
        \draw (R1) -- (M1);
        \path[name intersections={of=Add and PAM}]
          (intersection-1) coordinate (o2);

        \draw (M0) --++ (0.1,0) --++(-0.2,0);
        \draw (M1) --++ (0.1,0) --++(-0.2,0);
        \draw (M2) --++ (0.1,0) --++(-0.2,0);
        \draw (M3) --++ (0.1,0) --++(-0.2,0);
        \draw (M4) --++ (0.1,0) --++(-0.2,0);
        \draw (M5) --++ (0.1,0) --++(-0.2,0);
        \draw (M6) --++ (0.1,0) --++(-0.2,0);
        \draw (M7) --++ (0.1,0) --++(-0.2,0);
        \draw (M8) --++ (0.1,0) --++(-0.2,0);

      \end{scope}

      \begin{scope}[dotted]
        \def\myShift{14.6}
        \draw         (M0) --++ (\myShift,0);
        \draw[dashed] (M1) --++ (\myShift,0);
        \draw         (M2) --++ (\myShift,0);
        \draw         (M3) --++ (\myShift,0);
        \draw         (M4) --++ (\myShift,0);
        \draw         (M5) --++ (\myShift,0);
        \draw         (M6) --++ (\myShift,0);
        \draw[dashed] (M7) --++ (\myShift,0);
        \draw         (M8) --++ (\myShift,0);
      \end{scope}

      \begin{scope}[thick]
        \def\myShift{8em}

        \begin{scope}[shorten >=0.3em, shorten <=0.3em]
          \draw[->] ([xshift = \myShift]M6) -- ([xshift = \myShift]M4) node [right, yshift = 0.7em] {$\beta$};
          \draw[->] ([xshift = \myShift]M2) -- ([xshift = \myShift]M0) node [right, yshift = 0.7em] {$\beta$};
          \draw[->] ([xshift = \myShift]M4) -- ([xshift = \myShift]M3) node [right, yshift = 0.7em] {$\beta$};
          \draw[->] ([xshift = \myShift]M7) -- ([xshift = \myShift]M6) node [right, yshift = 0.7em] {$\beta$};
        \end{scope}

        \draw ([xshift = \myShift, yshift = -0.4em]M7) --++(-0.1,0) --++(0.2,0) node [right] {$0$};
        \draw ([xshift = \myShift, yshift = -0.4em]M6) --++(-0.1,0) --++(0.2,0) node [right] {$0$};
        \draw ([xshift = \myShift, yshift = -0.4em]M4) --++(-0.1,0) --++(0.2,0) node [right] {$0$};
        \draw ([xshift = \myShift]M5) --++(-0.1,0) --++(0.2,0) node [right] {$\sfrac{1}{4}$};
        \draw ([xshift = \myShift]M1) --++(-0.1,0) --++(0.2,0) node [right] {$0$};
      \end{scope}

      \begin{scope}[thick, every node/.style={right, rotate=90, black, opacity = 1, align = center}]
        \def\myShift{1.0em}
        \def\myShiftPlus{3.5em}
        \draw[top color = \MyColorG, bottom color = \MyColorD,opacity =0.3] ([xshift = \myShift]M3) rectangle node [xshift = -2.0em] {$\alpha=1$} ([xshift = \myShiftPlus]M7);
        \draw[fill = \MyColorC, opacity =0.3]                               ([xshift = \myShift]M2) rectangle node [xshift = -2.0em, text width =3em, ] {$\alpha\in$\\$(0,1)$} ([xshift = \myShiftPlus]M3);
        \draw[fill = \MyColorB, opacity =0.3]                               ([xshift = \myShift]M1) rectangle node [xshift = -1.7em] {$\alpha=0$} ([xshift = \myShiftPlus]M2);
      \end{scope}

      \begin{scope}[thick, every node/.style={right, rotate=90, black, opacity = 1, xshift = -1.7em}]
        \def\myShift{4em}
        \def\myShiftPlus{6em}
        \draw[fill = \MyColorG,opacity =0.4]                                 ([xshift = \myShift]M6) rectangle node {$\kappa<1$} ([xshift = \myShiftPlus]M7);
        \draw[top color = \MyColorF, bottom color = \MyColorE, opacity =0.3] ([xshift = \myShift]M4) rectangle node {$\kappa=1$} ([xshift = \myShiftPlus]M6);
        \draw[fill = \MyColorD,opacity =0.4]                                 ([xshift = \myShift]M3) rectangle node {$1<\kappa$} ([xshift = \myShiftPlus]M4);
      \end{scope}

      \begin{scope}[left color = white, opacity = 0.4, draw = none]
        \filldraw [right color = \MyColorA] (o0) -- (o1) -- (R1) -- (M1) -- (M0) -- (o0);
        \filldraw [right color = \MyColorB] (o1) -- (R1) -- (M1) -- (M2) -- (R2) -- (o1);
        \filldraw [right color = \MyColorC] (o2) -- (R2) -- (M2) -- (M3) -- (R3) -- (o2);
        \filldraw [right color = \MyColorD] (o2) -- (R3) -- (M3) -- (M4) -- (R4) -- (o2);
        \filldraw [right color = \MyColorE] (o2) -- (R4) -- (M4) -- (M5) -- (R5) -- (o2);
        \filldraw [right color = \MyColorF] (o2) -- (R5) -- (M5) -- (M6) -- (R6) -- (o2);
        \filldraw [right color = \MyColorG] (o2) -- (R6) -- (M6) -- (M7) -- (R7) -- (o2);
        \filldraw [right color = \MyColorH] (o2) -- (R7) -- (M7) -- (M8) -- (R8) -- (o2);
      \end{scope}

      \filldraw[white, line width = 2pt] (\vx,0) -- (\vx, 10);
      \draw [->] (-0.5,0) -- (\vx,0) node [below, xshift =-0.5em] {$u$};
      \draw [->] (0,-0.5) -- (0,4) node [above] {$\rho(u)$};

      \begin{scope}[xshift =29em]
        \path (M0) -- (M1) node [midway, xshift = +3.2em] (A) {bounded};      \draw[->] (A) --++(-1.7,0);
        \path (M8) -- (M7) node [midway, xshift = +3.5em] (A) {super linear}; \draw[->] (A) --++(-1.6,0);
      \end{scope}

      \begin{scope}[every node/.style={midway, xshift = +3.2em}, rounded corners = 10pt]

        \draw[dashed] ([xshift = 13em, yshift = -0.3em]M0) -- ([xshift = 10em]M1) ([xshift = 38em]M1) -- ([xshift = 35em, yshift = -0.3em]M0) -- ([xshift = 13em, yshift = -0.3em]M0) node [midway, below,xshift = -5em]{(Additive SHE)};
        \draw[dashed] ([xshift = 13em, yshift = -0.3em]M8) -- ([xshift = 10em]M7) ([xshift = 38em]M7) -- ([xshift = 35em, yshift = -0.3em]M8) -- ([xshift = 13em, yshift = -0.3em]M8) node [midway, above,xshift = -4em]{(PAM)};
        \draw[fill = gray, opacity = 0.2] ([xshift = 10em]M1)  -- ([xshift = 10em]M5)  -- ([xshift = 38em]M5) -- ([xshift = 38em]M1) -- cycle;
        \draw[pattern=north east lines, pattern color = gray, draw = gray] ([xshift = 10.5em, yshift = +0.2em]M5) rectangle ([xshift = 37.5em,yshift = -0.2em]M7);

        \def\myShift{13em}
        \path ([xshift = \myShift]M0) -- ([xshift = \myShift]M1) node {$\asymp p\sqrt{t}$};
        \path ([xshift = \myShift]M1) -- ([xshift = \myShift]M2) node {$p\sqrt{t}\left(\log\left(p\sqrt{t}\right)\right)^{-2\beta}$};
        \path ([xshift = \myShift]M1) -- ([xshift = \myShift]M2) node {$p\sqrt{t}\left(\log\left(p\sqrt{t}\right)\right)^{-2\beta}$};
        \path ([xshift = \myShift]M2) -- ([xshift = \myShift]M3) node {$\left(p\sqrt{t}\right)^{1/(1-\alpha)}$};
        \path ([xshift = \myShift]M3) -- ([xshift = \myShift]M4) node {$e^{\exp\left\{ \left[(2\beta)^{-1}\log\left(C p\: \sqrt{t}\right)\right]^{1/\kappa}\right\} }$};
        \path ([xshift = \myShift]M4) -- ([xshift = \myShift]M5) node {$e^{C\left(p \sqrt{t}\right)^{1/(2\beta)}}$};
        \path ([xshift = \myShift]M7) -- ([xshift = \myShift]M8) node {$e^{C p^2 t}$};
        \path ([xshift = \myShift]M8) -- ([xshift = \myShift]M8) node [yshift = 2em] {$\Norm{u(t,x)}_p^2 \lesssim$};

        \def\myShift{27em}
        \path ([xshift = \myShift]M0) -- ([xshift = \myShift]M1) node {$\asymp \left(\sqrt{t}\log R\right)^{1/2}$};
        \path ([xshift = \myShift]M1) -- ([xshift = \myShift]M2) node {$\left[\sqrt{t}\log R\right]^{1/2} \left[\log\left(\sqrt{t}\:\log R\right)\right]^{-\beta}$};
        \path ([xshift = \myShift]M2) -- ([xshift = \myShift]M3) node {$\left(\sqrt{t}\log R\right)^{1/(2(1-\alpha))}$};
        \path ([xshift = \myShift]M3) -- ([xshift = \myShift]M4) node {$e^{\frac{1}{2}\exp\left\{ \left[(2\beta)^{-1}\log\left(C \sqrt{t}\log R\right)\right]^{1/\kappa}\right\}}$};
        \path ([xshift = \myShift]M4) -- ([xshift = \myShift]M5) node {$e^{ C \left(\sqrt{t} \log R\right)^{1/(1 + 2\beta)}}$};
        \path ([xshift = \myShift]M7) -- ([xshift = \myShift]M8) node {$e^{C\left(\sqrt{t}\log R\right)^{2/3}}$};
        \path ([xshift = \myShift]M8) -- ([xshift = \myShift]M8) node [yshift = 2em] {$\sup_{|x|\le R} u(t,x) \lesssim$};

      \end{scope}
    \end{tikzpicture}
  \end{center}

  \caption{Summary of results in Propositions~\ref{P:AlphaCst},~\ref{P:Log}
    and~\ref{P:VSV} in case of $d=1$, space-time white noise ($f=\delta$), and a
    constant initial condition ($\mu\equiv 1$). The moment asymptotics are
    considered under the condition that $x\in\R$ is arbitrary but fixed and
  $t\to\infty$, while the spatial asymptotics are considered under the condition
$R\to\infty$, with $t \geq 1$.}

  \label{F:Summary}
\end{figure}

\subsection{Moment bounds for other SPDEs}

In the previous Section~\ref{SS:App-SHE}, we studied the combinations of various
$\rho$ with different types of noise for SHE~\eqref{E:SHE}. One can readily
construct concrete examples for the fractional SPDE~\eqref{E:fracdiff} by
considering various $\rho$, as provided in
Propositions~\ref{P:Rho-Alpha},~\ref{P:Rho-Log}, and~\ref{P:Rho-VSV} (although
only for the space-time white noise case). Similarly, for the SWE~\eqref{E:SWE},
thanks to Proposition~\ref{P:SWE-h}, one can easily investigate examples with
various noise structures as presented in Proposition~\ref{P:SWE-h} and various
$\rho$, such as those given in Propositions~\ref{P:Rho-Alpha}, \ref{P:Rho-Log},
and~\ref{P:Rho-VSV}. This application is straightforward, so we won't carry out
the details here. Instead, we will consider one particular example below.

Let us consider the diffusion coefficient $\rho$ given in
Proposition~\ref{P:Rho-Alpha} (so that $F^{-1}(x) \le C x^{1/(1-\alpha)}$) and
restrict ourselves to the one-dimensional space-time case (i.e., $f = \delta_0$
and $d=1$).
\begin{enumerate}
  \item Fractional SPDE~\eqref{E:fracdiff} with $a=2$ and $\gamma=0$: In this
    case, $\sigma = 2-3b/2$. By Theorem~\ref{T:mupfracdiff}, we can deduce the
    moment growth of $u(t,x)$ the solution to equation~\eqref{E:fracdiff} is as
    follows:
    \begin{align}\label{E:fracdiffalpha}
      \Norm{u(t,x)}_p^2 \leq C \left(\calJ_0^2(t,x) +  t^{1 - 3b/2}\calJ_1(t, x) +  p\, t^{3\beta/2 - 1}  +  \left( p\, t^{3b/2 - 1}\right)^{1/(1 - \alpha)}\right).
    \end{align}
  \item SWE~\eqref{E:SWE}: Let $u$ be a solution to equation~\eqref{E:SWE}. It
    follows from Theorem~\ref{T:mupwav} that
    \begin{align}\label{E:swealpha}
      \Norm{u(t,x)}_p^2 \leq C\left(\calJ_0^2(t,x) +  t^{-2}\calJ_1(t, x) +  p\, t^{2}  +  \left( p\, t^{2}\right)^{1/(1 - \alpha)}\right).
    \end{align}
\end{enumerate}

Taking account of Lemma~\ref{L:J1d=1}, we find that
inequalities~\eqref{E:SHEalpha} and~\eqref{E:swealpha} are special cases of
inequality~\eqref{E:fracdiffalpha} when $b = 1$ and $b = 2$, respectively. This
aligns with the interpolation characteristic of the stochastic heat and wave
equation; see~\cite{chen.eisenberg:22:interpolating,chen.guo.ea:22:moments} for
more instances. Furthermore, with the initial condition $\mu_0 \equiv 1$ and
$\mu_1 \equiv 0$, we can write the tail estimate and spatial asymptotics (for
$t>1$ and $x\in\R$):

\begin{enumerate}
  \item for the solution to~\eqref{E:fracdiff}:
   \begin{gather*}
       \log \P\left(|u(t,x)| \geq z\right) \lesssim - t^{1 - 3b/2} z^{2(1-\alpha)}, \quad \text{as $z\to\infty$} \shortintertext{and}
      \sup_{|x|\le R} u(t,x) \lesssim \Big( t^{3b/2 - 1} \log R\Big)^{1/(2(1-\alpha))}\:, \quad \text{a.s., as $R \to \infty$;}
    \end{gather*}
  \item for the solution to~\eqref{E:SWE}:
   \begin{gather*}
      \log \P\left(|u(t,x)| \geq z\right) \lesssim - t^{-2} z^{2(1-\alpha)}, \quad \text{as $z\to\infty$} \shortintertext{and}
      \sup_{|x|\le R} u(t,x) \lesssim \Big( t^{2} \log R\Big)^{1/(2(1-\alpha))}\:, \quad \text{a.s., as $R \to \infty$.}
    \end{gather*}
\end{enumerate}

\appendix
\section{Asymptotic behaviors of \texorpdfstring{$h(t)$}{}}\label{S:app}

In this section, we present asymptotic behaviors of $h$ as in
Theorem~\ref{T:mup} at infinity with some well-known correlation functions in
Appendix~\ref{SS:apph}; see~\cite[Examples 1.2-1.5]{chen.kim:19:nonlinear}
and~\cite[Section 5.4]{chen.eisenberg:22:invariant}. Similar asymptotic
behaviors of $h$ for the SWE (defined as in Theorem~\ref{T:mupwav}) are
presented in Appendix~\ref{SS:appw}

\subsection{Stochastic heat equations}\label{SS:apph}

\begin{proposition}\label{P:SHE-h}
  Let $h(t)$ be defined in~\eqref{E:h}. The following statements hold:
  \begin{enumerate}[(i)]
    \item (Riesz Kernel) If $f(x) = |x|^{-\alpha}$ with $\alpha \in (0, 2\wedge
      d)$, then
      \begin{align*}
        h(t)= C \:t^{1-\alpha/2} \quad \text{with
        $C = \frac{\Gamma\left((d-\alpha)/2\right)}{2^{\alpha/2}(2-\alpha)\Gamma\left(1+d/2\right)}$}.
      \end{align*}
    \item (Ornstein-Uhlenbeck kernel) If $f(x) = \exp (-|x|^{\alpha})$ with
      $\alpha > 0$, then
      \begin{align} \label{E:SHE-OU}
        h(t) \asymp \begin{dcases}
          \sqrt{t}, & d=1, \\
          \log(t),  & d=2, \\
          1,        & d \geq 3,
        \end{dcases}
        \quad \text{as $t\to\infty$,} \quad \text{and} \quad
        h(t) \asymp t \quad \text{as $t\downarrow 0$,}
      \end{align}
      and in particular, when $\alpha = 2$,
      \begin{align}\label{E:SHE-OU_2}
        h(t)=
          \begin{dcases}
            \sqrt{2 t+1}-1                                            & \text{if $d=1$,} \\
            4^{-1}\log(1+2 t)                                         & \text{if $d=2$,} \\
            \left[(d-2) d \right]^{-1}\left(1-(1+2 t)^{1-d/2} \right) & \text{if $d\ge 3$.}
          \end{dcases}
      \end{align}
    \item (Brownian case: $W(t,x) = W(t)$) If $f(x)\equiv 1$, then $h(t) = t$ for
      all $t\ge 0$.
    \item (One-dimensional space-time white case) If $d=1$ and $f(x) = \delta
      (x)$, then $h(t) = \sqrt{t/\pi}$.
    \item (Bessel potential; see~\cite[Section 1.2.2]{grafakos:14:modern},
      aka.~Mat\'{e}rn correlation family; see~\cite{guttorp.gneiting:06:studies})
      Suppose that
      \begin{align}\label{E:Bessel}
        f_\nu(x) = \int_{\R^d} \frac{e^{-i x\cdot\xi}}{(1 + |\xi|^2)^{\nu/2}}\ud \xi ,\quad \nu > 0.
      \end{align}
      Then,
      \begin{align}\label{E:SHE-Bessel}
        h(t) \asymp
          \begin{dcases}
            \sqrt{t}, & d=1, \\
            \log(t),  & d=2, \\
            1,        & d \geq 3,
          \end{dcases}
          \qquad \text{as $t\to\infty$,} \quad \text{and} \quad
        h(t) \asymp t \quad \text{as $t\downarrow 0$.}
      \end{align}
      (Note that the large time asymptotics is the same as the
      Ornstein-Uhlenbeck kernel because $f_{\nu}(x)$ behaves like
      $e^{-\frac{|x|}{2}}$ for $|x|\geq 2$; see~\cite[Proposition
      1.2.5]{grafakos:14:modern}).
    \item (Bessel potential as the spectral measure) Suppose that $f_\nu(x) =
      \left(1 + |x|^2\right)^{-\nu/2}$ with $\nu > 0$. Then
      \begin{align}\label{E:SHE-Bessel_2}
        h(t) \asymp
          \begin{dcases}
            t^{1-\frac{\nu\wedge d}{2}}, & \nu \wedge d < 2, \\
            \log(t),                     & \nu \wedge d = 2, \\
            1,                           & \nu \wedge d > 2,
          \end{dcases}
          \quad \text{as $t\to\infty$},
          \quad \text{and} \quad
        h(t) \asymp t \quad \text{as $t\downarrow 0$.}
      \end{align}
  \end{enumerate}
\end{proposition}
\begin{proof}
  Parts (i), (iii), (iv) and~\eqref{E:SHE-OU_2} are from Examples 1.2--1.5 in
  Appendix of~\cite{chen.kim:19:nonlinear}. Moreover, the asymptotics at $0$
  in~\eqref{E:SHE-OU},~\eqref{E:SHE-Bessel} and~\eqref{E:SHE-Bessel_2} can be
  formulated by evaluating the following limit by L'H\^{o}pital's rule,
  \begin{align*}
  \lim_{t\downarrow 0}h(t)/t = 2 \lim_{t\downarrow 0} \langle p_{2t}, f\rangle = 2 f(0),
  \end{align*}
  for all $f$ that is locally continuously bounded and has growth of order
  $o(e^{\epsilon x^2})$ for all $\epsilon>0$ at infinity. \medskip

  \noindent\textit{Eq.~\eqref{E:SHE-OU} in part~(ii):~} Note that the
  correlation function is bounded by $1$. Thus,
  \begin{align*}
    k(t) = \int_{\R^d} \ud z p_t(z) f(z) \leq \int_{\R^d} \ud z p_t(z) = 1.
  \end{align*}
  On the other hand, using the polar coordinates, we can write
  \begin{align*}
    k(t) =   C t^{-\frac{d}{2}} \int_{\R} \ud x  |x|^{d-1} \exp \left( - \frac{x^2}{2t} - |x|^{\alpha} \right)
        \leq C t^{-\frac{d}{2}} \int_{\R} \ud x  |x|^{d-1} \exp \left( - |x|^{\alpha} \right)
        \leq C t^{-\frac{d}{2}}.
  \end{align*}
  Therefore, for $t\geq 1$
  \begin{align*}
    h(t) = \int_0^{2t} \ud s\, k(s) \leq \int_0^1 \ud s + C\int_1^{2t} \ud s\, s^{- d/2}
    \leq
    \begin{dcases}
      C\left(1+\sqrt{t}\right),    & d=1, \\
      C\left(1 + \log t \right),   & d=2, \\
      C\left(1+t^{1 - d/2}\right), & d\geq 3.
    \end{dcases}
  \end{align*}
  For the lower bound, using the polar coordinates again, we can deduce that for
  $t\geq 1$,
  \begin{align*}
    h(t) \geq C \int_1^{2t} \ud s\, s^{-\frac{d}{2}} \int_{\R} \ud x\, |x|^{d-1} \exp \left( - \frac{x^2}{2} - |x|^{\alpha} \right)
         \geq C \int_1^{2t} \ud s\, s^{- d/2}.
  \end{align*}
  This completes the proof of the asymptotics of $h(\cdot)$ at infinity.
  \bigskip

  \noindent\textit{Part (v):~} From the proof of~\cite[Proposition
  5.12]{chen.eisenberg:22:invariant}, we have that
  \begin{align*}
     k(2t) = h'(t) \asymp U \left(\frac{d}{2}, \frac{2+d-\nu}{2}, t\right),\qquad \mathrm{as}\ t\to\infty,
  \end{align*}
  where $U$ denotes the \textit{confluent hypergeometric function};
  see~\cite[Formula 13.4.4 on page 326]{olver.lozier.ea:10:nist}. Using
  asymptotic behavior of $U (a,b,t)$ as $t\to \infty$ (see Formula 13.2.6 on
  page 322, \textit{ibid.}), one obtains the large time asymptotics
  in~\eqref{E:SHE-Bessel}.  \bigskip

  \noindent\textit{Part (vi):~} From the proof of~\cite[Proposition
  5.13]{chen.eisenberg:22:invariant} and by the asymptotic analysis for $U$ as
  in the previous case, we deduce that
  \begin{align*}
     k(2t) \asymp\, t^{-\frac{d}{2}} U \left(\frac{d}{2}, \frac{2+d-\nu}{2}, \frac{1}{4t}\right)
           \asymp\, t^{-\frac{\nu \wedge d}{2}}, \qquad \mathrm{as}\ t\to\infty,
  \end{align*}
  from which one obtains the large time asymptotics of $h(t)$
  in~\eqref{E:SHE-Bessel_2}. This completes the whole proof of
  Proposition~\ref{P:SHE-h}.
\end{proof}

\subsection{One-dimensional stochastic wave equations}\label{SS:appw}

\begin{proposition}\label{P:SWE-h}
  Let $h(t)$ be defined as in~\eqref{E:SWE-h}. The following statements hold:
  \begin{enumerate}[(i)]
    \item (Riesz Kernel) If $f(x) = |x|^{-\alpha}$ with $\alpha \in (0, 1)$,
      then for all $t>0$,
      \begin{align*}
        h(t) = C t^{3-\alpha} \quad \text{with $C= \frac{2^{1-\alpha}}{(1-\alpha)(2-\alpha)(3-\alpha)}$.}
      \end{align*}
    \item (Ornstein-Uhlenbeck kernel) If $f(x) = \exp (-|x|^{\alpha})$ with
      $\alpha > 0$, then,
      \begin{align}\label{E:SWE-OU}
        h(t) \asymp\, t^2 \quad \text{as $t\to \infty$ and}  \quad
        h(t) \asymp\, t^3 \quad \text{as $t\downarrow 0$}.
      \end{align}
    \item (Brownian case: $W(t,x) = W(t)$) If $f(x)\equiv 1$, then $h(t) =
      t^3/3$ for all $t\ge 0$.
    \item (Space-time white case) If $f(x) = \delta (x)$, then $h(t) = t^2/2$
      for all $t\ge 0$.
    \item (Bessel potential) Suppose that $f_\nu(\cdot)$ is the Bessel potential
      given in~\eqref{E:Bessel} with $d=1$ and the parameter $\nu>0$. Then
      \begin{align} \label{E:SWE-Bessel}
        h(t) \asymp t^2 \quad \text{as $t\to\infty$,} \quad \text{and} \quad
        h(t) \asymp t^3 \quad \text{as $t\downarrow 0$.}
      \end{align}
    \item (Bessel potential as the spectral measure) Suppose that $f_\nu(x) =
      \left(1 + x^2\right)^{-\nu/2}$ with $\nu > 0$ and $x\in\R$. Then,
      \begin{align}\label{E:SWE-Bessel_2}
        h(t) \asymp \begin{dcases}
          t^{3-\nu},  & \nu < 1, \\
          t^2\log(t), & \nu = 1, \\
          t^2,        & \nu > 1,
        \end{dcases}
        \qquad \text{as $t\to\infty$,} \quad \text{and} \quad
        h(t) \asymp t^3 \quad \text{as $t\downarrow 0$.}
      \end{align}
  \end{enumerate}
\end{proposition}

\begin{proof}
  Parts~(i), (iii) and~(iv) are straightforward. Similar as in
  Proposition~\ref{P:SHE-h}, the asymptotics for $t\downarrow 0$
  in~\eqref{E:SWE-OU}--\eqref{E:SWE-Bessel_2} are obtained because
  \begin{align*}
    \frac{c_1 t^3}{3}
    = c_1 \int_0^{t} \ud s \int_{\R^2} \ud y \ud y'\, G(s,y) G(s,y')
    \leq  \int_0^{t} \ud s \int_{\R^2} \ud y \ud y'\, G(s,y) G(s,y') f(y-y')
    \leq  \frac{c_2 t^3}{3},
  \end{align*}
  for all $t\in[0,\epsilon)$ and $f(\cdot)$ such that $0< c_1 \leq f(x) \leq
  c_2$ for all $x\in [-2\epsilon, 2\epsilon]$ with some $\epsilon>0$.

  For part~(ii), we first note that
  \begin{align*}
    h(t) \leq C \int_0^t \int_{-s}^s \ud y  \int_{\R} \ud y' f(y-y') \leq C t^2.
  \end{align*}
  For the low bounds, assume that$t > 1$. Then
  \begin{align*}
    h(t) \geq & C \int_1^t \ud s \int_{[-s,s]^2} \ud y \ud y' f(y-y') \one_{[-1,1]} (y-y')
    \geq        C 2^{\frac{3}{2}} e^{-2^{\alpha}}\int_1^t \ud s\,   (s-1) \geq C (t - 1)^2.
  \end{align*}
  This complete the proof of the large time asymptotics in~\eqref{E:SWE-OU}.

  As for part~(v), by using the Fourier transform of the wave kernel, we see
  that
  \begin{align*}
    h(t) = & (2\pi)^{-1}\int_0^t\ud s\int_\R\ud\xi \frac{\sin^2(s\xi)}{\xi^2 (1+\xi^2)^{\nu/2}}                                                                                              \\
      \leq & C \int_0^t\ud s \left(\int_{0}^{\frac{\pi}{2s}}\ud\xi \frac{(s\xi)^2}{\xi^2 (1+\xi^2)^{\nu/2}} + \int_{\frac{\pi}{2s}}^{\infty} \ud\xi \frac{1}{\xi^2 (1+\xi^2)^{\nu/2}}\right) \\
      \leq & C \int_0^t\ud s \left(\int_{0}^{\frac{\pi}{2s}} s^2 \ud\xi + \int_{\frac{\pi}{2s}}^{\infty} \frac{1}{\xi^2} \ud\xi\right)
      \leq   C\int_0^t \ud s (s + s) \leq C t^2.
  \end{align*}
  On the other hand, it is clear that
  \begin{align*}
    h(t)
    \geq & C \int_0^t\ud s \int_{0}^{\frac{\pi}{2s}}\ud\xi \frac{\sin^2 (s\xi)}{\xi^2 (1+\xi^2)^{\nu/2}}
    \geq   C \int_1^t \ud s \int_{0}^{\frac{\pi}{2s}}\ud\xi \frac{ (s\xi)^2}{\xi^2 (1+\xi^2)^{\nu/2}} \\
    \geq & C \int_1^t \ud s \int_{0}^{\frac{\pi}{2s}}\ud\xi \frac{ s^2}{(1+(\pi/2)^2)^{\nu/2}}
    \geq   C (t-1)^2,
  \end{align*}
  for all $t \geq 1$. This proves the large time asymptotics
  in~\eqref{E:SWE-Bessel}.

  As for part~(vi), by changing of variables $y - y' = z$ and $y + y' = z'$, we
  have that $\ud y \ud y' = 2^{-1} \ud z \ud z'$. Thus, for $t > 2$,
  \begin{align*}
    h(t) \leq \frac{1}{8} \int_0^t \ud s \int_{-2s}^{2s}  \int_{-2s}^{2s} \frac{\ud z \ud z'}{(1 + z^2)^{\nu/2}}
         \leq C \left(1 + \int_0^t \ud s\, s \int_{1}^{2s} \ud z\, z^{-\nu}\right)
  \end{align*}
  and
  \begin{align*}
    h(t) \geq \frac{1}{8} \int_2^t \ud s \int_{-s/2}^{s/2}  \int_{-s/2}^{s/2}  \frac{\ud z \ud z'}{(1 + z^2)^{\nu/2}}
    \geq \frac{1}{8} \int_2^t \ud s\, s \int_{1}^{s/2} \frac{\ud z}{(1 + z^2)^{\nu/2}}
    \geq \frac{1}{2^{3+\nu/2}} \int_2^t \ud s\, s \int_{1}^{s/2} \ud z z^{-\nu}.
  \end{align*}
  By considering three cases $\nu\in (0,1)$, $\nu =1$, and $\nu>1$ in the above
  two inequalities, we obtain the large time asymptotics
  in~\eqref{E:SWE-Bessel_2}. This completes the proof of
  Proposition~\ref{P:SWE-h}.
\end{proof}

\section{Supplementary proof for Proposition~\ref{P:Alpha}}\label{S:app_exam}

\begin{proof}[Proof of Proposition~\ref{P:Alpha}]
  All formulas in the special cases (i)--(iii), except~\eqref{E:alpha-Riesz}
  and~\eqref{E:alpha-ExpInit}, can be deduced by plugging in the specific $h(t)$
  given in Appendix~\ref{SS:apph}. For~\eqref{E:alpha-Riesz}, by~\cite[Example
  5.8]{chen.eisenberg:22:invariant}, we find that
  \begin{align*}
    \calJ_0(t,x) \leq c_1 t^{-\ell/2} \quad \text{and} \quad
    \calJ_1(t,x) \leq c_2 t^{1-\ell-\beta/2}.
  \end{align*}
  Then, plugging these inequalities into~\eqref{E:alpha-Riesz-corr}, one
  gets~\eqref{E:alpha-Riesz} immediately.

  The proof of~\eqref{E:alpha-ExpInit} also relies on the estimates for
  $\calJ_0(t,x)$ and $\calJ_1(t,x)$. By using the equivalence of norms in
  Euclidean spaces, there exists $c_2 \geq c_1 >0$ such that
  \begin{align*}
    c_1\sum_{i=1}^d |x_i| \leq |x|
    = \left(\sum_{i=1}^d x_i^2\right)^{1/2}
    \leq c_2 \sum_{i=1}^d |x_i|.
  \end{align*}
  It follows that
  \begin{align*}
    e^{\ell |x|}\leq
    \exp \left(c \ell \sum_{i=1}^d |x_i|\right),\quad \text{for all $x = (x_1,\dots, x_d)\in \R^d$,}
  \end{align*}
  where $c = c_2$ if $\ell \geq 0$ and otherwise $c = c_1$. As a result, by
  using~\cite[Proposition A.11]{chen.dalang:15:moments*1}, we can deduce that
  \begin{align}\label{E_:J0riz}
    \calJ_0(t,x)
    \leq \prod_{i=1}^d\int_{\R} \ud y_i e^{c \ell |y_i|} \frac{1}{\sqrt{2\pi t}} e^{-\frac{(x_i - y_i)^2}{2t}}
    \leq 2 e^{c^2 d \ell^2 t + c \ell |x|}.
  \end{align}
  We still need to estimate $\calJ_1(t,x)$. Using~\eqref{E:alpha-Riesz},
  by~\cite[Example 5.8]{chen.eisenberg:22:invariant}, we have
  \begin{align}\label{E:p*Riz}
    \int_{\R^d} \ud y p_t(x-y) |y|^{-\beta} \leq c_3 t^{-\beta/2},
  \end{align}
  with some universal constant $c_3 >0$. Hence, plugging~\eqref{E_:J0riz}
  and~\eqref{E:p*Riz} into~\eqref{E:J1}, it follows that
  \begin{align*}
    \calJ_1(t,x) \leq  2 c_3
    \int_0^t \ud s (t-s)^{-\beta/2}
    \int_{\R^{d}} \ud y \: p_{t-s} (x-y) e^{2c^2 d \ell^2 s + 2c \ell |y|}.
  \end{align*}
  Finally, following the same argument as in the proof of~\eqref{E_:J0riz}, one
  can show that
  \begin{align*}
    \calJ_1(t,x)
    \leq & 2 c_3 \int_0^t \ud s (t-s)^{-\beta/2} \exp\left(2c^2 d \ell^2 s + 4 c^4 d \ell^2 (t-s) + 2c^2 \ell |x|\right) \\
    \leq & t^{1-\beta/2} \exp \left(C_1 \ell^2 t + C_2 \ell |x|\right).
  \end{align*}
  This completes the proof of~\eqref{E:alpha-ExpInit}.
\end{proof}

\printbibliography[title={References}]

\end{document}